\documentclass[a4paper, 11pt]{article} \usepackage{amsmath} \usepackage{amsfonts}
\usepackage{mathrsfs} \usepackage{amsthm} \usepackage{amssymb}
\usepackage{bbm} \usepackage{bm}
\usepackage[left=3cm,right=3cm,top=3cm,bottom=3cm]{geometry}
\usepackage{hyperref}
\usepackage{algorithm}
\usepackage{algorithmic}
\usepackage[T1]{fontenc}

\usepackage[normalem]{ulem} \usepackage{enumerate}

\usepackage[usenames]{color}

\newcommand{\bpf}[1][Proof]{{\noindent {\sc #1: }}}
\newcommand{\epf}{{{\hfill $\Box$ \smallskip}}}

\newcommand{\R}{\mathbb{R}}
 \newcommand{\N}{\mathbb{N}}

 \newcommand{\tbf}{\mathbf{t}}
 \newcommand{\sbf}{\mathbf{s}}
 \newcommand{\ibf}{\mathbf{i}}
 \newcommand{\jbf}{\mathbf{j}}
 \newcommand{\ubf}{\mathbf{u}}
 \newcommand{\vbf}{\mathbf{v}}
 
\newcommand{\Pp}{\mathsf{P}} 
\newcommand{\Ee}{\mathsf{E}} 
 \newcommand{\E}{\mathbf{E}}
 
\newcommand{\Bc}{\mathcal{B}} 
 \newcommand{\M}{\mathcal{M}}
\newcommand{\F}{\mathcal{F}} 
\newcommand{\G}{\mathcal{G}}

 \newcommand{\supp}{\mathop{\rm supp}}

\newcommand{\EE}{\mathbb{E}}
\newcommand{\eps}{\epsilon}

\newcommand{\cl}{\text{cl}} 
\newcommand{\PP}{\mathbb{P}}

\newcommand{\id}{\mathbbm{1}}

\newtheorem{theorem}{Theorem} \newtheorem{lemma}{Lemma}
\newtheorem{remark}{Remark} \newtheorem{corollary}{Corollary}
\newtheorem{proposition}{Proposition} \newtheorem{example}{Example}
\newtheorem{definition}{Definition}
\newtheorem{assumption}{Assumption}
\def\1{{\rm 1\mskip-4.4mu l}}

\begin{document}

\title{Ergodicity and regularity properties of ODEs with semi-Markov switching}

\author{Tobias Hurth$^{1}$, Edouard Strickler$^{2}$}

\footnotetext[1]{Universit\'e de Neuch\^atel, Institut de Math\'ematiques, CH-2000 Neuch\^atel, Switzerland}
\footnotetext[2]{Universit\'e de Lorraine, CNRS, Inria, IECL, F-54000 Nancy, France \\
  E-mail:  edouard.strickler@univ-lorraine.fr}
\maketitle

\begin{abstract}
This paper is devoted to the study of a stochastic process obtained by random switching between a finite collection of vector fields. Such processes have recently been the focus of much attention in the case where the switching times are exponentially distributed, i.e., Markovian switching. In this contribution, we admit any distribution on $\mathbb{R}_+$ as a law for the switching times. We show that whenever this law is not singular with respect to the Lebesgue measure, the stochastic process obtained from the random switching is Feller. More importantly, we give conditions on the switching and on the vector fields ensuring that the Lie bracket condition considered in the Markovian case in \cite{BH12} and \cite{BMZIHP} still imply ergodicity of the process.  
\end{abstract}

\tableofcontents

\section{Introduction}

Let $(F^i)_{i \in E}$ be a family of vector fields on $\mathbb{R}^d$, indexed by a finite set $E$. For a c\`adl\`ag signal $I : \mathbb{R}_+ \to E$, we consider the switched system
\begin{equation}
\label{eq:intro1}
\frac{d X_t}{d t} = F^{I_t}(X_t).
\end{equation} 
This kind of system has recently been the focus of much attention in the particular case where $(I_t)_{t \geq 0}$ is a Markov chain on the space $E$. This means in particular that the holding time of $I$ in each state $i \in E$ is exponentially distributed. In other words, if at some time $T$,  $I$ switches to some value $i$, then the next switch will occur at time $T + S$, where $S$ is a random variable exponentially distributed with some parameter $\lambda^i$, and the new value of $I$ will be chosen randomly, and will be equal to some $ j \neq i$ with a probability $q_{i,j}$. Under this assumption, the long-term behaviour of $X_t$ has been studied extensively in the past decade (see e.g. \cite{BH12}, \cite{BMZIHP}, \cite{li17}, \cite{BHS18} and also \cite{CKW21}, where the continuous component $X$ may jump at switching times), and used to model several phenomena from ecology, genetics, epidemiology... (see e.g.  \cite{BL16}, \cite{CDGMMY17}, \cite{H19},   \cite{li17} and references therein)

One particularly important result from \cite{BH12} and \cite{BMZIHP} is that a Doeblin condition for the process $(X_t, I_t)$ can be guaranteed by checking a Hörmander-type condition, related to the Lie algebra generated by the vector fields.  Combined with accessibility and compactness, this yields, by classcial arguments, exponential ergodicity of the process $(X,I)$.  
%This point $x^*$ furthermore has to be accessible, meaning that $X_t$ can come arbitrarily close to $x^*$ with some positive probability (see Section \ref{sec:uni_ac} for details). 
The fact that the exponential law admits a positive density with respect to the Lebesgue measure on $\mathbb{R}_+$ is used in the proofs in \cite{BH12} and \cite{BMZIHP}. In particular, the fact that the process $I$ can switch arbitrarily fast with positive probability is crucial.  
%Of course, the lack of memory of the exponential law is also used, since this makes the process $(X,I)$ a Markov process.

In the aforementioned papers, the choice of the exponential law to govern the times between switches is made essentially to ensure the Markov property of the process $(X,I)$. Nonetheless, from a modelling perspective, this choice is restrictive. In particular, the assumption that the holding time of $I$ in one state $i$ can be arbitrarily short is often not realistic - think for example of epidemiological models, where switching between environments represents a change in health policy (e.g. lockdown versus no restriction), which cannot switch several times per day.

The present paper deals with switched system as the one in~\eqref{eq:intro1}, where the signal $I$ is random, with any possible law for the holding time in each state $i$ of $E$. In other words, if at some time $T$,  $I$ switches to some value $i$, then the next switch will occur at time $T + S$, where $S$ is a random variable with some distribution $\mu^i$. As before, the new value of $I$ will be chosen randomly, and will be equal to some $ j \neq i$ with a probability $q_{i,j}$, so that the sequence of the states visited by $I$ forms a Markov chain. Then, $I$ is a semi - Markov process.  This kind of switching process has been studied for example in \cite{A12}, \cite{LLC18}, \cite{LW21}, \cite{O23}...  Motivated by some applications (see e.g. \cite{H19}), we consider in  the present paper  the more general situation where the switching times and the postjump locations of $I$ depend on the continuous component $X$.

%The present paper deals with switched system as the one in~\eqref{eq:intro1}, where the signal $I$ is random, with any possible law for the holding time in each state $i$ of $E$. In other words, if at some time $T$,  $I$ switches to some value $i$, then the next switch will occur at time $T + S$, where $S$ is a random variable with some distribution depending on $i$ and possibly on $X_T$. As before, the new value of $I$ will be chosen randomly with some probability depending on $i$ and on $X_{T+S}$.

%One can consider the extended process $(X,\tau,I)$, 
If $\tau$ records the time elapsed since the latest switch, the three-component process $(X,\tau,I)$ is  Markov (see Section \ref{sec:setting} for details). We examine under which conditions $(X,\tau,I)$ has the Feller property and, under additional hypotheses, extend key results from \cite{BH12} and \cite{BMZIHP}. In particular, we provide two different hypotheses, each of which, together with the bracket conditions appearing in \cite{BH12} and \cite{BMZIHP}, implies a Doeblin condition for $(X, \tau, I)$. Roughly, these two hypotheses are as follows:
\begin{enumerate}
\item For all $i$,  $\mu^i$ is regular in a neighbourhood of $0$;
\item For all $i$, $\mu^i$ is regular \emph{somewhere}, and $F^i$ is analytic.
\end{enumerate}
Here, by \emph{regular}, we mean that, in the Lebesgue decomposition of $\mu^i$, the density of the absolutely continuous part is locally bounded away from $0$. Note that the first condition
%is really reminiscent from the exponential law, and 
implies  that the process $I$ can switch arbitrarily fast. The second condition, however, does not ask for $0$ to be in the support of the $\mu^i$'s and might be satisfied even if the times between switches are almost surely greater than some positive constant $\tau_D > 0$ (often referred to as \textit{dwell-time}). These two conditions can be explained briefly as follows. The bracket conditions are local and ensure, roughly speaking, that the process can go in any direction in a neighbourhood of the point where the bracket condition holds, provided that all the vector fields can be used in a short time. This is mathematically translated into the fact that a particular composition of the flows associated to the vector fields is a submersion for certain short times. If arbitrarily short switching times are allowed, then such a submersion arises with positive probability, which creates a non-trivial absolutely continuous component for the associated transition probability. Under the second condition, there is still a positive probability for a submersion, but for larger times. We note that in \cite{LLC18}, the authors prove that the bracket condition implies a Doeblin condition for $(X, \tau, I)$ under the assumptions that (i) the switched component $I$ does not depend on $X$ and (ii) $\mu^i$ has a density with respect to the Lebesgue measure, which is bounded away from $0$. One of the major contribution of the present paper is that we get rid of these two assumptions.

The paper is organised as follows. In Section~\ref{sec:setting} we give a rigorous construction of the switched process and prove that it is a strong Markov process. In Section~\ref{sec:Feller}, we provide necessary and sufficient conditions on the laws of the switching times for the process to be Feller. Moreover, when the process is Feller, we prove that whenever there exists a compact forward invariant set for the flows, then the process admits at least one invariant probability measure. Section~\ref{sec:ergodicity}, the main section of the paper, is devoted to the ergodicity and the regularity of the invariant probability measure under the aforementioned bracket condition. Some applications of our results to models in mathematical ecology are given in Section~\ref{sec:applications}, while Sections~\ref{sec:proofFeller} to~\ref{sec:proofaccess} contain the proofs of the results stated in Sections ~\ref{sec:Feller} and~\ref{sec:ergodicity}.

\section{Setting and first properties}
\label{sec:setting}

Let $E = \{1, \ldots, N\}$ be a finite index set. To construct the two-component process $(X,I)$ and, later on, the three-component process $(X,\tau,I)$, we need four ingredients:
\begin{itemize}
\item For all $i \in E$, a vector field $F^i$ on $\mathbb{R}^d$, with associated flow $(\varphi_t^i)_{ t \geq 0}$ describing the evolution of the continuous part $X$ between the jumps of $I$; 
\item For all $i \in E$, a probability measure $\mu^i$ on $\mathbb{R}_+$, related to the law of the time spent by the discrete part $I$ in state $i$;
\item For all $i\in E$, a map $\lambda^i : \mathbb{R}^d \to \mathbb{R}_+$, encoding the dependence of $I$ on $X$;
\item For all $x \in \mathbb{R}^d$, a stochastic matrix $Q(x) = (q_{i,j}(x))_{i,j \in E}$, encoding the switching law of $I$ when $X$ is at position $x$.
\end{itemize}
From these ingredients, the process $(X,I)$ starting from a point $(x,i)$ can be constructed informally as follows. Let $H_1$ be a random variable with law $\mu^i$ and set 
\begin{equation}
\label{eq:S1intro}
S_1 = \inf\left\{ t \geq 0 \: : \int_0^t \lambda^i( \varphi^i_r(x)) dr > H_1 \right\}
\end{equation}
and $T_1 = S_1$. Then, for all $t \geq 0$, 
$$ 
\mathbb{P}(S_1 > t) = G^i \biggl(\int_0^t \lambda^i(\varphi_r^i(x)) dr \biggr), 
$$ 
where $G^i(t) := \mu^i(t, \infty)$ is $1$ minus the cumulative distribution function of $\mu^i$. For all $t \in [0,T_1)$, we let $X_t = \varphi_t^i(x)$ and $I_t = i$. At time $T_1$, the component $X$ is made continuous by letting $X_{T_1} = \varphi_{T_1}^i(x)$, while a jump occurs for the discrete component $I$: the value $I_{T_1}$ is chosen equal to some $j \neq i$ with probability $q_{i,j}(X_{T_1})$. Now, assuming that $I_{T_1} = j$, let $H_2$ be a random variable with law $\mu^j$, independent from anything else, and set $T_2 = T_1 + S_2$, where
\[
S_2= \inf\left\{ t \geq 0 \: : \int_0^t \lambda^j( \varphi^j_r(X_{T_1})) dr > H_2 \right\}.
\]
Then, for all $t \in [T_1, T_2)$, $X_t = \varphi^j_{t - T_1}(X_{T_1})$ and $I_t = j$. At time $t=T_2$, $I_t$ jumps from $j$ to $k$ with probability $q_{j,k}(X_{T_2})$, and so on.

In other words, $X_t$ is solution to 
\begin{equation}
\label{eq:Xintro}
\frac{d X_t}{dt} = F^{I_t} (X_t),
\end{equation}
with continuous path, and $I_t$ is a jump process on $E$ whose jump times $(T_n)_{n \geq 1}$ and postjump locations depend on $X$ as follows:
\begin{equation}
\label{eq:introI}
\begin{cases}
\mathbb{P}\left( T_{n+1} - T_n > t \, | \, I_{T_n} =i, \, X_{T_n}=x \right) = \mu^i \left( \int_0^t \lambda^i(\varphi^i_u(x)) du , +\infty \right), \\
\mathbb{P} \left( I_{T_{n+1}} = j \, | \, I_{T_n} = i, X_{T_{n+1}} = y \right) = q_{i,j}(y).
\end{cases}
\end{equation}
Note that the motion of $(X,I)$ between the jump times of $I$ is deterministic. Hence, $(X,I)$ is a piecewise deterministic process, which is in general not Markov. However, setting 
\begin{equation}
\label{eq:introtau}
\tau_t = t - T_{N_t},
\end{equation} 
with $N_t$ the number of switches up to time $t$, we prove in Section \ref{sec:Markov} that $(X,\tau,I)$ is Markov. Note that $\tau_t$ is the  time elapsed since the latest switch before time $t$. If the probability measures $\mu^i$ have compact support, it is intuitive that $\tau_t$ cannot become larger than some constant. This is why, in the next section, we will introduce for the process $Z = (X, \tau,I)$ a state space $K$ which is typically smaller than $\R^d \times \R_+ \times E$.

\begin{remark} \rm
Adding the variable $\tau_t$ is  common in the study of so-called \emph{semi-Markov processes} (see e.g. \cite{LO01}). In particular, in the case where the jump rates do not depend on the position (i.e. the $\lambda^i$ are constant functions), the process $I$ is a semi-Markov process, $\tau_t$ is called \emph{the backward recurrence time} and the couple $(I, \tau)$ is Markov (see \cite[Theorem 3.2]{LO01}).
\end{remark}

\begin{remark} \rm
The classical case of exponential switching corresponds to the choice of an exponential law of parameter $1$ for all $\mu^i$. In that case, the process $(X,I)$ is Markov and whenever the $\lambda^i$ are constant, $I$ is a continuous-time Markov chain.
\end{remark}
\subsection{Standing assumptions and augmented state space} \label{ssec:assumptions} 

We first give the following standing assumptions on the four ingredients mentioned earlier:
\begin{assumption}
\label{hyp:standingF}
For all $i \in E$, $F^i$ is smooth and forward-complete, that is, for every $x_0 \in \R^d$, the initial-value problem 
\begin{align*}
\frac{dx}{dt} =& F^i(x(t)), \quad \forall t > 0, \\
x(0) =& x_0 
\end{align*}
has a unique solution, defined on a maximal interval $(-s^i(x_0), + \infty)$ with $s^i(x_0) > 0$. We denote by $t \mapsto \varphi^i_t(x_0)$ this maximal solution.
\end{assumption}
\begin{assumption}
\label{standingmu}
For all $i \in E$, $\mu^i$ is a Borel probability measure on $\R_+ := [0,\infty)$ such that $\mu^i(\{0\}) = 0$.
\end{assumption}

\begin{assumption}
\label{standinglambda}
For all $i \in E$, $\lambda^i: \mathbb{R}^d \to \mathbb{R}_+$ is continuous and 
\[
0 < \lambda_{\min}  := \inf_{(x,i) \in \R^d \times E} \lambda^i(x) \leq \sup_{(x,i) \in \R^d \times E} \lambda^i(x) := \lambda_{\max} < + \infty.
\]
\end{assumption}

\begin{assumption}
\label{hyp:standingQ}
The mapping $x \mapsto Q(x)$ from $\mathbb{R}^d$ into the space of $(N \times N)$ matrices with real entries is continuous and for each $x \in \R^d$, $Q(x)$ is a stochastic matrix, irreducible, with zeros on the diagonal. 
\end{assumption}
\smallskip 

\noindent To define the state space $K$ for the three-component process $(X,\tau,I)$, let us introduce for $i \in E$ the time 
$$
 \bar t^i = \inf\{t \geq 0:  G^i(t) = 0\}, 
$$
where $G^i(t) := \mu^i(t,\infty)$ and with the usual convention that $\inf \emptyset = +\infty$. Note that by right continuity of $ G^i$, $ G^i( \bar t^i) = 0$.  Set 
\begin{equation}
\label{eq:defK}
K= \biggl\{z = (x,s,i) \in \R^d \times \R_+ \times E: s < s^i(x), \quad \int_{-s}^0 \lambda^i(\varphi_r^i(x)) \ dr < \bar t^i  \biggr\}
\end{equation}
where we recall from Assumption \ref{hyp:standingF} that for all $x \in \R^d$, $s^i(x)$ is such that $(-s^i(x), \infty)$ is the maximal interval of definition for $t \mapsto \varphi^i_t(x)$. We also set, for all $i \in E$, 
$$ 
K^i = \left\{(x,s) \in \R^d \times \R_+: (x,s,i) \in K\right\}. 
$$
For $i \in E$, the set $K^i$ is open in $\R^d \times \R_+$: Indeed, writing 
$$ 
K^i = \{(x,s) \in \R^d \times \R_+: s < s^i(x)\} \cap \biggl\{(x,s) \in \R^d \times \R_+: \int_{-s}^0 \lambda^i(\varphi^i_r(x)) dr < \bar{t}^i \biggr\} =: K^i_1 \cap K^i_2, 
$$ 
one can conclude by observing that $K^i_2$ is open by joint continuity of 
$$ 
(x,s) \mapsto \int_{-s}^0 \lambda^i(\varphi^i_r(x)) dr 
$$ 
and that $K^i_1$ is open by Theorem I.3.1 in \cite{hale1980ordinary}. Also note that $\R^d \times \{0\} \times E\subset K$ because $\bar t^i > 0$ by virtue of $G^i(0) = 1$ and the fact that $G^i$ is right-continuous. 
For  $z = (x,s,i) \in K$ and $t \geq 0$, we define  
$$
 G_z(t) = G_{x,s}^i(t):=
\frac{G^i \left( \int_{-s}^t \lambda^i(\varphi^i_r(x)) \ dr \right)}{ G^i \left( \int_{-s}^0 \lambda^i(\varphi^i_r(x)) \ dr \right)}.  
$$
Observe that  $ G^i \left(\int_{-s}^0 \lambda^i(\varphi^i_r(x)) \ dr \right) > 0$ if and only if $z \in K$, so the quotient above is well-defined when $z \in K$. Moreover,  as $\lambda^i$ is bounded away from $0$, $1 - G_z$ is the cumulative distribution function of a Borel probability measure $\mu_z = \mu_{x,s}^i$ on $\mathbb{R}_+$. In case when $z = (x,0,i)$, we sometimes write $G^i_x$ and $\mu^i_x$ in lieu of $G^i_{x,0}$ and $\mu^i_{x,0}$. Note that the random variable $S_1$ defined by~\eqref{eq:S1intro} is distributed according to $\mu_x^i$. Informally, the construction of the process $(X,\tau,I)$ starting from a point $(X_0, \tau_0, I_0) = z = (x, s, i)$ with $s \geq 0$ is now as follows. Until the first jump time $T_1$ of $I$, whose law is given by $\mu_z$, $X_t = \varphi^i_t(x)$ while $\tau_t = s + t$. At time $T_1$, $I$ jumps to some $j \neq i$ with probability $q_{i,j}(X_{T_1})$, while $\tau$ jumps to $0$. After that, the process evolves as described at the beginning of Section \ref{sec:setting}, according to Equations~\eqref{eq:Xintro}, \eqref{eq:introI} and \eqref{eq:introtau}. 

We make this construction precise in Section~\ref{sec:PDP} below. There, we also give a rigorous justification for why $K$ can be used as the state space for $(X, \tau, I)$. To conclude Section~\ref{ssec:assumptions}, we motivate the choice of the law of $T_1 = S_1$ when $\tau_0 = s > 0$. By definition, $\tau_t$ is at time $t$ the time elapsed since the latest jump. If at time $0$, this time is equal to $s$, it is as if, when starting, the process had already waited for a time $s$ since some fictitious jump at time $-s$. Since, at time $0$, the process $(X,I)$ is at the point $(x,i)$, it is as if $(X,I)$ had started at the point $(\varphi_{-s}^i(x),i)$ at time $-s$ and had not performed a jump on the time interval $[-s,0]$. Thus, intuitively, $S_1 + s$ should have the law of the first jump time of the process $(X,I)$ starting from $(\varphi_{-s}^i(x),i)$, knowing that this first jump time is greater than $s$. In other words, we want that
\begin{equation}
\label{eq:surivalS1}
\mathbb{P}(S_1 > t) = \mathbb{P}(S_1 + s > t +s) = \frac{G^i(\int_0^{t+s} \lambda^i(\varphi^i_r(\varphi^i_{-s}(x))) dr)}{G^i(\int_0^s \lambda^i(\varphi^i_r(\varphi^i_{-s}(x))) dr)} = G_z(t),
\end{equation}
where the last equality comes from the flow property of $\varphi^i$.

%Finally, observe that if $Z_0$ was not an element of $K$, with the intuition described above, the process $(X,\tau,I)$ would have to jump immediately (and then would immediately enter $K$), a situation we want to avoid. That is why we use $K$ as the state space for the process $Z$. In Proposition \ref{prop:K} at the end of the next section, we give a rigorous justification for why this can indeed be done. 

\subsection{Construction of the process $(X,\tau,I)$} \label{sec:PDP}
We first define a marked point process $(T_k, \bm{Z}_k)_{k \geq 0}$ (see e.g. Definition 2.21 in \cite{J06}) with $(T_k)_{k \geq 0}$ taking values in $(0, \infty)$ and representing the switching times, and $(\bm{Z}_k)_{k \geq 0}= (\bm{X}_k, \bm{\tau}_k, \bm{I}_k)_{k \geq 0}$ taking values in $\R^d \times \R_+ \times E$ and giving the position of the process $(X,\tau,I)$ at switching times. Let $(U_k)_{k \geq 1}$ and $(V_k)_{k \geq 1}$ be independent i.i.d. sequences of random variables, living on some underlying probability space $(\Omega, \F, \mathbb{P})$, that are uniformly distributed on $(0,1]$. Expectation with respect to the probability measure $\mathbb{P}$ will be denoted by $\mathbb{E}$. For further use in the paper, we introduce the following $\sigma$-fields on $\Omega$: for a family of random variables $W  = (W_k)_{k \geq 1}$ (typically $W_k=U_k$, $V_k$ or $(U_k,V_k)$) and $m \in \mathbb{N}$,
\[
\mathcal{F}_m(W) = \sigma( W_k, \, 1 \leq k \leq m), \quad  \mathcal{F}(W) = \sigma( W_k, \,  k \geq 1).
\]
 For $z = (x,s,i) \in K$, let 
$$ 
\psi_z(u) = \psi_{x,s}^{i}(u) := \inf \{r \geq 0:  G_{x,s}^{i}(r) \leq u\}, \quad u \in (0,1].  
$$ 
Note that since $\lim_{r \to \infty} G_{x,s}^{i}(r) = 0$, the set $\{r \geq 0: \  {G}_{x,s}^{i}(r) \leq u\}$ is nonempty for every $u \in (0,1]$, so $\psi_{x,s}^{i} < \infty$. Also notice that $\psi_{x,s}^{i}(u) \leq t$ if and only if $G_{x,s}^{i}(t) \leq u$, which implies that the random variable $\psi_{x,s}^{i}(U_1)$ has distribution $\mu_{x,s}^{i}$. In addition, we define 
$$ 
\theta_x^i(u) = \theta_{x,i}(u)  := \begin{cases}
                                (x,0,1), & \quad u \in (0,q_{i,1}(x)], \\
                                (x,0,2), & \quad u\in (q_{i,1}(x), q_{i,1}(x) + q_{i,2}(x)], \\
                                      \vdots & \quad \vdots \\
                                 (x,0,N), & \quad u \in (1 - q_{i,N}(x), 1]. 
\end{cases}
$$ 
Let $z=(x,s,i)$ be a starting point in $K$. Set $T_0 := 0$ and $\bm{Z}_0 = (\bm{X}_0, \bm{\tau}_0, \bm{I}_0) := (x,s,i)$. Then we set
$$
S_1 := \psi_{\bm{Z}_{0}}(U_1), \quad T_1 := T_{0} + S_1 = S_1, \quad \bm{Z}_1 = (\bm{X}_1, \bm{\tau}_1, \bm{I}_1) := \theta_{\varphi^{\bm{I}_{0}}_{S_1}(\bm{X}_{0})}^{ \bm{I}_{0}}(V_1). 
$$
As noticed before, the law of $S_1$ is $\mu_{x,s}^i$, exactly as for the random variable $S_1$ of Equation~\eqref{eq:S1intro} in the case where $s=0$, or, in the general case, as for $S_1$ in Equation~\eqref{eq:surivalS1}. Thus, $T_1 = S_1$ represents the first switching time. Now, by definition of $\theta$, $\bm{X}_1 = \varphi^i_{T_1}(x)$ is the position of the continuous component $X$ at time $T_1$; while $\bm{I}_1 = j$ with probability $q_{i,j}(\bm{X}_1) = q_{i,j}(\varphi^i_{T_1}(x))$, and hence $\bm{I}_1$ can be used to construct the discrete component $I$ at the first switching time $T_1$. Finally, $\bm{\tau}_1 = 0$, which is consistent with the idea that a jump occurs at time $T_1$. We now  iteratively define for $k \geq 2$ 
$$
S_k := \psi_{\mathbf{Z}_{k-1}}(U_k), \quad T_k := T_{k-1} + S_k, \quad \bm{Z}_k = (\bm{X}_k, \bm{\tau}_k, \bm{I}_k) := \theta_{\varphi^{\bm{I}_{k-1}}_{S_k}(\bm{X}_{k-1})}^{ \bm{I}_{k-1}}(V_k). 
$$
Observe that $\bm{\tau}_k = 0$ for every $k \geq 1$. 
The resulting process $(T_k, \bm{Z}_k)_{k \geq 0}$ is a marked point process from which we now define the process $Z = (X,\tau,I)$.

\begin{definition}
For $k \in \N_0$ and $T_k \leq t < T_{k+1}$, let 
\begin{equation}
\label{eq:PDPfromMPP}
Z_t = (X,\tau,I)_t :=( \varphi_{t-T_k}^{\bm{I}_k}(\bm{X}_k), \bm{\tau}_k + t-T_k, \bm{I}_k).
\end{equation}
\end{definition}
Expressed slightly differently, for $T_k \leq t < T_{k+1}$, $ (X, \tau, I)_t = \phi_{t-T_k}(\bm{Z}_k)$, where $(\phi_t)_{t  \geq 0}$ is the semi-flow on $\R^d  \times \R_+ \times E$ defined by
$$
\phi_t(x,s,i) := ( \varphi_t^i(x), t +s, i).
$$
One sees easily that the component $X$ in \eqref{eq:PDPfromMPP} satisfies Equation \eqref{eq:Xintro}, while
\[
\tau_t = \begin{cases}
\bm{\tau}_0 + t \quad \mbox{if} \quad t < T_1\\
t - T_{N_t} \quad \mbox{otherwise.}
\end{cases}
\]
Here, 
$$
N_t := \sum_{k=1}^{\infty} \mathbbm{1}_{t \geq T_k}
$$
denotes the number of switches that have occurred up to time $t$. 

\begin{remark}
\rm It is clear from its definition that $(Z_t)_{t \geq 0}$ has right-continuous sample paths almost surely.
\end{remark}

\begin{remark} \rm Note that since $\psi_z(u) = 0$ is equivalent to $u \geq G_z(0) = 1$, we have $\mathbb{P}(\bigcup_{k \geq 1} \{S_k = 0\}) = 0$. Hence, with probability one, there is a strictly positive time gap between any two switches.
\end{remark}

We conclude this subsection with two propositions. In Proposition~\ref{prop:finitejumps}, we prove that there cannot be an accumulation of jumps in finite time, i.e. $T_k \to \infty$ almost surely as $k$ goes to infinity. This implies that the process $Z$ is non-explosive and well defined for all $t \geq 0$. In Proposition~\ref{prop:K}, we prove that one has almost surely $Z_t \in K$ for all $t \geq 0$, so $K$ can indeed by taken as the state space for the process $Z$.

\begin{proposition}
We have
$$
\sup_{z \in K} \mathbb{E}_z( N_t) < \infty.
$$
In particular, $T_k \to \infty$ almost surely as $k \to \infty$.
\label{prop:finitejumps}
\end{proposition}

\bpf
%Let $\lambda > 0$ such that
%$$
%\sup_{(x,i) \in \R^d \times E} \lambda_i(x) \leq \lambda.
%$$
By Assumption \ref{standinglambda}, for all $(x,i) \in \R^d \times E$ and $t \geq 0$, $G_{x,0,i}(t) \geq G^i( \lambda_{\max} t) \geq H(t)$, where 
$$
H(t) := \min_{i \in E} G^i( \lambda_{\max} t).
$$
Since for all $i$, $1-G^i$ is a cumulative distribution function, so is $1-H$. If we let $(\tilde S_k)_{k \geq 1}$ be an i.i.d. sequence of random variables with cumulative distribution function $1-H$, $\tilde{T}_0 = 0$ and $\tilde{T}_k = \tilde{T}_{k-1} + \tilde{S}_k$ for $k \geq 1$, then for all $z \in K$ 
\begin{equation}   \label{eq:N_t_ineq}
\EE_{z}( N_t ) \leq 1 + \sum_{k=1}^{\infty} \mathbb{P}(t \geq \tilde{T}_k).
\end{equation}
Note that since the $\mu^i$ do not put mass on $0$, neither does the distribution associated to $1-H$. This implies that the right-hand side of \eqref{eq:N_t_ineq} is finite and thus the result (see e.g. section 1.4 in \cite{LO01}).
\epf

%\bpf 
%With the notations of the proof of Lemma \ref{lem:finitejumps}, let $\hat{T}_0 = 0$, $\hat{T}_1 = T_1$ and  $\hat{T}_{k+1} =  \hat{T}_k + \tilde{S}_{k}$ and denote by $\hat{N}$ the counting process associated to $(\hat{T}_k)_{k \in \mathbb{N}}$. Then, we have 
%\[
%\mathbb{E}_z(N_t) \leq \mathbb{E}_z( \hat N_t).
%\]
%For all $k \geq 1$, we have $\hat{T}_k = T_1 + \tilde{T}_{k-1}$, which implies that for $k \geq 2$, 
%\begin{equation}
%\mathbb{P}_z( t \geq \hat{T}_k) \leq \mathbb{P} ( t \geq \tilde{T}_{k - 1})
%\end{equation}
%and thus 
%\[
%\mathbb{E}_z(N_t) \leq \mathbb{P}_z(t \geq T_1) + \mathbb{E}( \tilde{N}_t),
%\]
%which concludes the proof.
%\epf

\begin{proposition}    \label{prop:K}
For every $t \geq 0$, one has $Z_t \in K$.      
\end{proposition}

\begin{proof}
We recall the notation
$$ 
K^i = \{(x,s) \in \R^d \times \R_+: s < s^i(x)\} \cap \biggl\{(x,s) \in \R^d \times \R_+: \int_{-s}^0 \lambda^i(\varphi^i_r(x)) dr < \bar{t}^i \biggr\} =: K^i_1 \cap K^i_2
$$ 
and for $h = 1, 2$ we set
\[
K_h = \bigcup_{i \in E} K^i_h \times \{i\}.
\]
We also note that by the Cauchy - Lipschitz theorem, for all $x \in \R^d$, $i \in E$ and $t \geq 0$, $s^i(\varphi_t^i(x)) = s^i(x) + t$.

Let $Z_0 = (x,s,i) \in K$.    First, since $\mathbb{R}^d \times \{0\} \times E \subset K$, we have that $\bm{Z}_n \in K$ for all $n \geq 0$. Now, we prove that for all $n \geq 0$, one has $Z_t \in K \ \forall t \in [T_n, T_{n+1})$. First, consider the case where $n = 0$. If $t \in [0, T_1)$, then $Z_t = ( \varphi_t^i(x), t+s, i)$. Since $Z_0 \in K$, we have $s < s^i(x)$ and therefore $t + s < t + s^i(x) = s^i(\varphi^i_t(x))$ which entails that for all $t \in [0, T_1)$, $Z_t \in K_1$. Now, we claim that for all $t < T_1$, $\int_{-s}^t \lambda^i( \varphi^i_r(x)) dr < \bar t^i$, which implies that $Z_t \in K_2$. Indeed, if one had $\int_{-s}^t \lambda^i( \varphi^i_r(x)) dr \geq \bar t^i$, then, by definition of $\bar t^i$, $G^i(\int_{-s}^t \lambda^i( \varphi^i_r(x)) dr) = 0$, hence $G^i_{x,s}(t) = 0$. But since 
    \[
T_1 = S_1 = \psi_z(U_1) = \inf\{ r \geq 0: \: G^i_{x,s}(r) \leq U_1 \},
    \]
 $G^i_{x,s}(t) = 0$ implies that $T_1 \leq t$. As a result,  $Z_t \in K$ for all $t \in [0, T_1)$.  Now, let $n \geq 1$. Then, for all $t \in [T_n, T_{n+1})$, we have $Z_t = (\varphi_{t - T_n}^{\bm{I}_n}(\bm{X_n}), t - T_n, \bm{I}_n)$. In addition, for all $t \in [T_n, T_{n+1})$, one has $t - T_n < t - T_n + s^{\bm{I}_n}( \bm{X}_n) = s^{\bm{I}_n}( \varphi^{\bm{I}_n}_{t - T_n}(\bm{X}_n))$. Therefore, for all $t \in [T_n, T_{n+1})$, $Z_t \in K_1$. Now, as for the case $n = 0$, using that 
 \[
 S_{n+1} = \psi_{\bm{Z}_n}(U_{n+1}) = \inf \{r \geq 0 : \: G^{\bm{I}_n}_{\bm{X}_n}(r) \leq U_{n+1} \},
 \]
 we have that for all $t \in [T_n, T_{n+1})$, $\int_0^{t - T_n} \lambda^{\bm{I}_n} ( \varphi_r^{\bm{I}_n}(\bm{X}_n)) dr < \bar t^{\bm{I}_n} $, which proves that for all $t \in [T_n, T_{n+1})$, $Z_t \in K_2$. Hence, we have proven that, for all $t \geq 0 $, $Z_t \in K$.
\end{proof}

\medskip

%We also prove that  $K$ can indeed be taken as the state space for the process $Z$:

\subsection{The Markov property and link with Davis' Piecewise Deterministic Markov Process} \label{sec:Markov}

We now state and prove the main result of Section \ref{sec:setting}.

\begin{proposition}      \label{prop:strong_Markov} 
The stochastic process $Z = (X,\tau, I)$ is strong Markov. 
\end{proposition} 

\bpf We will derive the strong Markov property from Theorem~7.5.1 in~\cite{J06}. In order for $Z$ to fit the framework in~\cite{J06}, we need to define the process on a larger state space $\widetilde{K} \supset K$ which, unlike $K$, is invariant under the semi-flow $(\phi_t)_{t \geq 0}$. 
Define  
    $$
    \widetilde{K} = \left\{z = (x,s,i) \in \R^d \times \R_+ \times E: s < s^i(x) \right\}. 
    $$
If the starting point $z = (x,s,i)$ of $Z$ is contained in $K$, we define $Z$ exactly as before. And if $z \in \widetilde{K} \setminus K$, we simply set $Z_t = (\varphi^i_t(x), s+t, i) = \phi_t(z)$ for every $t \geq 0$. By Theorem 7.5.1 in~\cite{J06}, it is enough to show that the process $Z$, considered on the extended state space $\widetilde{K}$, is a Piecewise Deterministic Markov Process (PDMP) in the sense of Jacobsen. That is, it is a Piecewise Deterministic Process (in the sense of Section 3.3 in \cite{J06}) with the Markov property. For $z \in \widetilde{K} \setminus K$, define $G_z(t) = 1$ for every $t \geq 0$. Then $1 - G_z$ can be interpreted as the cumulative distribution function of the probability measure on $[0, +\infty]$ which gives full mass to $+\infty$. By our construction and by Theorem 7.3.2 in \cite{J06}, we only need to check that for all $t,r \geq 0$ and for all $z \in \widetilde{K}$,
\begin{equation}
\label{eq:Gfunctional}
G_z(t+r) = G_z(r)  G_{\phi_r(z)}(t).
\end{equation}
Assume first that $\phi_r(z) \in K$. Then one also has $z \in K$ and, with $z = (x,s,i)$, 
\begin{align*}
 G_{\phi_r(z)}(t) & = 
\frac{ G^i \left( \int_{- (s+r)}^t \lambda^i(\varphi^i_u(\varphi^i_r(x))) \ du \right)}{ G^i \left( \int_{- (s+r)}^0 \lambda^i(\varphi^i_u(\varphi^i_r(x))) \ du \right)} \\
& = \frac{G^i \left( \int_{- s}^{t+r} \lambda^i(\varphi^i_u(x)) \ du \right)}{ G^i \left( \int_{-s}^r \lambda^i(\varphi^i_u(x)) \ du \right)} = \frac{ G_z(t+r)}{ G_z(r)}.
\end{align*}
Now assume that $\phi_r(z) \in \widetilde{K} \setminus K$. If $z \in \widetilde{K} \setminus K$, then $G_z(t+r) = G_z(r) = 1$ and~\eqref{eq:Gfunctional} clearly holds. If $z \in K$, we argue as follows: Since $\phi_r(z) \in \widetilde{K} \setminus K$, one has 
$$
\bar{t}^i \leq \int_{-(r+s)}^0 \lambda^i(\varphi^i_u(\varphi^i_r(x))) \ du = \int_{-s}^r \lambda^i(\varphi^i_u(x)) \ du, 
$$
so 
$$
G^i \left(\int_{-s}^r \lambda^i(\varphi^i_u(x)) \ du \right) = 0 = G^i \left(\int_{-s}^{t+r} \lambda^i(\varphi^i_u(x)) \ du \right).  
$$
This implies $G_z(r) = G_z(t+r) = 0$ and hence~\eqref{eq:Gfunctional}.

%Now assume that $z = (x,s,i) \in \widetilde{K} \setminus K$. Then, either $ G_i \left( \int_{-s}^{0} \lambda_i(\varphi^i_u(x)) \ du \right) = 0$, in which case $  G_{x,s,i}(r) =0$, or $ G_i \left( \int_{-s}^{0} \lambda_i(\varphi^i_u(x)) \ du \right) > 0$  and $ G_i \left( \int_{-s}^{t+r} \lambda_i(\varphi^i_u(x)) \ du \right)=0$, in which case $ G_{\phi_r(x,s,i)}(t) = 0$. It is clear that when $G_{x,s,i}(r) =0$, then $e G_{x,s,i}(t+r) =0$. Finally, assume that $ G_{\phi_r(x,s,i)}(t) = 0$. Then, either $G_i \left( \int_{-s}^r \lambda_i(\varphi^i_u(x)) \ du \right)=0$  or $ G_i \left( \int_{-s}^r \lambda_i(\varphi^i_u(x)) \ du \right) >0$ and $ G_i \left( \int_{-s}^{t+r} \lambda_i(\varphi^i_u(x)) \ du \right) = 0$. In both cases, it is readily seen that $ G_{x,s,i} = 0$. Thus \eqref{eq:Gfunctional} is satisfied, and the proposition is proved by Theorems 7.3.2 and 7.5.1 in \cite{J06}.
\epf

\begin{remark}   \rm 
  Note that when $z=(x,0,i)$, it is easily proven that the stochastic process $(\bm{X}_k, \bm{I}_k)_{k \geq 0}$ is a Markov chain on $\R^d \times E$, with Markov kernel 
$$ 
P_{x,i}(A \times E') := \sum_{j \in E'} \int_0^{\infty} \mathbbm{1}_A(\varphi^i_t(x)) q_{i,j}(\varphi^i_t(x)) \mu_{x,i}(dt). 
$$ 
\end{remark}
According to Proposition~\ref{prop:strong_Markov}, the process $Z$ is a Markov process. It is also clear that $Z$ is piecewise deterministic. A natural question is then whether $Z$ is a Piecewise Deterministic Markov Process (PDMP) in the sense of Davis \cite{Dav93}. According to Davis' definition of PDMP, the law of the first jump time starting from a point $z \in K$ should be given by 
$$
\mathbb{P}_z(T_1 > t) = \exp \biggl( - \int_0^t \alpha(\phi_r(z)) \ dr \biggr) \1_{t < t^*(z)}, 
$$
where $\alpha : \R^d  \times \R_+ \times E \to \R_+$ is a measurable function and $t^*(z)$ denotes the hitting time of the boundary of $K$ by the flow $t \mapsto \phi_t(z)$. 

\begin{proposition}
 Assume that each $ G^i$ is  absolutely continuous on $[0, \bar{t}^i)$. Then the process $Z$ is a PDMP in the sense of Davis.
\end{proposition}

\bpf We first observe that for all $z = (x,s,i) \in K$,
$$
t^*(z) = \inf\{ t \geq 0 \: : \phi_t(z) \in \partial K \} = \inf \{ t \geq 0 \: : \int_{-s}^t \lambda^i( \varphi_r^i(x)) \ dr = \bar t^i \}.
$$
For all $z \in K$, we have by definition
$$
\mathbb{P}_z(T_1 > t) =  G_z(t) =  \frac{ G^i \left( \int_{-s}^t \lambda^i(\varphi^i_r(x)) \ dr \right)}{ G^i \left( \int_{-s}^0 \lambda^i(\varphi^i_r(x)) \ dr \right)},
$$
which is strictly positive if and only if $t < t^*(z)$. 
%If $\mu_i$ is a Dirac mass at $\bar t_i$, then the above quantity is equal to $1$ for all $t < t^*(z)$, hence choosing $\alpha(x,s,i) =0$ gives the result. 
Since  the function $G^i$ is absolutely continuous on $[0, \bar t^i)$,  there exists an integrable nonpositive function $g^i$ such that for all $t \in [0, \bar t^i)$, $ G^i(t) = 1 + \int_0^{t} g^i(r) \ dr.$ Set for all $z=(x,s,i) \in K$
$$
\alpha(z) = \frac{- \lambda^i(x) g^i \left( \int_{-s}^0 \lambda^i(\varphi^i_r(x)) \ dr \right)}{ G^i \left( \int_{-s}^0 \lambda^i(\varphi^i_r(x)) \ dr \right)}.
$$
Note that, for all $t < t^*(z)$, $\alpha(\phi_t(z)) = - {G}'_z(t)/ G_z(t)$. 
Thus,
$$
\mathbb{P}_z(T_1 > t) = \exp \biggl( - \int_0^t \alpha(\phi_s(z)) \ ds \biggr) \1_{t < t^*(z)},
$$
which concludes the proof.
\epf

\bigskip

\begin{remark}  \rm 
As an immediate corollary of the previous proposition and of Theorem 26.14 in \cite{Dav93}, we can write down a formula for the extended generator $L$ of $Z$ whenever the $G^i$'s are absolutely continuous. In that case, $L$ acts on regular functions $f : K \to \mathbb{R}$ according to 
 
\begin{align*}
Lf(x,s,i) = & \langle F^i(x), \nabla_x f(x,s,i) \rangle + \frac{\partial f}{\partial s}(x,s,i)  \\
& \quad + \lambda^i(x) \frac{g^i \left( \int_{-s}^0 \lambda^i(\varphi^i_r(x)) \ dr \right)}{ G^i \left( \int_{-s}^0 \lambda^i(\varphi^i_r(x)) \ dr \right)} \sum_{j \in E} q_{i,j}(x) \left( f(x,0,j) - f(x,s,i) \right).
\end{align*}
\end{remark}

%The statement below is an immediate corollary of the previous lemma and of Theorem 26.14 in \cite{Dav93}.
%
%\begin{corollary}
% Assume  the $G_i$'s are absolutely continuous. Then the extended generator $L$ of $Z$ acts on functions $f : K \to \mathbb{R}$ according to 
% 
%\begin{align*}
%Lf(x,i,s) = & \langle F^i(x), \nabla_x f(x,i,s) \rangle + \frac{\partial f}{\partial s}(x,i,s)  \\
%& \quad + \lambda_i(x) \frac{g_i \left( \int_{-s}^0 \lambda_i(\varphi^i_r(x)) \ dr \right)}{\overline G_i \left( \int_{-s}^0 \lambda_i(\varphi^i_r(x)) \ dr \right)} \sum_{j \in E} q_{i,j}(x) \left( f(x,j,0) - f(x,i,s) \right).
%\end{align*}
%
%\end{corollary}

\bigskip 

\section{Feller Property and invariant measure}
\label{sec:Feller}
 In this section we shall investigate the Feller property of the Markov process $Z$, as well as existence of invariant measures under compactness assumption on the flows.
 \subsection{Feller Property} In order to avoid confusion, we give our definition of Feller here, which is weaker than the usual one.
 
 \begin{definition}
 A Markov semigroup $(Q_t)_{t \geq 0}$ on a metric space $(M,d)$ is \emph{Feller} if for all continuous bounded functions $f : M \to \R$ and all $t \geq 0$, the function $(t,x) \mapsto Q_t f(x)$ is also continuous and bounded. 
 \end{definition}
 
From now on, we let $(P_t)_{t \geq 0}$ be the Markov semigroup on $K$ associated to $Z$, and let $C_b(K)$ denote the set of continuous bounded functions $f : K \to \mathbb{R}$. The following example shows that in general, this semigroup is not Feller.

\begin{example}    \rm
We consider $d = 1$, $E = \{0,1 \}$ and $\mu_0 = \delta_1$, $\lambda_0 = \lambda_1 = 1$ and two globally integrable vector fields $F^0 \neq F^1$. In particular, we have $G_0(t) = \1_{t \geq 1}$. Let $x \in \R$ and $s \in (0,1)$. Then the first jump time $T_1$ starting from $(x,0,s)$ is deterministic and occurs at time $1 - s$. At time $t - s$, the component $I$ in $E$ jumps to $1$ and $\tau$ jumps to $0$.  In particular, 
$$P_{1 - s} f(x,0,s) = f( \varphi^0_{1 - s}, 1, 0).
$$
 Now for $\varepsilon < s$, the process $Z$ started from $(x,0, s - \varepsilon)$ performs its first jump at the deterministic time $1  - s + \varepsilon $. In particular, it has not jumped at time $1 - s$, and thus 
 $$
 P_{1 - s} f(x,s,i - \varepsilon) = f( \varphi_{1 - s}^0(x), 0, 1 - \varepsilon).
 $$
 Thus $P_{1 - s} f$ is not continuous in general; for example if $f(x,s,i) = \frac{1}{1 + s}$. 
\end{example} 

The following proposition shows that the continuity of the $ G^i$ is a necessary and sufficient condition for the Feller property. Its proof is given in Section~\ref{sec:proofFeller}

\begin{proposition}
\label{prop:Feller}
The process $Z$ is Feller if and only if $ G^i$ is continuous for all $i$. 
\end{proposition}

\begin{remark}
A widespread definition of the Feller property for a Markov semigroup $(Q_t)_{t \geq 0}$ on a locally compact space $M$ is that $Q_t$ preserves $C_0(M)$, the set of continuous function that vanishes at infinity. Let us call such a semigroup $C_0$ - Feller. Then, PDMPs defined by the switching between vector fields are, in general, not $C_0$ - Feller. Here is simple example. Consider  two vector fields on $\R_+$, given by  $F^0(x) = - x^2$ and $F^1(x) = x$, randomly switched with rate $\lambda_0$ and $\lambda_1$ (ie, $  G_i(t) = e^{-  t}$). This is a  PDMP in the classical sense, whose state space is $\R_+ \times E$, with $ E = \{0, 1 \}$. The choice of $F^1$ implies that there is no positively invariant compact set for the flow generated by the $F^i$. 
The flow $\varphi^0$ is  given by $\varphi^0_t(x) = 1 / ( 1/x + t)$, for all $x \in \R_+$ and $t \geq 0$. Let  $f(x,i) = 1/( 1 + x)$, for $i= 0, 1.$ Then obviously, $f$ belongs to $C_0$. Let us show that for all $t > 0$,  $P_t f$ is not in $C_0$.
Since $f$ is positive, one has
\[
P_t f(x,0) \geq f( \varphi^0_t (x) ) \PP_{x,i} ( T_1 > t)  = \frac{1}{1 + 1 / ( 1 / x + t) } e^{ - \lambda_0 t}\]

Yet when $x$ goes to infinity,  $1 / ( 1 + 1 / ( 1 / x + t) )$  is equivalent to  $1 / ( 1 + 1 / t)$, which proves that  $P_t f (x,0)$ does not converges to $0$ as $x$ goes to infinity, thus that $P_t$ is not $C_0$ - Feller. 
\end{remark}

\subsection{Existence of an invariant measure under compactness assumption}
\label{ssec:existence_im_comp}
In this section, we make the following assumption:
\begin{assumption}
\label{hyp:compact}
There exists a compact set $M \subset \mathbb{R}^d$ which is positively invariant for the flow $\varphi^i$, for all $i \in E$. That is, for all $t \geq 0$, $\varphi_t^i(M) \subset M$.
\end{assumption}
Under Assumption \ref{hyp:compact}, it is easily seen that the set $\{(x,s,i) \: : (x,\varphi_{-s}^i(x)) \in M^2 \}$ is positively invariant for the flow $\phi$. Thus, when Assumption \ref{hyp:compact} holds, we introduce the  space 
\begin{equation}
\label{eq:KM}
K_M = \{ (x,s,i) \in K \: : (x,\varphi_{-s}^i(x)) \in M^2 \},
\end{equation}
which is positively invariant for the process $Z$ : if $Z_0 \in K_M$, then $Z_t \in K_M$ for all $t \geq 0$. Therefore, in that case, we use $K_M$ as the natural state space for $Z$. 
\begin{example}
We compute $K_M$ on a simple example. On $\R$, consider the two vector fields $F^0(x) = x$ and $F^1(x) = 1 - x$. Let $\mu_0 = \mu_1$ be the uniform law over $[0,2]$ and set $\lambda_0 = \lambda_1 \equiv 1$. In that case, $K = \R \times [0,2) \times \{0,1\}$. Furthermore, the set $M = [0,1]$ is positively invariant for both $\varphi^0$ and $\varphi^1$, which are given, for all $t \in \R$ by
\[
\varphi^i_t(x) = i + (x- i)e^{-t}.
\]
Thus, it is easy to compute that for all $x \in M$, $\varphi^i_{-s}(x)$ is also in $M$ if and only if $s \leq -\log | i - x|$. Hence :
\[
K_M = \{ (x,s,i) \in K \: : x \in [0,1], \quad s \in [0, - \log|i-x|] \}.
\] 
\end{example}
\begin{proposition}
Assume that the semigroup is Feller and that Assumption \ref{hyp:compact} holds. Then, there exists at least one invariant probability measure $\mu$ for $(P_t)_{t \geq 0}$ on $K_M$.
\end{proposition}
\bpf
By the Krylov-Bogoliubov procedure and the Feller property, it is sufficient to prove that there exists $z_0 \in K_M$ such that the sequence of empirical probability measures $(\mu_t^{z_0})_{t > 0}$ is tight, where for all $t > 0$, 
\[
\mu_t^{z_0} = \frac{1}{t}\int_0^t P_u(z_0, \cdot) du.
\]
We prove that  $(\mu_t^{z_0})_{t > 0}$ is tight for every $z_0 = (x,i,0) \in K_M$. First, we claim that for all $\varepsilon > 0$, there exists $T > 0$ such that, for all $n \geq 0$,
\[
\mathbb{P}_{z_0} ( T_{n+1} - T_n > T) \leq \varepsilon.
\]
Indeed, let $n \geq 0$. Then, with the notations introduced in Section \ref{sec:setting},
\begin{align*}
\mathbb{P}_{z_0} ( T_{n+1} - T_n > T) & =\mathbb{P}_{z_0}( S_{n+1} > T)\\
& = \mathbb{E}_{z_0} \left[   G_{Y_n} (T) \right].
\end{align*}
Thus, the claim is proven if we show that for all $(x,i) \in M \times E$, there exists $T > 0$ such that $  G_{x,i}(T) \leq \varepsilon$. 
%Since $M$ is compact and $\lambda_i$ is continuous and positive, there exists $\lambda_{\min} > 0$ such that for all $x' \in M$, $\lambda_i(x') \geq \lambda_{\min}$. In particular, 
Assumption~\ref{standinglambda} yields that for all $t \geq 0$ and $(x,i) \in M$,  
\[
  G_{x,i}(t) =   G_i \left( \int_0^t \lambda_i( \varphi^i_u(x) ) du \right) \leq   G_i( \lambda_{\min} t) \leq \max_i   G_i(\lambda_{\min} t).
\]
This proves the claim, since for all $i \in E$, $  G_i(\lambda_{\min} t)$ converges to $0$ when $t$ goes to infinity. Now let $z_0 =(x,i,0) \in K_M$. Then, since $\tau_0 = 0$,  for all $u \geq 0$, $\tau_u = u - T_{N_u}$. In particular, $\tau_u \leq T_{N_u + 1}-T_{N_u}$ and thus
\[
\mathbb{P}_{z_0}( \tau_u > T) \leq \mathbb{P}_{z_0}( T_{N_u + 1} - T_{N_u} > T)\leq \varepsilon.
\]
This implies that for all $u \geq 0$, $P_u(z_0, M \times E \times [0,T] ) \geq 1 - \varepsilon$ and therefore that $\mu_t^{z_0}( M \times E \times [0,T]) \geq 1 - \varepsilon$. This concludes the proof.
\epf

\subsection{Construction of a Lyapunov function}
Here we have proven the existence of an invariant probability measure by proving by hand that the sequence of empirical measure is tight. Another way to obtain the tightness of this sequence, if the existence of a \textit{Lyapunov function}. We shall say that a function $f : K_M \mapsto \R_+$ is a Lyapunov function if $f$ is a proper map (ie, $\{ f \leq R\}$ is compact for all $R \geq 0$) and if there exists positive constants $C, \gamma$ and nonnegative $C'$ such that, for all $t \geq 0$, 
\[
P_t f \leq C e^{- \gamma t} f + C'.
\]
The existence of a Lyapunov function implies that the process comes back exponentially fast in compact sets. There is no hope to find such a function if one of the laws $\mu^i$ of the jump times have a heavy tail. Indeed, in that case, the return time to compact sets of the form $\{(x,s,i) : \: s \in [0,s^*] \}$ for some $s^* >0$, is given when the process starts at a point $(x,s,i)$ with $s > s^*$  by $T_1$, which does not have exponential moments.   The next proposition gives existence of a Lyapunov function under the assumption that the survival function $  G_i$ have an exponential decay. Its proof is given in Section~\ref{sec:proofLyap}
\begin{proposition}
\label{prop:lyap}
Grant Assumption \ref{hyp:compact} and assume that:
\begin{enumerate}
\item For all $i \in E$, $  G_i$ is continuous;
\item There exists $C, \beta > 0$ such that, for all $i \in E$, for all $t \geq 0$,
\[
  G_i(t) \leq C e^{- \beta t}.
\]
\end{enumerate}  
Let $\lambda_{\min} = \min_{(x,i) \in M \times E} \lambda_i(x)  > 0$, and set $\gamma = \delta \beta \lambda_{\min}$, for some $\delta \in (0, 1)$. Then, the function
\[
f(x,s,i) = \frac{e^{-\gamma s}}{  G_i\left(  \int_{-s}^0 \lambda_i (\varphi^i_u(x)) du \right) }\left( \mathbb{E}_{\varphi_{-s}^i(x), 0,i}[ e^{\gamma S_1} \1_{S_1 \leq s}] - 1 \right) + 1
\]
is a Lyapunov function. More precisely, for all $t \geq 0$ 
\[
P_t f\leq e^{- \gamma t} f + 1. 
\]
\end{proposition}

\section{Uniqueness, exponential ergodicity, and absolute continuity of the invariant probability measure}   \label{sec:uni_ac} 
\label{sec:ergodicity}
In the previous section, we have given sufficient conditions for the existence of an invariant probability measure for $Z$. The purpose of this section is to give conditions ensuring that the process $Z$ admits at most one invariant probability measure, and, in case there does exist a unique invariant probability measure $\pi$, ensuring exponential convergence of the law of $Z_t$ to $\pi$ as $t \to \infty$. We shall also prove that under these conditions, $\pi$ is absolutely continuous with respect to $\mathbf{L}$, the product of $(d+1)$-dimensional Lebesgue measure $\bm{\lambda}_{d+1}$ and the counting measure on $E$, restricted to $K$.  

Our results rely on the one hand on classical results on Markov processes that we recall in Section~\ref{ssec:background_Markov} and on the other hand on the strategy adopted in \cite{BH12} and \cite{BMZIHP} for the particular case of PDMP with exponential switching times.

\subsection{Some background on Markov processes} \label{ssec:background_Markov} 

Below we collect the definitions of several classical concepts from the theory of Markov processes and recall a crucial result on uniqueness and exponential ergodicity for the invariant probability measure. Throughout Section~\ref{ssec:background_Markov}, let $M$ be a separable metric space and let $(Q_t)_{t \geq 0}$ be a Markov semigroup on $M$ equipped with the Borel $\sigma$-field $\Bc(M)$. 

\begin{definition}
We say that $z^* \in M$ is a \emph{Doeblin point} if there exist a neighbourhood $U$ of $z^*$, a nonzero measure $\nu$ on $(M, \Bc(M))$ and $t^* > 0$ such that 
$$
Q_{t^*}(z, \cdot) \geq \nu(\cdot), \quad \forall z \in U. 
$$
In this case, we also call $z^*$ a Doeblin point with respect to the Markov kernel $Q_{t^*}$. The set $U$ is called a \emph{small set} with respect to $Q_{t^*}$. 
\end{definition} 

\begin{definition}     \label{def:resolvent} 
The \emph{resolvent} associated with $(Q_t)$ is the Markov kernel $R$ on $M$ defined by 
$$
R(z, \cdot) := \int_{\R_+} e^{-t} Q_t(z, \cdot) \ dt, \quad z \in M. 
$$
We call a set $U \in \Bc(M)$ a \emph{petite set} if there exists a nonzero measure $\nu$ on $(M, \Bc(M))$ such that 
$$
R(z, \cdot) \geq \nu(\cdot), \quad \forall z \in U. 
$$
\end{definition}

The existence of a petite (or small) set yields results on the uniqueness and absolute continuity of the invariant probability measure if the set is accessible, in the following sense:

\begin{definition}
A point $z^* \in M$ is called $(Q_t)$-\emph{accessible} from $z \in M$ if for every neighbourhood $U$ of $z^*$, there is $t > 0$ such that $Q_t(z, U) > 0$. It is called $(Q_t)$-accessible from a set $A \in \Bc(M)$ if it is $(Q_t)$-accessible from every point $z \in A$. Finally, $z^*$ is simply called $(Q_t)$-accessible if it is $(Q_t)$-accessible from $M$. 
\end{definition} 

\begin{remark}    \rm \label{rm:support_res} 
Suppose that for every $z \in M$ and for every open $U \subset M$, the mapping $t \mapsto Q_t(z,U)$ is lower semicontinuous from the right, i.e., $\liminf_{u \downarrow t} Q_u(z,U) \geq Q_t(z,U)$ for every $t$. By the Portmanteau theorem (see, e.g., Theorem 2.1 in \cite{Bil}), this is for instance the case if $(Q_t)$ is Feller. It also holds for the semigroup $(P_t)$ of our switching process $Z$, even if $(P_t)$ does not have the Feller property. Then a point $z^* \in M$ is $(Q_t)$-accessible from $z \in M$ if and only if for every neighbourhood $U$ of $z^*$, one has $R(z, U) > 0$ (i.e., $z^*\in \supp R(z, \cdot)$). Notice that the ``if'' direction of this statement does not require lower semicontinuity. 
\end{remark} 

In the following theorem, we recall some classical results on Markov processes that will be useful in our investigation. 

\begin{theorem}     \label{thm:uniqueness}
The following statements hold. 
\begin{enumerate}
\item Suppose that for every $z \in M$ and for every open $U \subset M$, $t \mapsto Q_t(z,U)$ is lower semicontinuous from the right. If there exists an open and petite set $U$ which contains a $(Q_t)$-accessible point $z^*$, then $(Q_t)$ admits at most one invariant probability measure. 
\item Suppose that $(Q_t)$ is Feller. Let $z^* \in M$ be a $(Q_t)$-accessible Doeblin point and assume that $(Q_t)$ admits a Lyapunov function $f$, i.e., $f:M \to \R$ is nonnegative, proper and $Q_t f(\cdot) \leq C e^{-\alpha t} f(\cdot) + C' \ \forall t \geq 0$ for some $C,\alpha, C' > 0$. Then there exist a unique invariant probability measure $\pi$ and $c, \gamma > 0$ such that, for all $t \geq 0$ and $z \in M$,
\[
\| Q_t(z, \cdot)  - \pi \|_{TV} \leq c (1+f(z)) e^{- \gamma t}, 
\]
\end{enumerate} 
where $\| \mu - \nu \|_{TV}$ denotes the total variation distance between two probability measures $\mu, \nu$ on $(M, \Bc(M))$.  
\end{theorem} 

\begin{remark}    \rm 
In the setting of part (1), since $z^*$ is $(Q_t)$-accessible, the point $z^*$ is contained in $\supp R(z, \cdot)$ for every $z \in M$ (see Remark~\ref{rm:support_res}). This fact and the assumption that $U$ is a petite open neighbourhood of $z^*$ can be linked by a simple argument to prove that $\nu$, the minorizing measure from Definition~\ref{def:resolvent}, is absolutely continuous with respect to every invariant probability measure for $R$ and hence also with respect to every invariant probability measure for $(Q_t)$. As shown for instance in the proof of Proposition 6.1.9 in \cite{Duf97}, this implies that $(Q_t)$ has at most one invariant probability measure.  

As far as part (2) is concerned, if, for a Feller semigroup, $z^*$ is $(Q_t)$-accessible and a Doeblin point with respect to $Q_{t^*}$ for some $t^* > 0$, one can show the existence of $\tilde z \in M$ and $\tilde t \geq t^*$ such that, for every $t \geq \tilde t$, $\tilde z$ is $(Q_t)$-accessible and a Doeblin point with respect to $Q_t$. Let $\hat t \geq \tilde t$ be so large that $C e^{-\alpha \hat t} < 1$, where $C$ and $\alpha$ are the constants pertaining to the Lyapunov function $f$. Fix $R > 2 C'/(1- C\exp(-\alpha \hat t))$ and let $\M := \{z \in M: f(z) \leq R\}$. Since $f$ is proper, $\M$ is compact. Then one can show that there exists a positive integer $m$ such that $\M$ is a small set with respect to $Q_{m \hat t}$. By Harris's ergodic theorem (see, e.g., Theorem 1.2 in \cite{HM08}), $Q_{m \hat t}$ admits a unique invariant probability measure $\pi$, which is then also the unique invariant probability measure for $(Q_t)$. Furthermore, there exist $c_0, \gamma_0 > 0$ such that, for every bounded measurable function $g$ on $M$ and for every positive integer $n$, 
$$
\sup_{z \in M} \frac{\lvert Q_{mn \hat t} g(z) - \pi g \rvert}{1 + f(z)} \leq c_0 e^{-\gamma_0 n} \sup_{z \in M} \frac{\lvert g(z) - \pi g \rvert}{1+f(z)}. 
$$
From this estimate, one infers 
$$
\| Q_t(z, \cdot) - \pi \|_{TV} \leq c (1+f(z)) e^{-\gamma t}, \quad \forall t \geq 0, \ z \in M,    
$$
where 
$$
c := 2c_0 e^{\gamma_0}, \quad \gamma := \frac{\gamma_0}{m \hat t}. 
$$
\end{remark} 

\bigskip 

\subsection{Bracket conditions; some motivating examples}     \label{ssec:counterexamples} 

%From now on, we will write $(P_t)$-accessible in lieu of $(Q_t)$-accessible.
 In light of Theorem~\ref{thm:uniqueness}, a first goal is to provide conditions ensuring the existence of petite or small sets. For exponential switching times, the existence of such sets is granted by Lie bracket conditions, as proved in \cite{BH12} and \cite{BMZIHP}. After recalling these conditions, we give several examples showing they are not sufficient in the non-exponential general setting.
%  as well as to characterize the set of $(P_t)$-accessible points.
% As announced, we shall use the strategy developed in \cite{BH12} and \cite{BMZIHP}, which relies on ., these conditions imply the existence of petite or small sets. 
%After recalling precisely these conditions, we give several examples enlightening why they might not be sufficient in the non-exponential setting.

Recall that the \emph{Lie bracket} of two smooth vector fields $V$ and $W$ on $\R^d$ with differentials $DV$ and $DW$ is the vector field $[V,W]$ defined by 
$$
[V,W](x) := DW(x) V(x) - DV(x) W(x), \quad x \in \R^d. 
$$

\begin{definition}
We recursively define the families of vector fields 
$$
\F_0 := \left\{F^i: i \in E \right\}, \quad \F_{k+1} := \F_k \cup \left\{[F^i, V]: i \in E, V \in \F_k \right\}, \ k \in \N_0, 
$$
and
$$
\G_0 := \left\{G^i - G^j: i, j \in E \right\}, \quad \G_{k+1} := \G_k \cup \left\{[G^i, V]: i \in E, V \in \G_k \right\}, \ k \in \N_0. 
$$
We say that the \emph{weak} or \emph{strong bracket condition} holds at $x \in \R^d$ if $\R^d$ is generated by the set of vectors 
$$
\biggl\{V(x): V \in \bigcup_{k \in \N_0} \F_k \biggr\}
$$ 
or
$$
\biggl\{V(x): V \in \bigcup_{k \in \N_0} \G_k \biggr\}, 
$$ 
respectively.
We say that the weak or strong bracket condition holds at a point $(x,s,i) \in K$ if the weak or strong bracket condition holds at $x$.
\end{definition}
One of the main results in \cite{BH12} and \cite{BMZIHP} (see also \cite{BHS18}, based on \cite{li17}) states that if the weak or strong bracket condition holds at a point $x$, then $(P_t)_{t \geq 0}$ admits a petite or a small set, respectively. The proofs rely on the fact that the strong bracket condition at $x \in \R^d$ implies that for some $m \geq d$, $T > 0$ and $(i_1, \ldots, i_m, i_{m+1}) \in E^{m+1}$, the map $\Psi: (v_1, \ldots, v_m) \mapsto \varphi^{i_{m+1}}_{T - (v_1 + \ldots + v_m)} \circ \varphi^{i_m}_{v_m} \circ \ldots \circ \varphi^{i_1}_{v_1}(x)$ is a submersion at some point $(t_1, \ldots, t_m)$ (i.e., the matrix $D \Psi(t_1, \ldots, t_m)$ has full rank $d$). Then, for exponential switching times, the process $X_t$ has a positive probability to ``follow'' the submersion $\Psi$, i.e., to follow the vector fields indexed by $i_1, \ldots, i_{m+1}$ for times close to $t_1, \ldots, t_m, T - (t_1 + \ldots + t_m)$. This implies that the law of $X_t$ has a nonzero absolutely continuous component with respect to the Lebesgue measure on $\R^d$, which yields the existence of a small set. The reasoning above is made possible by the fact that the exponential law has a density which is bounded away from 0 on every compact interval. In particular, the strong bracket condition being a local condition, the point $(t_1, \ldots, t_m)$ at which $\Psi$ is a submersion consists in general of small times $t_i$, which means that switches between the flows have to be fast in order for $X_T$ to be close to $\Psi(t_1, \ldots, t_m)$. This is why, if $0$ is not included in the support of the law of switching times, the strong bracket condition at a point $x^*$ may fail to result in a nontrivial absolutely continuous component for $X_T$ and thus $(x^*, s, i)$ may not be a Doeblin point. This is illustrated in the following example. 

\begin{example} \label{ex:non-analytic}  \rm 
On $\R$, consider random switching between the vector fields $F^0 \equiv 1$ and 
$$
F^1(x) := \begin{cases}
                1 + e^{-1/x^2}, & \quad x < 0, \\
                1, & \quad x \geq 0. 
                \end{cases}
$$
Since $F^0$ and $F^1$ are bounded and $C^{\infty}$ smooth, they are also globally integrable. Let 
$$
Q \equiv \begin{pmatrix}
                0 & 1 \\
                1 & 0
                \end{pmatrix}, 
$$
let $\lambda^i \equiv 1$ for $i \in \{0,1\}$, and let $\mu^0$ and $\mu^1$ be probability distributions whose support is contained in $[1,\infty)$. For $x^* = -\tfrac{1}{2}$, we have 
$$
F^1(x^*) - F^0(x^*) = e^{-4} > 0,
$$
so the strong bracket condition holds at $x^*$. Let $i \in \{0,1\}$, let $U \subset K$ be a neighbourhood of $z_1 := (x^*,0,i)$ and let $t \geq 0$. Pick $z_2 \in U$ such that $z_2 = (x^* + \eps,0,i)$ for some $\eps > 0$. For $j \in \{1,2\}$, let $Z^j$ denote the process $Z$ starting at $z_j$. Since the second component of $Z^j$ does not switch before the first component has reached $\R_+$ and since $F^0$ and $F^1$ coincide on $\R_+$, the evolution of the first component of $Z^j$ is entirely deterministic. As a result, for every $t > 0$, the probability measures $P_t(z_1,\cdot)$ and $P_t(z_2,\cdot)$ are mutually singular, which implies that there is no nontrivial measure $\nu$ minorizing both $P_t(z_1, \cdot)$ and $P_t(z_2, \cdot)$. Thus, $(x^*,0,i)$ is not a Doeblin point.  
\end{example}

Another reason why the law of $X_T$ and Lebesgue measure may be mutually singular despite the strong bracket condition is insufficient regularity for the law of the switching times. This can be seen in the following example.   

\begin{example} \rm    \label{ex:rational_supp} 
On $\R$, let $F^+ \equiv 1$ and $F^- \equiv - 1$, $\mu^+ = \mu^- = \sum_{i=1}^{\infty} 2^{-i} \delta_{q_i}$, where $(q_i)_{i \geq 1}$ is an enumeration of the rational numbers in $(0,1]$, and $\lambda^+ = \lambda^- \equiv 1$. Note that $0$ belongs to the support of $\mu^+ = \mu^-$. Obviously, the strong bracket condition holds everywhere. However, no point in $K = \R \times [0,1) \times \{+,-\} $ is a Doeblin point. Indeed, let $z =(x,0,i) \in K$ and let $U$ be a neighbourhood of $z$ in $K$. Let $z' = (y,0,i) \in U$. We let $Z^z$ and $Z^{z'}$ be the switching processes starting from $z$ and $z'$, respectively. Then, for all $t \geq 0$, one has $X^z_t \in (x+t + \mathbb{Q}) \cup (x-t + \mathbb{Q})$ and $X^{z'}_t \in (y + t + \mathbb{Q}) \cup (y-t+\mathbb{Q})$. If we choose $y$ in such a way that $x - y \notin \mathbb{Q} \cup (2t + \mathbb{Q}) \cup (-2t + \mathbb{Q})$, the laws of  $X_t^z$ and $X_t^{z'}$ are mutually singular and $z$ is not a Doeblin point.
\end{example}

In order to circumvent the issues raised by Examples~\ref{ex:non-analytic} and~\ref{ex:rational_supp}, we put forward two ideas. The first one is to impose that $\mu^i$ has  a density which is bounded away from $0$ in a neighbourhood of $0$.  
The second one is to impose analyticity on the vector fields $F^i$. Analyticity allows us to ``transform'' the submersion with small switching times - whose existence is guaranteed by the bracket condition - into a submersion with possibly larger switching times and thus in agreement with the law of the jump times (note that in Example \ref{ex:non-analytic}, $F^1$ is not analytic). 
However, even if $(F^i)$ are analytic, $0$ not being in the support of the $(\mu^i)$ can lead to a situation where there are no Doeblin points in spite of the strong bracket condition being satisfied. This may happen if the jump process $I$ is not allowed to switch to certain vectors fields at certain times due to the structure of $Q$, even if $Q(x)$ is irreducible for every $x \in \R^d$. The kind of problems that can arise are developed in the next example.

\begin{example} \label{ex:non-irreducible} \rm
On $\R^2$, consider the vector fields 
\begin{align*} 
F^1(x,y) =& (-x,0), & F^2(x,y) =& (-(x-1), 0), \\
F^3(x,y) =& (0, -y), & F^4(x,y) =& (0, -(y-1)). 
\end{align*} 
Note that the compact set $M=[0,1]^2$ is positively invariant under all the flows $(\varphi^i)$ and that the strong bracket condition holds everywhere. Let $\mu^1=\mu^2=\mu^3=\mu^4$ be a law on $\mathbb{R}_+$ whose support is included in $[\log(4), + \infty)$ and set $\lambda^i \equiv 1$ for every $i \in E$. Finally, consider a family of $4 \times 4$ matrices $(Q(x,y))_{(x,y)\in M}$ with entries denoted by $q_{i,j}(x,y)$ such that
\begin{enumerate}
\item For all $(x,y) \in M$, $Q(x,y)$ is irreducible;
\item If $x \leq 1/4$, $q_{1,3}(x,y) = q_{1,4}(x,y) = 0$, while if $x \geq 3/4$, $q_{2,3}(x,y) = q_{2,4}(x,y) = 0$;
\item If $y \leq 1/4$, $q_{3,1}(x,y) = q_{3,2}(x,y) = 0$, while if $y \geq 3/4$, $q_{4,1}(x,y) = q_{4,2}(x,y) = 0$.
\end{enumerate}
Note that if $x \leq 1/4$ (resp. $x \geq 3/4$) it is not possible to switch directly from $1$ (resp. $2$) to $\{3,4\}$, and if $y \leq 1/4$ (resp. $y \geq 3/4$) it is not possible to switch directly from $3$ (resp. $4$) to $\{1,2\}$.

Let us show that no point $(x,y,s,i) \in K_M$  is a Doeblin point, where $K_M$ is the set introduced in Section~\ref{ssec:existence_im_comp}. We claim that when $I_0 \in \{1,2\}$, then $I_t$ stays in $\{1,2\}$ for all $t \geq 0$, and if $I_0 \in \{3,4\}$, $I_t$ remains in $\{3,4\}$. This means in particular that if $I_0 \in \{1,2\}$, the vector fields $F^3$ and $F^4$ are not used, and thus the process $(X_t, Y_t)_{t \geq 0}$ cannot move vertically. Granting this claim, fix $z=(x,y,s,i) \in K_M$ with $i \in \{1,2\}$, let $U$ be a neighbourhood of $z$ and let $z'=(x,y',s,i) \in U$ for some $y' \neq y$. Then the process $Z^z$ stays almost surely in $[0,1] \times \{y\} \times \R_+ \times \{1,2\}$, while $Z^{z'}$ stays in  $[0,1] \times \{y'\} \times \R_+ \times \{1,2\}$. Therefore, for all $t \geq 0$, $P_t(z, \cdot)$ and $P_t(z',  \cdot)$ are mutually singular and $z$ is not a Doeblin point. A similar reasoning holds if $i \in \{3,4\}$. 

It remains to prove the claim. Let $Z_0 =(x,y,s,i) \in K_M$ with $i = 1$. Let $S_1$ be the time of the first switch and let us prove that $X_{S_1} \leq 1/4$. Since $\varphi_t^1(x,y) = (xe^{-t},y)$, this is obvious if $x \leq 1/4$. Assume now that $x > 1/4$. Since $z \in K_M$, one has $s \leq \log(1/x) < \log (4)$. Moreover, due to the fact that $\mu^1$ has its support included in $[\log(4), +\infty)$, we have almost surely $S_1 \geq \log(4) - s$. Hence, $X_{S_1} = xe^{-S_1} \leq 1/4$. 
Note also that $Y_t = y$ for all $t \leq S_1$. At time $S_1$, a switch occurs and, due to item (2) above, $I_{S_1} = 2$ almost surely. Now, $S_2 \geq \log(4)$ almost surely and therefore, 
\[
X_{S_1+S_2} = 1 + (X_{S_1} - 1)e^{-S_2} \geq 1 - 1/4 = 3/4.
\]
A second switch occurs at time $S_1 + S_2$ and item (2) implies that $I$ switches from $2$ to $1$. Alternately repeating theses two steps, one sees that $I_t \in \{1,2\}$ for all $t \geq 0$, which proves the claim for $i=1$.
The proof for $i = 2$ is analogous, as is the proof for $I_0 \in \{3,4\}$. 
\end{example}

In the previous example, the issue is that, even though $Q(x,y)$ is irreducible for all $(x,y)$, the structure of irreducibility changes with $(x,y)$: if $X_t \leq 1/4$ and $I_t = 1$, the process $I$ cannot jump directly to $\{3,4\}$; it has to switch first to $2$. But then, $I$ has to remain equal to $2$ for a time at least $\log(4)$ during which $X_t$ becomes larger than $3/4$. Here, the structure of irreducibility has changed and it is no longer possible to switch from $2$ to $\{3,4\}$; therefore $I$ goes back to $1$ and so on. To exclude this phenomenon, we will  introduce in the next section a stronger notion of irreducibility (see Definiton \ref{def:globalirreducible}).

\bigskip 

\subsection{Unique ergodicity} 

The examples of the previous section show that it is necessary to add assumptions on the support of the $\mu^i$ (see Example~\ref{ex:non-analytic}), the regularity of $\mu^i$ (see Example~\ref{ex:rational_supp}) or the irreducibility structure of $Q$ (see Example~\ref{ex:non-irreducible}). In order to state these assumptions, we first give two definitions. The first one is motivated by Example~\ref{ex:rational_supp}. Recall that $\bm{\lambda}_m$ denotes the Lebesgue measure on $\R^m$.

\begin{definition}
Let $m \geq 1$ and let $\mu$ be a finite Borel measure on $\R^m$. Let $\mu^{\textrm{ac}}$ be the absolutely continuous component of $\mu$ with respect to $\bm{\lambda}_m$ and set $h := d \mu^{\textrm{ac}}/d\bm{\lambda}_m$. We say that a point $\tbf_0 \in \R^m$ is $\mu$-\emph{regular} if, on a neighbourhood of $\tbf_0$, $h$ is bounded below by a positive constant. 
\end{definition} 

Clearly, if $\tbf_0$ is $\mu$-regular, then $\tbf_0$ is an element of the support of $\mu$. It is also clear that being $\mu$-regular is stronger than simply being contained in the support of $\mu$ (e.g., if $\mu$ and $\bm{\lambda}_m$ are mutually singular, then there are no $\mu$-regular points). The following definition is motivated by  Example \ref{ex:non-irreducible}.

\begin{definition}
  \label{def:globalirreducible}
We say that $Q = (Q(x))_{x \in \R^d}$ is \emph{globally irreducible} if for all $i,j \in E$, there exist a positive integer $n$ and indices $i_0 = i, i_1, \ldots, i_{n-1}, i_n =j$ such that $q_{i_k, i_{k+1}}(x) > 0$ for all $x \in \R^d$ and $k=0, \ldots, n-1$.
\end{definition}

Obviously, $Q$ given in Example \ref{ex:non-irreducible} is not globally irreducible. The issues raised in the examples of the previous section justify the following two assumptions:

\begin{assumption}    \label{cond:regular_point} 
There exists $\varepsilon_0 > 0$ such that for all $i \in E$ and all $t \in (0, \varepsilon_0)$, $t$ is $\mu^i$-regular. 
\end{assumption} 

\begin{assumption}    \label{cond:analyticity} 
$Q$ is globally irreducible and for every $i \in E$, the vector field $F^i$ is analytic and there exists a $\mu^i$-regular point $t^r_i > 0$. 
\end{assumption}

We are now ready to provide sufficient conditions for the existence of petite and small sets in the context of semi-Markov switching. By virtue of Theorem~\ref{thm:uniqueness}, these can be linked to the uniqueness of the invariant probability measure as well as to exponential ergodicity, which is done in Corollary~\ref{cor:uniqueandergodic}. Recall our standing assumption that $Q(x)$ is irreducible for every $x \in \R^d$. We let $R$ denote the resolvent associated with $(P_t)$ (see Definition~\ref{def:resolvent}).

\begin{theorem}   \label{thm:Doeblin}
Assume that at least one of the Assumptions~\ref{cond:regular_point} and~\ref{cond:analyticity} holds. 
Then
\begin{enumerate}
\item If the weak bracket condition holds at a point $x^* \in \R^d$, then for every $i \in E$ there exist a neighbourhood $U \subset K$ of $(x^*,0,i)$, a constant $c > 0$, $i' \in E$ and a nonempty open set $V \subset K^{i'}$ such that, for every $z' \in U$, 
\[
R(z', A \times \{i'\}) \geq c \ \bm{\lambda}_{d+1}(A \cap V), \quad \forall A \in \Bc(K^{i'}). 
\]
In particular, $U$ is a petite set.
\item If the strong bracket condition holds at a point $x^* \in \R^d$,  then for every $i \in E$ there exist a neighbourhood $U \subset K$ of $z = (x^*,0,i)$, a constant $c > 0$, $i' \in E$, $T > 0$ and a  nonempty open set $V \subset K^{i'}$ such that, for every $z' \in U$, 
\[
P_{T} (z', A \times \{i'\}) \geq c \ \bm{\lambda}_{d+1}(A \cap V), \quad \forall A \in \Bc(K^{i'}). 
\]
In particular, $z$ is a Doeblin point and $U$ is a small set.
\end{enumerate} 
\end{theorem} 

The proof of Theorem~\ref{thm:Doeblin} is given in Section~\ref{ssec:proof_Doeblin}. 
Thanks to Theorem \ref{thm:uniqueness} and Remark~\ref{rm:support_res}, we can deduce from it the following corollary. 

\begin{corollary}
\label{cor:uniqueandergodic}
Assume that at least one of the Assumptions~\ref{cond:regular_point} and~\ref{cond:analyticity} holds. Then
\begin{enumerate}
\item If the weak bracket condition holds at a $(P_t)$-accessible point $(x^*, 0, i^*) \in K$, then $Z$ admits at most one invariant probability measure;
\item Under Assumption~\ref{hyp:compact} from Section~\ref{ssec:existence_im_comp}, suppose that the strong bracket condition holds at a $(P_t)$-accessible point $(x^*, 0, i^*) \in K_M$. Here, $K_M$ replaces $K$ as the state space for $Z$, as explained in Section~\ref{ssec:existence_im_comp}; the notion of $(P_t)$-accessibility is then also defined with respect to $K_M$. Assume further that conditions (1) and (2) of Proposition \ref{prop:lyap} hold. Then $Z$ admits a unique invariant probability measure $\pi$ on $K_M$ and there exist $\gamma, c > 0$ such that for all $t \geq 0$ and $z \in K_M$,
\[
\| P_t(z, \cdot)  - \pi \|_{TV} \leq c (1+f(z)) e^{- \gamma t},
\]
where $f$ is the Lyapunov function constructed in Proposition \ref{prop:lyap}.
\end{enumerate}
\end{corollary}

\begin{remark}   \rm 
    Assume that the vector fields $F^i$ are analytic and assume that 
    \begin{enumerate}
        \item[\textbf{(i)}] There exist a point $x^* \in \mathbb{R}^d$ and $\alpha^1, \ldots, \alpha^N \in \mathbb{R}$ with $\sum_i \alpha^i = 1$ such that $\sum_i \alpha^i F^i(x^*) = 0$;
        \item[\textbf{(ii)}] There exists a point $x \in \R^d$ which is $\{F^i\}$-accessible from $x^*$ (see Definition~\ref{def:attainable} in Section~\ref{ssec:accessibility}) and at which the weak bracket condition holds.
    \end{enumerate}
    Then Proposition 2.11 in \cite{BHS18} implies that the strong bracket condition holds at $x^*$. In particular, under Assumption~\ref{cond:analyticity}, point 2 of Theorem~\ref{thm:Doeblin} and point 2 of Corollary~\ref{cor:uniqueandergodic} remain valid if the strong bracket condition at $x^*$ is replaced with conditions \textbf{(i)} and \textbf{(ii)}. In the case where the $F^i$ are not analytic, conditions \textbf{(i)} and \textbf{(ii)} imply that for some $m \geq d$, $T > 0$ and $(i_1, \ldots, i_{m+1}) \in E^{m+1}$, the map $\Psi: (v_1, \ldots, v_m) \mapsto \varphi_{T - (v_1 + \ldots + v_m)}^{i_{m+1}} \circ \varphi^{i_m}_{v_m} \circ \ldots \circ \varphi^{i_1}_{v_1}(x^*)$ is a submersion at some point $(t_1, \ldots, t_m)$ (see Proposition~2.9 in~\cite{BHS18}). Under Assumption~\ref{cond:regular_point} and reasoning as in the proof of Theorem~\ref{thm:Doeblin}, one can also show that point 2 of Theorem~\ref{thm:Doeblin} and point 2 of Corollary~\ref{cor:uniqueandergodic} remain valid if the strong bracket condition at $x^*$ is replaced with conditions \textbf{(i)} and \textbf{(ii)}.
\end{remark}

\subsection{Absolute continuity} 
In \cite{BH12} and \cite{BMZIHP}, it is proven that, for exponential switching times, the weak bracket condition at a $(P_t)$-accessible point yields absolute continuity of the invariant probability measure -- provided that it exists -- with respect to the product of the Lebesgue measure on $\R^d$ and counting measure on $E$. In this section, we give conditions under which this result can be extended to our more general setting. 

\begin{proposition}       \label{thm:ac_component} 
Assume that at least one of the Assumptions~\ref{cond:regular_point} and~\ref{cond:analyticity} holds. Assume further that the weak bracket condition is satisfied at a $(P_t)$-accessible point $(x^*, 0, i^*) \in K$. If $(P_t)$ has an invariant probability measure, then it is unique and has a nontrivial absolutely continuous component with respect to $\mathbf{L}$, the product of ${\bm \lambda}_{d+1}$ and counting measure on $E$, restricted to $K$.
\end{proposition}

\bpf By point 1 of Theorem~\ref{thm:Doeblin}, there exist a neighbourhood $U \subset K$ of $(x^*, 0, i^*)$, a constant $c > 0$, $i' \in E$ and a nonempty open set $V \subset \bigcup_{k \in E} K^k$ such that 
$$
R(z, \cdot \times \{j\}) \geq c \ \bm{\lambda}_{d+1}(\cdot \cap V) \delta_{i'}(j), \quad \forall z \in U, \ j \in E. 
$$
Let $A \in \Bc(K)$ such that $\bm{L}(A) = 0$ and let $(A_j)_{j \in E}$ be a family of Borel subsets of $\R^d \times \R_+$ such that 
$$
A = \bigcup_{j \in E} A_j \times \{j\}. 
$$
Suppose that $(P_t)$ admits an invariant probability measure $\pi$, which is then also invariant with respect to $R$. Let $A^c = K \setminus A$ and let $A_j^c = K^j \setminus A_j$, $j \in E$. In order to establish that $\pi$ has a nontrivial absolutely continuous component with respect to $\bm{L}$, we need to show that $\pi(A^c) > 0$. One has  
\begin{align*}
\pi(A^c) =& \pi R(A^c) = \int_K R(z, A^c) \ \pi(dz) \geq \int_U R(z,A^c) \ \pi(dz) \\
=& \sum_{j \in E} \int_U R(z, A_j^c \times \{j\}) \ \pi(dz) \geq c \pi(U) \bm{\lambda}_{d+1}(A_{i'}^c \cap V).
\end{align*}
As will be shown in Section~\ref{ssec:proof_Doeblin}, there exists $j^* \in E$ such that $V \subset K_{j^*}$ (see Proposition~\ref{prop:weaksubmersion}). Together with $\bm{L}(A) = 0$ and the fact that $V$ is a nonempty open set, this yields 
$$ 
\sum_{j \in E} \bm{\lambda}_{d+1}(A^c_j \cap V) \geq \bm{\lambda}_{d+1}(A^c_{j^*} \cap V) > 0. 
$$ 
Finally,  
$$
\pi(U) = \int_K R(z,U) \ \pi(dz) > 0 
$$
because the interior of $U$ contains a $(P_t)$-accessible point (see Remark~\ref{rm:support_res}). 
\epf 

\bigskip 

Observe that Proposition~\ref{thm:ac_component} only guarantees an absolutely continuous component for the invariant probability measure, not absolute continuity itself. In order to obtain the latter, we make an additional assumption. 

\begin{assumption}
\label{hyp:barGC1}
For all $i \in E$, there exist sequences $(a_n^i)_{n \in N^i}$ and $(b_n^i)_{n \in N^i}$ with $N^i = \{1, \ldots, n^i\}$ or $N^i = \mathbb{N}^*$, $a_n^i < b_n^i < a_{n+1}^i$, and such that $\mathrm{supp}(\mu^i) = \cup_{n \in N^i} [a_n^i, b_n^i]$. Moreover, for all $n \in N^i$, there exists a continuous function $  g^i_n : [a_n^i, b_n^i] \to [0,1]$ which, on $(a_n^i, b_n^i)$,  is $C^1$ with strictly negative derivative, and such that $  G^i$ restricted to $[a_n^i, b_n^i]$ is equal to $  g_n^i$. 
\end{assumption}

\begin{example}
Assume that $\mu = \frac{1}{2} \mathrm{U}_1 + \frac{1}{2} \mathrm{U}_2$, where $\mathrm{U}_1$ is the uniform law on $[1,2]$ and  $\mathrm{U}_2$ is the uniform law on $[3,4]$. Then, setting $N = \{1,2\}$,  $a_1 = 1$, $a_2 = 3$, $b_1 = 2$ and $b_2 = 4$, one has $\mathrm{supp}(\mu) = \cup_{n \in N} [a_n, b_n]$. Moreover, $  G$ is $C^1$, and equal on $(1,2)$ to $t \mapsto   g_1(t) = \frac{ 3 - t}{2}$ and equal on $(3,4)$ to $t \mapsto g_2(t) = \frac{4-t}{2}$. 
\end{example}

The proof of the following proposition is postponed to Section~\ref{ssec:proof_ac}. 

\begin{proposition}      \label{prop:ac_preserve} 
Grant Assumption \ref{hyp:barGC1} and assume that $(\lambda_i)_{i \in E}$ are $C^1$ on $\R^d$. Then, if $\mu$ is a finite measure on $(K, \Bc(K))$ such that $\mu \ll \bm{L}$, we also have $\mu P_t \ll \bm{L}$ for every $t > 0$.   
\end{proposition} 

This leads to the following result on absolute continuity for the invariant probability measure.

\begin{theorem}    \label{thm:result_ac} 
Let $(\lambda_i)_{i \in E}$ be $C^1$ on $\R^d$ and let $(x^*, 0, i^*) \in K$ be a $(P_t)$-accessible point at which the weak bracket condition holds. In addition, assume that at least one of the following conditions is satisfied: 
\begin{enumerate}
\item Assumption~\ref{hyp:barGC1} holds with $a^i_1 = 0$ for every $i \in E$. 
\item Assumptions~\ref{cond:analyticity} and~\ref{hyp:barGC1} hold. 
\end{enumerate}
Then, if $(P_t)$ has an invariant probability measure, it is unique and absolutely continuous with respect to $\mathbf{L}$.
\end{theorem}

\bpf By Proposition~\ref{prop:ac_preserve}, we have $\mu P_t \ll \mathbf{L}$ for every finite measure $\mu$ on $(K, \Bc(K))$ with $\mu \ll \mathbf{L}$ and for every $t > 0$. Let $\pi$ be an ergodic measure with respect to $(P_t)$. A well-known result from the ergodic theory of Markov chains on metric spaces (see, e.g., Lemma 1.1 in~\cite{CKW21}) then implies that $\pi$ is either absolutely continuous or singular with respect to $\mathbf{L}$. Note that under condition (1) of Theorem~\ref{thm:result_ac}, Assumption~\ref{cond:regular_point} holds as well.  
Then, by Proposition~\ref{thm:ac_component}, $\pi$ is unique and has a nontrivial absolutely continuous component, hence it must be absolutely continuous.  
\epf 

\begin{remark}
In case where the jump rates $\lambda^i$ are constant, our process can be viewed as a special case of the family of Piecewise Deterministic Process as studied in \cite{CKW21}: the process $(X_t, \tau_t)$ follows a semiflow between the switching times, and if a switch occurs when $(X, \tau) = (x,s)$, then the process jumps to $(x,0)$. Hence, with the notation of \cite{CKW21}, the transformation at jump times is deterministic and equal to $w : (x,s) \mapsto (x,0)$. However, this transformation is not non-singular with respect to ${\bm \lambda}_{d+1}$, as required in the main results of \cite{CKW21} on absolute continuity of the invariant measures.
\end{remark}

\subsection{Accessibility}     \label{ssec:accessibility} 

Corollary \ref{cor:uniqueandergodic} underscores the importance of $(P_t)$-accessible points in the ergodic theory of switching processes. In particular, it is crucial to be able to check whether a given point is $(P_t)$-accessible. In this section, we characterise the set of $(P_t)$-accessible points in terms of the building blocks of $Z$, i.e., in terms of the vector fields $(F^i)$, the measures $(\mu_i)$, and the family of transition matrices $Q$.

\subsubsection{When $0$ is in the support of the $\mu^i$: $\{F^i\}$-accessible points}

For PDMPs with exponential switching times and irreducible transition matrices, the set of $(P_t)$-accessible points has been described in \cite{BMZIHP} (see also \cite{BCL17}). Roughly speaking, for the semigroup $(P_t)$ of a PDMP on $\R^d \times E$, a point $z^* = (x^*, i^*)$ is $(P_t)$-accessible from $z = (x,i)$ if and only if there exists, for every neighbourhood $U$ of $x^*$, a deterministic path, obtained by switching between the flows, that goes from $x$ to $U$. 

More precisely, for $\ibf = (i_1, \ldots, i_m) \in E^m$ and $\sbf = (s_1, \ldots, s_m) \in \R^m_+$, we call the pair $(\sbf, \ibf)$ a \emph{control sequence} and introduce the composite flow 
$$
\Phi_{\sbf}^{\ibf} = \varphi_{s_m}^{i_m} \circ \ldots \circ \varphi_{s_1}^{i_1}.
$$
We then make the following definition.

\begin{definition}           \label{def:attainable} 
A point $x^* \in \R^d$ is called \emph{attainable} from $x \in \R^d$ if there is a  control sequence $(\sbf, \ibf)$ such that  
$x^* = \Phi_{\sbf}^{\ibf}(x)$.
We denote the set of points $x^* \in \R^d$ that are attainable from $x$ by $\mathcal{O}^+(x)$. A point $x^* \in  \R^d$ is called \emph{$\{F^i\}$-accessible} from $x \in  \R^d$ if $x^*$ is contained in the  closure of $\mathcal{O}^+(x)$.
\end{definition}

It is proven in \cite[Lemma 3.2]{BMZIHP} that for PDMPs with exponential switching times and irreducible transition matrices, a point $(x^*,i^*)$ is $(P_t)$-accessible from $(x,i)$ if and only if $x^*$ is $\{F^i\}$-accessible from $x$. In our setting, the situation is not as clear-cut. In Example \ref{ex:non-irreducible}, every $(x',y') \in [0,1]^2$ is $\{F^i\}$-accessible from every $(x,y) \in [0,1]^2$, but as we proved, starting from $(x,y,s,i) \in K_M$ such that $i \in \{1,2\}$, every point of the form $(x',y',s',i')$ with $y ' \neq y$ is not $(P_t)$-accessible. The next example is similar, but this time the family of matrices $Q$ is globally irreducible. 

\begin{example} \label{ex:F^iacc-nonPt-acc}   \rm 
For $\alpha, \beta > 0$, consider the vector fields $F^1(x) = -\alpha x$ and $F^2(x) = \beta x$ on $\R^d$. Clearly, the origin is a globally asymptotically stable equilibrium point for $F^1$ and is a repulsive equilibrium for $F^2$. In particular, $0$ is $\{F^i\}$-accessible. For simplicity, let us assume that $\lambda_1$ and $\lambda_2$ are constantly equal to $1$, so $\mu_{x,i} = \mu_i$ for every $x \in \R^d, i \in\{1,2\}$. Suppose there are positive real numbers $r_1$ and $r_2$ with $r_1/r_2 \leq \beta / \alpha$ such that the support of $\mu_1$ is contained in $(0, r_1]$ and the support of $\mu_2$ is contained in $[r_2, \infty)$. Then, for every $(x,s,i) \in K$ such that $x \neq 0$ and for every $(s^*, i^*)$ such that $(0,s^*, i^*) \in K$, the point $(0,s^*,i^*)$ is not  $(P_t)$-accessible from $(x,s,i)$. Indeed, fix $(x,s,i) \in K$ such that $x \neq 0$. We claim that, for all $t \geq 0$, $\|X_t \| \geq e^{ - 3 \beta r_2} \|x\|$ almost surely. Hence, with $Z$ starting from $(x,s,i)$ and for $t > 0$ fixed, $X_t$ almost surely lies outside of the open ball centred at the origin with radius $e^{ - 3 \beta r_2} \|x\|$, which means that $(0,s^*,i^*)$ is not $(P_t)$-accessible from $(x,s,i)$. We prove the claim. For all $t$, there exists a (random) control sequence $(\sbf, \ibf)=(s_1, \ldots, s_m, i_1, \ldots, i_m)$ with $s_1 + \ldots + s_m = t$ and with the components of $\ibf$ alternating between $1$ and $2$, such that $X_t = \Phi_{\sbf}^{\ibf}(x)$. For $j \in \{1,2\}$, let $\mathcal{K}_j$ be the set of $k \geq 1$ such that the index $i_k$ in $\ibf$ equals $j$. Let $m_j$ denote the cardinality of $\mathcal{K}_j$. By construction and the choice of supports for $\mu_1$ and $\mu_2$, for every $k \in \mathcal{K}_1$, we have almost surely $s_k \leq r_1$, while if $k \in \mathcal{K}_2 \setminus \{1,m\}$, $s_k \geq r_2$ almost surely. Moreover,  since $1$ and $2$ alternate in $\ibf$, we have $\lvert m_1 - m_2 \rvert \leq 1$. As a result, 
\begin{align*}
\|X_t\|  =& \|x\|  \exp \biggl(\beta \sum_{k \in \mathcal{K}_2} s_k - \alpha \sum_{k \in \mathcal{K}_1} s_k \biggr) \\
\geq& \|x\| \exp \left(\beta (m_2 - 2) r_2 - \alpha m_1 r_1 \right) \\
\geq& e^{- 3 \beta r_2} \|x\| \exp(m_1 (\beta r_2 - \alpha r_1)) \geq e^{-3 \beta r_2} \|x\|.  
\end{align*}
Therefore, in this example, the origin is $\{F^i\}$-accessible, but $(0,s^*, i^*)$ is not $(P_t)$-accessible.
\end{example} 

The previous example shows that the existence of a deterministic path from a point $x$ to a point $x^*$ does not imply that $(x^*, s^*, i^*)$ is $(P_t)$-accessible from $(x,s,i)$. However, we show that this implication does hold when $0$ is in the support of every $\mu_i$. Recall the assumption for this section that $Q(x)$ is irreducible for every $x \in \R^d$.  

\begin{proposition}
\label{prop:Fiaccessible}
Assume that $0$ is in the support of $\mu_i$ for every $i \in E$. Let $z=(x,0,i) \in K$, $x^* \in \R^d$, and $i^* \in E$. Then  $x^*$ is $\{F^i\}$-accessible from $x$ if and only if there exists $s^* > 0$ such that $(x^*,s^*,i^*)$ is $(P_t)$-accessible from $z$.
\end{proposition}

The proof of Proposition~\ref{prop:Fiaccessible} can be found in Section~\ref{sec:proofaccess}.

\subsubsection{The general case: $\{F^i, \mu^i, Q\}$-accessible points}

We now characterize the set of accessible points in the general case, where $0$ is not necessarily in the support of every $\mu_i$. As exhibited in Example \ref{ex:F^iacc-nonPt-acc}, the presence of a deterministic path from a point $x$ to a point $x^*$ does not imply in general the accessibility of points of the form $(x^*,s^*,i^*)$ from points $(x,s,i)$. The reason is that the times of switching for the discrete component $I$ have to be in the supports of the respective measures $(\mu_i)$, and these supports are in general not equal to the whole $\R_+$. Moreover, if the deterministic path requires to switch from a state $i$ to a state $j$, this jump has to be allowed for the process $I$. This is why we introduce the concept of an admissible control sequence.

\begin{definition}        \label{def:admissible_cs} 
For $m \geq 1$, $\sbf = (s_1, \ldots, s_m) \in (0,\infty)^m$, and $\ibf = (i_1, \ldots, i_m) \in E^m$, we call the control sequence $(\sbf, \ibf)$ \emph{admissible} with respect to a starting point $(x,s,i) \in K$ if the following holds:
\begin{enumerate}
\item If $m=1$: 
\begin{enumerate}
\item We have $i_1 = i$; 
\item The intersection of $(s + s_1, \infty)$ and the support of $\mu_{\varphi^i_{-s}(x),i}$ is nonempty. 
\end{enumerate}
\item If $m \geq 2$: 
\begin{enumerate}
\item We have $i_1 = i$ and $s + s_1$ lies in the support of $\mu_{\varphi^i_{-s}(x), i}$; 
\item For $2 \leq k \leq m$, let  
\begin{equation}
\label{eq:xk}
x_k = \Phi_{(s_1, \ldots, s_{k-1})}^{(i_1, \ldots, i_{k-1})}(x) = \varphi_{s_{k-1}}^{i_{k-1}} \circ \ldots \circ \varphi_{s_1}^{i_1}(x). 
\end{equation}
For $2 \leq k \leq m$ we have $q_{i_{k-1} i_k}(x_k) > 0$, and for $2 \leq k \leq m-1$ the time $s_k$ lies in the support of $\mu_{x_k, i_k}$;
\item The intersection of $(s_m, \infty)$ and the support of $\mu_{x_m, i_m}$ is nonempty. 
\end{enumerate} 
\end{enumerate}
\end{definition}

\begin{remark}  \rm 
If $\mu_i$ is supported on $\R_+$ for every $i \in E$ (as is the case for classical PDMPs with exponential switching times), then a control sequence $(\sbf, \ibf)$ is admissible with respect to $(x,s,i) \in K$ if and only if $i_1 = i$ and $q_{i_{k-1} i_k}(x_k) > 0$ for $2 \leq k \leq m$. 
\end{remark}   

With the definition of an admissible control sequence in hand, we can now define the set of points that are \emph{reachable}, a notion which is stronger than merely being attainable (see Definition~\ref{def:attainable}). 

\begin{definition}
A point $(x^*, s^*, i^*) \in K$ is called \emph{reachable} from $(x,s,i) \in K$ if there is an admissible control sequence $(\sbf, \ibf)= (s_1, \ldots, s_m, i_1, \ldots, i_m)$ with respect to $(x,s,i)$ such that  
$$
i^* = i_m, \quad x^* = \Phi_{\sbf}^{\ibf}(x), \quad s^* = \begin{cases}
                                     s + s_1, & \quad m = 1, \\
                                     s_m, & \quad m > 1.
                                     \end{cases} 
$$
We denote the set of points $z^* \in K$ that are reachable from $z$ by $\gamma^+(z)$. A point $z^* \in K$ is called \emph{$\{F^i,\mu_i,Q \}$-accessible} from $z \in K$ if $z^*$ is contained in the relative closure of $\gamma^+(z)$ in $K$. It is called $\{F^i,\mu_i,Q\}$-accessible from a set $Z \subset K$ if it is $\{F^i,\mu_i,Q\}$-accessible from every point $z \in Z$. We simply call $z^*$ \emph{$\{F^i,\mu_i,Q\}$-accessible} if it is $\{F^i,\mu_i,Q\}$-accessible from $K$. 
\end{definition}

\begin{remark}
\label{rem:F-mu-Q-acess-implies-F-acces}
    Note that if $(x^*, s^*, i^*)$ is $\{F^i, \mu^i, Q\}$ - accessible from $(x,s,i)$, then $x^*$ is necessarily $\{F^i\}$ - accessible from $x$.
\end{remark}

The following proposition fully characterizes the set of $(P_t)$-accessible points by asserting the equivalence of $(P_t)$-accessibility and $\{F^i, \mu_i, Q\}$-accessibility. It is proved in Section~\ref{sec:proofaccess}. 

\begin{proposition}
\label{prop:acces}
Let $z, z^* \in K$. Then $z^*$ is $(P_t)$-accessible from $z$ if and only if it is $\{F^i,\mu_i,Q\}$-accessible from $z$. In particular, $z^*$ is $(P_t)$-accessible if and only if it is $\{F^i,\mu_i,Q\}$-accessible.  
\end{proposition}

\begin{remark}  \rm 
We will typically call $(P_t)$-accessible/$\{F^i,\mu_i,Q\}$-accessible points \emph{accessible} from now on. 
\end{remark}

\subsection{Accessibility and ergodicity for one-dimensional switched systems}

Let $F^0$ and $F^1$ be two $C^1$ vector fields on $\mathbb{R}$ such that
\begin{enumerate}
\item $F^0(0) = F^1(1) = 0$;
\item For all $(x,i) \in [0,1] \times E$, $DF^i(x) < 0$.
\end{enumerate}
Let $\mu^0$ and $\mu^1$ two probability measures on $\mathbb{R}_+$. Assume that $\lambda_0 = \lambda_1 \equiv 1$. We claim that there exists $x^* \in (0,1)$ such that $(x^*,0,0)$ is accessible.

\begin{remark}     \rm 
Of course, if for some $i \in \{0,1\}$, the support of $\mu_i$ is not bounded, then $(i,0,1-i)$ is accessible, because we are allowed to follow vector field $F^i$ for an arbitrarily long time, which brings the process as close as we want to $i$.  Conversely, if $0$ is included in the support of $\mu_i$, then $(1-i,0,i)$ is accessible, since in that case, we are allowed to follow vector field $F^i$  for an arbitrarily short time, and thus prevent the process from moving away from $1-i$. Therefore, the proof below is especially interesting in the case where the supports of $\mu_0$ and $\mu_1$ are included in some bounded interval, bounded away from $0$.
\end{remark}

First, note that conditions (1) and (2) imply that for all $x \in (0,1)$, $F^0(x) < 0$ and $F^1(x) > 0$; in particular, the set $[0,1]$ is positively invariant by $\varphi^0$ and $\varphi^1$. 

\begin{lemma}
\label{lem:fixedpoint}
Let $t_0, t_1 > 0$, and consider $\Psi = \varphi^0_{t_0} \circ \varphi^1_{t_1}$. Then, $\Psi$ has a unique fixed point $x^*$ and for all $x \in [0,1]$, the sequence $\Psi^n(x)$ converges to $x^*$.
\end{lemma}

\begin{proof}
Since $[0,1]$ is positively invariant by the flows, $\Psi( [0,1]) \subset [0,1]$. Hence it suffices to prove that $\Psi$ is contracting and the result will follow from Banach-Picard fixed point theorem. For all $x \in [0,1]$, we have
\begin{align*}
\Psi'(x)  & =  \partial_x \varphi^0_{t_0} \left( \varphi^1_{t_1}(x) \right) \partial_x  \varphi^1_{t_1} (x)\\
& = \exp\left( \int_0^{t_0} DF^0\left( \varphi_s^0( \varphi^1_{t_1}(x) )\right) ds \right) \exp\left( \int_0^{t_1} DF^1\left( \varphi_s^1((x) )\right) ds \right) \\
& \leq e^{ - m (t_0 + t_1)} < 1,
\end{align*}
where $m = - \max_{(x,i) \in [0,1] \times E} DF^i(x) > 0$ by (2).
\end{proof}

\begin{lemma}
\label{lem:existacces}
For $i \in E$, choose $t_i$ in the support of $\mu_i$. Let $x^*$ be the unique fixed point of $\Psi =  \varphi^0_{t_0} \circ \varphi^1_{t_1}$ whose existence is granted by Lemma \ref{lem:fixedpoint}. Then, the point $(x^*,0,0)$ is accessible.
\end{lemma}
\begin{proof}
Let $(x,s,i) \in K$. We  first prove that $(x^*,0,0)$  is accessible from $(x,s,i)$ when $s= 0$ and $i=1$. For $m \geq 1$, consider the control sequence $(\sbf, \ibf ) \in (0, \infty)^{2m+1} \times E^{2m+1}$ defined by
\[
\sbf =(t_1, t_{0}, t_1, \ldots, t_0, \varepsilon) , \quad \ibf = (1, 0, 1, \ldots, 0, 1),
\]
for some $\varepsilon < t_0$. It is easily seen that this sequence is admissible for $(x,0,1)$ since $t_0$ and $t_1$ lies respectively in the support of $\mu_0$ and $\mu_1$. Thus, for all $m \geq 1$, the point $(\varphi_{\varepsilon}^{0}(\Psi^m(x)), \varepsilon, 0)$ is reachable from $(x,0,1)$. When $\varepsilon \to 0$, $\varphi^0_{\varepsilon}$ converges uniformly to the Identity, and, by Lemma \ref{lem:fixedpoint}, $\Psi^m(x) \to x^*$ as $m \to \infty$. Hence, $(x^*, 0, 0)$ is accessible from $(x,0,1)$. 

Similarly, the control sequence $(\sbf, \ibf ) \in (0, \infty)^{2m+2} \times E^{2m+2}$
\[
\sbf =(t_0, t_1, t_{0}, t_1, \ldots, t_0, \varepsilon) , \quad \ibf = (0, 1, 0, 1, \ldots, 0, 1),
\]
is admissible from $(x,0,0)$ and thus the point $(\varphi_{\varepsilon}^{0}(\Psi^m(\varphi_{t_0}(x))), \varepsilon, 0)$ is reachable from $(x,0,0)$, from which we conclude that $(x^*, 0, 0)$ is accessible from $(x,0,0)$.

Finally, we treat the case where $s > 0$ and $i = 1$, the case where $s > 0$ and $i = 0$ being similar and left to the reader. The control sequence $(\sbf, \ibf ) \in (0, \infty)^{2m+1} \times E^{2m+1}$ defined by
\[
\sbf =(t_1 - s, t_{0}, t_1, \ldots, t_0, \varepsilon) , \quad \ibf = (1, 0, 1, \ldots, 0, 1),
\]
is admissible for $(x,s,1)$ and thus the point $(\varphi_{\varepsilon}^{0}(\Psi^{m-1}(\varphi^0_{t_0} \circ \varphi^1_{t_1 - s}(x))), \varepsilon, 0)$ is reachable from $(x,s,1)$. Using once again Lemma \ref{lem:fixedpoint}, this proves that $(x^*,0,0)$ is accessible from $(x,s,1)$.
\end{proof}

Hence, we have proved the existence of an accessible point in the case where the jump rates do not depend on the position. This is however in general false, as we show in the following example. The intuition is that if the jump rates are such that, when the process is close to $0$, the environment spends much more time in environment $0$ than $1$, and conversely close to $1$; then the process starting near $0$ remains in a neighbourhood of $0$, say $\mathcal{U}$, while starting near $1$ it remains in a neighbourhood of $1$ at positive distance of $\mathcal{U}$. 

\begin{example}       \rm 
Assume that $\mu_0 = \mu_1 = \delta_1$. Let $t_0$ be such that $\varphi_{t_0}^0(3/8) \leq 1/4$ and $t_1$ such that $\varphi_{t_1}^1(1/4) \leq 3/8$. Let also $u_0$ such that $\varphi_{u_0}^0(3/4) \geq 5/8$ and $u_1$ be such that $\varphi_{u_1}^1(5/8) \geq 3/4$. Finally, let $\lambda_0, \lambda_1$ continuous functions on $[0,1]$ such that for all $x \in [0, 3/8]$, $\lambda_0(x) = t_0^{-1}$ and $\lambda_1(x) = t_1^{-1}$, while for all $x \in [5/8, 1]$, $\lambda_0(x) = u_0^{-1}$ and $\lambda_1(x) = u_1^{-1}$. We claim that for all $(x,s,i) \in K$, if $x \in [0, 1/4]$, then $X_t \in [0, 3/8]$ for all $t \geq 0$, while if $x \in [3/4, 1]$, then $X_t \in [5/8, 1]$ for all $t \geq 0$. In particular, there is no accessible point. Assume that $x \leq 1/4$ and $i=1$. Then, by definition of $t_1$ and construction of $\lambda_1$, for all $t \leq t_1$, 
\[
\int_{-s}^{t} \lambda_1(\varphi_{u}^1(x)) du = \frac{t+s}{t_1} 
\]
Since $ {G}_1(t) = \1_{t <1}$, for $t \leq t_1$,
\[
 {G}_{x,s,1}(t) = \1_{ \frac{t+s}{t_1} < 1} = \1_{ t < t_1 - s}
\]
This means that $\mu_{x,s,1}$ is a Dirac mass at $t_1 - s \leq t_1$. In particular,  for all $t \leq T_1 = t_1 - s$, $X_t = \varphi_{t}^1(x) \leq \varphi_{T_1}^1(x) \leq 3/8$. Now, by construction of $\lambda_0$, for all $t \geq 0$,
\[
\int_0^t \lambda_0( \varphi_u^0 \circ \varphi^1_{T_1}(x) ) du = \frac{t}{t_0},
\]
and thus $T_2 = t_0$ almost surely. Since $\varphi_{T_1}(x) \leq 3/8$, the choice of $t_0$ implies that for all $t \in [T_1, T_1 + T_2)$, $X_t = \varphi_{t - T_1}^0 \circ \varphi_{T_1}^1(x) \leq 1/4$. Now, since $X_{T_1 + T_2} \leq 1/4$, we can repeat the procedure to prove that for all $t \geq 0$, $X_t \leq 3/8$. The same reasoning proves that if $x \geq 3/4$, then $X_t \geq 5/8$ for all $t \geq 0$.
\end{example}

\section{Applications} 
\label{sec:applications}

\subsection{Application to inflation} This example is taken from \cite{BLSS21}, where a growth population model with dispersal between two patches is studied. Let $E = \{+, - \}$, and for $h \in E$, $\mu^h$ a probability measure on $\mathbb{R}_+$, and let $I_t$ a semi - Markov process on $E$ with jump laws given by $(\mu^h, h \in E)$ as described in Section~\ref{sec:PDP}. We consider the PDMP given by
\[
\begin{cases}
\frac{d X_1}{dt} = (I_t - \varepsilon)X_1 + m (X_2 - X_1)\\
\frac{d X_2}{dt} =( - I_t - \varepsilon)X_2 + m(X_1 - X_2),
\end{cases}
\]
where $X_k, k=1,2$ denote the population in patch $k$; $\varepsilon > 0$ is a parameter, and $m$ is the migration rate.
If we set $V = \ln(X_1) - \ln(X_2)$ and $U = \ln(X_1) + \ln(X_2)$, the system becomes
\[
\begin{cases}
\frac{d V}{dt} = 2(I_t - m \sinh(V))\\
\frac{d U}{dt} = 2( \cosh(V) - m - \varepsilon).
\end{cases}
\]
In particular, one sees that $V$ does not depend on $U$, hence $(V, \tau, I)$ is a PDMP. Moreover, once $V$ is know, $U$ is obtained from $V$ by a simple quadrature 
\[
U_t = U_0 + 2 \int_0^t (\cosh(V_u) - m - \varepsilon) du.
\]
Let $\mathbf{V}_n = V(T_n)$ the value of $V_t$ at the $n$-th jumping time. It is proven in \cite{BLSS21}, that the sequences $\mathbf{V}_{2n}$ and $\mathbf{V}_{2n+1}$ converge in law respectively to some random variable $\mathbf{V}_{\infty}^-$ and $\mathbf{V}_{\infty}^+$, which ensures that $\frac{U_t}{t}$ converges almost surely to some $\Delta$. In \cite[Remark 8]{BLSS21}, the authors claim that under additional assumptions on the jump time laws $\mu^h$, then $V_t$ converges in law to some random variable. We now give these additional assumptions, and state and prove the expected convergence result.
\begin{assumption}
\label{ass:jumplaw}
For $h \in \{+, -\}$, the measure $\mu_h$ is such that:
\begin{enumerate}
\item $ {G}_h$ is continuous;
\item There exists $C, \beta > 0$ such that $ {G}_h(t) \leq C e^{ - \beta t}$
\item $\mu_h$ admits a regular point $t^r_h$.
\end{enumerate}
\end{assumption}
\begin{proposition}
\label{prop:ergodic-inflation}
The process $(V_t,\tau_t,I_t)_{t \geq 0}$ admits a unique stationary distribution $\Pi$. Furthermore, there exists $\gamma > 0$ and $f : \mathbb{R}_+ \to \mathbb{R}$ such that for all $z=(v,s,i) \in K_M$,
\[
\| P_t(z, \cdot) - \Pi \|_{TV} \leq f(s) e^{ - \gamma t}. 
\]
If in addition, $ {G}_h$ is $C^1$, then $\Pi$ is absolutely continuous with respect to $\psi$.
\end{proposition}

\begin{proof}
For $h \in \{+, -\}$, and $v \in \mathbb{R}$, let $F^h(v) = 2(h1 - m \sinh(V))$. Then, setting $V^+ = \sinh( 1/m)$ and $V^- = - V^+$, one has $F^h(V^h)=0$. Moreover, $DF^h(v) = - 2 m \sinh(v) < 0$ for all $v \in [V^-, V^+]$. Hence, conditions (1) and (2) are satisfied and Lemma \ref{lem:existacces} ensure the existence of $v^* \in [V^-, V^+]$ such that $(v^*,0,0)$ is accessible. Since $F^+(v^*) - F^-(v^*) = 4 > 0$, the strong bracket condition holds at $v^*$. Moreover, $Q$ does not depend on $x$ hence is globally irreducible, $\mu^h$ has a regular point $t^h_r$ and $F^h$ is analytic. Thus, Assumption \ref{ass:jumplaw} and Corollary~\ref{cor:uniqueandergodic} yield the result.
\end{proof}

\subsection{Competitive Lotka - Volterra system in random environment with dwell time}
We consider the competitive Lotka - volterra model given by 
\begin{equation}
\label{eq:LV2d}
\begin{cases}
\frac{d X_t}{dt} = \alpha_{I_t}X_t ( 1 - a_{I_t} X_t - b_{I_t} Y_t)\\
\frac{d Y_t}{dt} = \beta_{I_t} Y_t ( 1 - c_{I_t} X_t - d_{I_t} Y_t)
\end{cases}
\end{equation}
where $(I_t)_{t \geq 0}$ is a semi - Markov process on $E = \{ 0, 1\}$, and for $i \in E$, $\alpha_i, \beta_i, a_i, b_i, c_i$ and $d_i$ are positive constants. This model of two species in competition was studied in \cite{BL16} in the particular case where $I_t$ is an autonomous Markov chain, i.e., $G^i : t \mapsto e^{-t}$ for all $i \in E$ and the functions $\lambda^i$ are constant. In \cite{BL16}, the authors assumed that in each environment, species 1 will survive will species 2 will go to extinction. This is true if and only if 
\begin{equation}
a_i < c_i, \quad b_i < d_i.
\end{equation}
They prove that, under some conditions, whenever the switching rates $\lambda^0$ and $\lambda^1$ are large enough, then the stochastic process $(X_t, Y_t)_{t \geq 0}$ has a completely different behaviour than the deterministic behaviour in each environment; namely, $X_t$ goes to $0$ while $Y_t$ is persistent (see Theorem 3.3 and Figure 7 in \cite{BL16}). This is due to the fact that, when switching quickly between the two environments, trajectories of $(X_t, Y_t)_{t \geq 0}$ become close to the trajectories of the deterministic solution to an ODE similar to~\eqref{eq:LV2d} but with the constants $\alpha_i, \beta_i, a_i, b_i, c_i$ and $d_i$ replaced by some convex combination of this constant $  \alpha$ etc, that can verify
\begin{equation}
\label{eq:mean_environment}
\bar a > \bar c, \quad \bar b  > \bar d.
\end{equation}
The proof is based on Stochastic Persistence arguments, which requires in particular to study the sign of the invasion rates of species 2, defined as 
\[
\Lambda_\mathrm{y} = \int \beta_i( 1 - c_i x) \nu(dx, i),
\]
where $\nu$ is the unique invariant measure of the process $(X_t, Y_t, I_t)$ restricted to the set $ \mathbb{R}_+^* \times \{0 \} \times E$. When \eqref{eq:mean_environment} holds, then for sufficiently large switching rates $\lambda^0, \lambda^1$, it is proven in \cite{BL16} that $\Lambda_\mathrm{y} > 0$. 

The purpose of this section is to prove that one can also define an invasion rate of species 2 for the solution to~\eqref{eq:LV2d} when $I_t$ is a semi-Markov process. In addition, if the law of the sojour times $\mu^i$ have dwell times - that is, their supports are included in $[\delta, + \infty)$ for some $\delta > 0$ - then $\Lambda_y < 0$. Before to precise the statement in Proposition~\ref{prop:LV2d} below, we introduce some notations. We assume that $\lambda^ i \equiv 1$ for $i \in E$. For some $\eta > 0$, the compact set 
\[
M = \{ (x,y) \in \mathbb{R}_+^2 \: : \eta \leq x + y \leq 1 / \eta \}
\]
is positively invariant by the flows induced by the the vector fields $F^0$ and $F^1$ appearing in the right hand side of~\eqref{eq:LV2d}. Hence,  we consider the state space $K_M$, as defined by~\eqref{eq:KM}, for the process $(X_t, Y_t, \tau_t, I_t)_{ t \geq 0}$. We also let $K_M^\mathrm{x} = \{ (x,y,s,i) \in K_m \: : y = 0 \}$ denote the extinction set of species 2. Let $p_i = a_i^{-1}$ denote the positive equilibrium of $\alpha_i x ( 1 - a_i x)$. We assume without loss of generality that $p_0 < p_1$.
\begin{proposition}
\label{prop:LV2d}
Let Assumption \ref{ass:jumplaw} holds true for $\mu^0$ and $\mu^1$. Then,
\begin{enumerate}
\item the process $(X_t, \tau_t, I_t)_{t \geq 0}$ on $K_M^x$ admits a unique stationary distribution $\nu$;
\item Set 
\[
\Lambda_y = \int_{K_M^\mathrm{x}} \beta_i ( 1 - c_i x) \nu(dx, s , i).
\]
If $\mathrm{supp}(\mu^1) \subset [\delta_1, + \infty)$ with 
\[
\delta_1 = \frac{c_1 p_1}{\alpha_1( c_1 p_1 - 1)}\ln\left( \frac{p_1}{p_0} \right),
\]
then $\Lambda_\mathrm{y} < 0$. In particular, if the environment stays a time larger than $\delta_1$ in state 1, species 2 cannot invade species 1. 

\end{enumerate}  
\end{proposition}

\begin{proof}
For the first point, we apply Lemma~\ref{lem:existacces} to prove the existence of an accessible point. The existence and uniqueness of the stationary distribution $\nu$ will follow from Corollary~\ref{cor:uniqueandergodic},  as in the proof of Proposition~\ref{prop:ergodic-inflation}. If we set $F^i(x) = \alpha_i x(1 - a_i x)$, then $F^i(p_i) = 0$ and the compact set $M = [p_0, p_1]$ is positively invariant by the flows generated by the $F^i$. since $DF^i(x) = \alpha_i (1  - 2 a_i x)$, one sees that if $2 p_0 < p_1$, then $DF^1(x) >0 $ for $x \in [p_0, \frac{p_1}{2})$ and condition (2) is not satisfied. The trick is to perform the change of variable $W_t = \ln(X_t)$, which is well defined since $X_t \in [p_0, p_1]$ almost surely if $X_0 \in [p_0, p_1]$. Is it easily seen that $W_t$ is solution to 
\[
\frac{d W_t}{dt} = \alpha_{I_t}\left( 1 - a_{I_t} e^{W_t} \right).
\]
Setting $G^i(w) = \alpha_i (1 - a_i e^w)$, one has $G^i(\ln(p_i)) = 0$ and $DG^i(w) = - \alpha_i a_i e^{w}<0$. Hence Lemma \ref{lem:existacces} applies and one has the existence of $w^* \in [\ln(p_0), \ln(p_1)]$ such that $(w^*, 0, 0)$ is accessible for $(W, \tau, I)$. Since $X_t = e^{W_t}$, this yields immediately that the point $(x^*,0,0) = (e^{w^*},0,0)$ is accessible for $(X, \tau, I)$.

We now prove the second point. since $\nu$ is uniquely ergodic, for all $z_0=(x,s,i) \in K_M^{\mathrm{x}}$, one has $\mathbb{P}_{z_0}$ - almost surely, 
\begin{align*}
\Lambda_\mathrm{y} & = \lim_{T \to \infty} \frac{1}{T} \int_0^T \beta_{I_t}( 1 - c_{I_t} X_t) dt\\
& = \lim_{N \to \infty} \frac{1}{T_N} \int_0^{T_N} \beta_{I_t}( 1 - c_{I_t} X_t) dt.
\end{align*}
Now, for every $N > 0$, setting $h(x,i) = \beta_i(1 - c_i x)$,
\[
\int_0^{T_N} \beta_{I_t}( 1 - c_{I_t} X_t) dt = \sum_{k=0}^{N-1} \int_{T_k}^{T_{k+1}} h(X_t, I_t) dt.
\]
Recall that for all $i \in E$, $a_i < c_i$. This implies that $h(x,0) < 0$ for all $x \in [p_0, p_1]$. Set $q_1 = \frac{1}{c_1}$. If $q_1 < p_0$, then we also have that $h(x,1) < 0$ for all $x \in [p_0, p_1]$ and thus $\Lambda_y < 0$ by the equalities above. Assume now that $q_1 \in [p_0, p_1]$. For all $k$ such that $I_{T_k} =0$, we have $\int_{T_k}^{T_{k+1}} h(X_t, I_t) dt < 0$ due to $h(x,0) < 0$ for all $x \in [p_0, p_1]$. Since $x \mapsto h(x,1)$ is decreasing, and since $\varphi_t^1(x) \geq \varphi_t^1(p_0)$ for all $t \geq 0$ and $x \geq p_0$, we deduce that for all $T > 0$,
\[
\int_0^T h( \varphi^1_t(x), 1) dt \leq \int_0^T h( \varphi^1_t(p_0), 1) dt.
\]
We have
\[
\varphi^1_t(p_0) = \frac{p_1}{1 + \left(\frac{p_1}{p_0} - 1\right) e^{ - \alpha_1 t}},
\]
from which we deduce
\begin{align*}
\int_0^T h( \varphi^1_t(p_0), 1) dt & = \beta_1 T - \beta_1 c_1 \int_0^T \frac{p_1}{1 + \left(\frac{p_1}{p_0} - 1\right) e^{ - \alpha_1 t}} dt \\
& = \beta_1 T - \beta_1 c_1 \frac{p_1}{\alpha_1} \int_{e^{- \alpha_1 T}}^1 \frac{1}{1 + \left(\frac{p_1}{p_0} - 1\right) u} \frac{du}{u} \\
& = \beta_1 T - \beta_1 c_1 \frac{p_1}{\alpha_1}\left( \alpha_1 T - \ln\left( \frac{p_1}{p_0} \right) + \ln\left( 1 + \left(\frac{p_1}{p_0} - 1\right)e^{ - \alpha_1 T} \right) \right)\\
& \leq \beta_1 T ( 1 - c_1 p_1) + \frac{\beta_1 c_1 p_1}{\alpha_1} \ln\left( \frac{p_1}{p_0} \right)
\end{align*}
Recall that $1 - c_1 p_1 < 0$. The above inequalities show that, if $T \geq \delta_1$, with 
\[
\delta_1 = \frac{c_1 p_1}{\alpha_1( c_1 p_1 - 1)}\ln\left( \frac{p_1}{p_0} \right),
\]
then $\int_0^T h(\varphi^1_t(x), 1) dt \leq 0$, for all $x \in (p_0, p_1)$. Since for all $k \geq 0$ such that $I_{T_k} = 1$, one has 
\[
\int_{T_k}^{T_{k+1}} h(X_t, I_t) dt = \int_0^{T_{k+1} - T_{k}} h( \varphi^1_t(X_{T_k}), 1) dt,
\]
this implies that the above integral is non positive as soon as $S_{k+1} = T_{k+1} - T_k$ is greater than $\delta_1$. Hence, whenever the support of $\mu^1$ is included in $[ \delta_1, + \infty)$, $S_{k+1} \geq \delta_1$ almost surely and $\int_{T_k}^{T_{k+1}} h(X_t, I_t) dt \leq 0$. Combined with the above inequalities, this shows that in that case, $\Lambda_\mathrm{y} < 0$ and finishes the proof.
\end{proof}

\section{Proof of Proposition \ref{prop:Feller}: necessary and sufficient condition for Feller property}
\label{sec:proofFeller}
We first prove that when the $ G_i$ is continuous for all $i$, the process if Feller. The proof is inspired by the proof given by Davis in \cite{Dav93}.
Let $f \in C_b(K)$, and consider the operator $H$ acting on $C_b(K \times \mathbb{R}+)$ defined, for all $\Psi \in C_b(K \times \mathbb{R}_+)$, $z = (x,i,\tau) \in K$, $t \geq 0$ by
\[
H \Psi(z,t) = f( \phi_t(z) ) {G}_z(t) + \sum_{j \in E} \int_0^t \Psi\left(t-u, (\varphi^i_u(x),j,0)\right) q_{i,j}(\varphi^i_u(x)) d \mu(z). 
\]
By the strong Markov property of $Z$, one can rewrite $H \Psi$ as
\[
H \Psi(z,t) = \EE_{z} \left[ f(Z_t) \1_{t < T_1} + \Psi\left(t-T_1, Z_{T_1}\right)\1_{t \geq T_1} \right],
\]
and it can be proven iteratively that for all $k \geq 1$,
\[
H^k \Psi(z,t) = \EE_{z} \left[ f(Z_t) \1_{t < T_k} + \Psi\left(t-T_k, Z_{T_k}\right)\1_{t \geq T_k} \right].
\]
Now, if $f \in C_b(K)$ and $\Psi \in C_b(K \times \mathbb{R}_+)$, we claim that 
\begin{enumerate}
\item $H \Psi$ (and hence $H^k \Psi$) belongs to $C_b(K \times \mathbb{R}_+)$;
\item  $H^k \Psi$ converges uniformly on compact sets to   $(t,x) \mapsto P_t f(x)$.
\end{enumerate}
These two claims imply immediately that $ (t,x) \mapsto P_t f(x)$ is continuous, thus the Feller property. We prove the first claim. Let $h : \mathbb{R}_+ \times \mathbb{R}_+ \times K \to \mathbb{R}$ be a continuous bounded map, and consider, the map 
\[
\eta : (z,t) \mapsto \int_0^t h(u, t-u, z) d \mu_z(u).
\]
We show that $\eta$ is continuous: let $(z,t)\in K \times \mathbb{R}_+$ be fixed. Then, for all $\varepsilon > 0$, there exists $\delta > 0$ such that, for all $z' \in K$ with $d(z,z') < \delta$ and $t' \in \mathbb{R}_+$ with $|t - t'| < \delta$,  
\[
\sup_{u \in [0,t]} | h(u,t-u,z) - h(s,t'-u,z')| < \varepsilon.
\]
Thus, if one assumes that $t' \geq t$,
\begin{align*}
| \eta(z,t) - \eta(z',t')| &  \leq \int_0^t | h(u, t-u, z) - h(u, t'-u, z') | d\mu_{z'}(u)\\
& \quad  + | \int_0^t h(u, t-u, z) d \mu_z(u) - \int_0^t h(u, t-u, z) d \mu_{z'}(u)|\\
& \quad  + \int_t^{t'} | h(u,t'-u,z')| d \mu_{z'}(u)\\
& \leq \varepsilon + \left| \int_0^t h(u, t-u, z) d \mu_z(u) - \int_0^t h(u, t-u, z) d \mu_{z'}(u)\right|\\
& \quad + \|h\||G_{z'}(t') - G_{z'}(t)|.
\end{align*}
Now, since $G_i$ is continuous for all $i$, $z' \mapsto \mu_{z'}$ is continuous for the topology of weak convergence : indeed, for all $u \geq 0$, the function $z' \mapsto G_{z'}(u)$ is continuous, hence as $z'$ converges to $z$, the cumulative distribution function of $\mu_{z'}$ converges to the cumulative distribution function of $\mu_z$. This implies that $\lim_{z' \to z} \int_0^t h(u, t-u, z) d \mu_{z'}(u) = \int_0^t h(u, t-u, z) d \mu_z(u)$, which in addition to the continuity of $(z',t') \mapsto G_{z'}(t')$ proves the continuity of $\eta$. This proves the first claim by noting that $H \Psi(z,t) = f(\phi_t(z)){G}_z(t) + \eta(z,t)$ for the function $h $ defined on $\mathbb{R}_+ \times \mathbb{R}_+ \times K$ by
\[
h(u,s,z) = \sum_j \Psi\left(s,(\varphi^i_u(x),j,0)\right)q_{i,j}\left( \varphi^i_u(x) \right);
\] 
which is bounded and continuous. 

Now we turn to the proof of the second claim. Let $t \in [0, T]$, for some $T > 0$. Then, 
\begin{align*}
| H^k \Psi(z,t) - P_t f(z)| & \leq  \EE_{z} \left( | \Psi\left(t-T_k, Z_{T_k}\right) - f(Z_t) | \1_{t \geq T_k} \right)\\
& \leq \left( \| \Psi \| + \|f\| \right) \mathbb{P}_z ( t \geq T_k)\\
& \leq \left( \| \Psi \| + \|f\| \right) \mathbb{P}( T \geq \tilde{T}_{k-1}),
\end{align*}
where $\tilde{T}_k$ is defined in Proposition \ref{prop:finitejumps}. This concludes the proof of the direct part.

Now, we prove that if one the $ G_i$ is not continuous, then the process if not Feller. Let $i \in E$ and $T > 0$ such that ${G}_i$ is not continuous in $T$. We denote $ G_i(T^-)$ the left limit of $ G_i(t)$ as $t$ goes to $T$ from the left. By assumption on $T$, we have $ G_i(T^-) -  G_i(T) > 0$.  We consider two cases :
\begin{enumerate}
\item $T < \bar t_i$
\item For all $j \in E$, the only possible discontinuity point of $  G_j$ is $\bar t_j$.
\end{enumerate} 
First assume that $T < \bar t_i$. Fix $x \in \R^d$, and 
let $s^*$ be such that 
\[
\int_{-s^*}^{0} \lambda_i ( \varphi_u^i(x) ) du = T 
\] and, for all $s \leq s^*$, set
\[
t(s) = \inf \{ t \geq 0 \: : \int_{-s}^{t} \lambda_i ( \varphi_u^i(x) ) du = T \}.
\]
Note that since $T < \bar t_i$, the point $z^*=(x,s^*,i)$ belongs to $K$. We now show that, for continuous and bounded $f : K\to \R$, the map $(z,t) \mapsto P_t f(z)$ is in general not continuous at $(z^*,0)$. Let $f : K \to \R$ be continuous and bounded. By definition of $t(s)$, one has, for all $(x,s)$ with $(x,s,i) \in K$, 
\begin{multline*}
P_{t(s)} f(x,s,i) = f( \varphi^i_{t(s)}(x), i, t(s) + s) \frac{  G_i(T)}{  G_i (\int_{-s}^{0} \lambda_i ( \varphi_u^i(x) ) du) } \\+
 \sum_{j \in E} q_{i,j} (\varphi^i_{t(s)}(x)) f(\varphi^i_{t(s)}(x),j,0) \frac{  G_i(T) -   G_i(T-)}{  G_i (\int_{-s}^{0} \lambda_i ( \varphi_u^i(x) ) du) } + \EE_{x,s,i} \left( f(Z_{t(s)} ) \1_{ S_1 < t(s)} \right).
\end{multline*}
Let $(s_n)_{n  \geq 0}$ be a sequence converging to $s^*$ from the left, and let $t_n = t(s_n)$. Then, 
\[
\lim_{n \to \infty} \int_{-s_n}^{0} \lambda_i ( \varphi_u^i(x) ) du)=T,
\] 
hence, by definition of $t(s_n)$ and continuity of the flow, we have that $(x,s_n,i,t_n)$ converges to $(z^*, 0)$  as $n$ goes to infinity.  Moreover, $  G_i (\int_{-s_n}^{0} \lambda_i ( \varphi_u^i(x) ) du)$ converges to $  G_i(T^-)$. Finally, we have that 
\[
\mathbb{P}_{x,s_n,i}(S_1 < t_n) = 1 - \frac{  G_i(T)}{  G_i (\int_{-s_n}^{0} \lambda_i ( \varphi_u^i(x) ) du)} - \frac{  G_i(T^-) -   G_i(T)}{  G_i (\int_{-s_n}^{0} \lambda_i ( \varphi_u^i(x) ) du)},
\]
which converges to $0$ as $n$ goes to infinity. Putting all together, we get that 
\[
\lim_{n \to \infty} P_{t_n} f(x,s_n,i) = 
 f( x, i, s^*) \frac{  G_i(T)}{  G_i (T^-) } +
 \sum_{j \in E} q_{i,j} (x) f(x,j,0) \frac{  G_i(T) -   G_i(T^-)}{  G_i (T^-) },
\]
which, since $  G_i(T) -   G_i(T^-) > 0$  is different from $f(x,s^*,i) = P_0 f(x,s^*,i)$ as soon as $\sum_{j \in E} q_{i,j} (x) f(x,j,0) \neq f(x,i,s^*)$. Hence, except for the functions $f$ satisfying $\sum_{j \in E} q_{i,j} (x) f(x,0,j) = f(x,s^*,i)$ for all $x$, the map $(z,t) \mapsto P_t f(z)$ is not continuous at $(z^*,0)$, hence the process is not Feller.

It remains to deal with the case where $T = \bar t_i$. For sake of simplicity, we do the proof in the special case where $\lambda_i$ is constant equal to $1$, the proof can be easily extend to the non-constant case.

We let  $f : K \to \R$ be a continuous bounded map, such that $f(x,s,j) =  s$ for all $(x,s,j) \in K$ with $s \leq T$. We show that there exists $\beta \in (0,1)$ such that $P_{(1-\beta) T} f$ is not continuous at any point $(x,s,i)$ with $s = \beta T$. On the one hand, starting from a point $(x, \beta T, i)$, the first jump time occurs either stricly before $(1-\beta)T$ with probability $1 -  {G}_i(T^-)/  G_i(\beta T)$ or exactly at $(1 - \beta) T$ with  probability $ {G}_i(T^-)/  G_i(\beta T)$. Thus, using that $f(x,0,j) = 0$ for all $j \in E$, we have
\[
P_{(1-\beta)T} f(x, \beta T, i) = 0 + \EE_{(x,\beta T, i)}\left[ \left( (1 - \beta) T - T_{N_t} \right) \1_{S_1 < (1 - \beta)T} \right].
\]
Hence, 
\[
P_{(1-\beta)T} f(x, \beta T, i) \leq (1 - \beta) T \left( 1 - \frac{ {G}_i(T^-)}{  G_i(\beta T)} \right).
\]
On the other hand, starting from a point $(x, \beta T - \eps, i)$, the first jump time occurs either stricly before $(1-\beta)T$ with probability $1 -  {G}_i(T-\eps)/  G_i(\beta T)$ or stricly after  $(1 - \beta) T$ with  probability $ {G}_i(T- \eps)/  G_i(\beta T)$. Thus, using that $f(\phi_{(1-\beta)T} (x, \beta T - \eps,i) = T - \eps$, we have
\[
P_{(1-\beta)T} f(x, \beta T - \eps, i) = (T - \eps)\frac{ {G}_i(T-\eps)}{  G_i(\beta T - \eps)} + \EE_{(x,\beta T - \eps, i)}\left[ \left( (1 - \beta) T - T_{N_t} \right) \1_{S_1 < (1 - \beta)T} \right].
\]
Hence, 
\[
P_{(1-\beta)T} f(x, \beta T - \eps, i) \geq (T - \eps)\frac{ {G}_i(T-\eps)}{  G_i(\beta T - \eps)}
\]
Now, since $  G_i$ is continuous at point $\beta T$, we get that
\[
\limsup_{ \eps \to 0}  P_{(1-\beta)T} f(x, \beta T - \eps, i) \geq T \frac{ {G}_i(T^-)}{  G_i(\beta T)}.
\]
As $\beta$ goes to $1$, $(1 - \beta) T \left( 1 - \frac{ {G}_i(T^-)}{  G_i(\beta T)} \right)$ goes to 0, whereas $T \frac{ {G}_i(T^-)}{  G_i(\beta T)}$ goes to $T$. Hence, we can choose $\beta \in (0,1)$ such that 
\[
\limsup_{ \eps \to 0}  P_{(1-\beta)T} f(x, \beta T - \eps, i) \geq \frac{T}{2} > (1 - \beta) T \left( 1 - \frac{ {G}_i(T^-)}{  G_i(\beta T)} \right) \geq P_{(1-\beta)T} f(x, \beta T, i),
\]
which proves that $P_{(1-\beta)T} f$ is not continuous at $(x, \beta T, i)$ and thus that the process is not Feller.

\section{Proof of Proposition \ref{prop:lyap}: construction of a Lyapunov function}
\label{sec:proofLyap}
First, let us prove that $f$ is a proper map. Note that we can rewrite $f$ as 
\[
f(x,s,i) = \frac{e^{-\gamma s}}{  G_i\left(  \int_{-s}^0 \lambda_i (\varphi^i_u(x)) du \right) } \mathbb{E}_{\varphi_{-s}^i(x), 0, i}\left[ (e^{\gamma S_1} - 1)  \1_{S_1 \leq s}\right]  + 1 - e^{ - \gamma s},
\]
which proves that $f$ is nonnegative. Now, we claim that 
\[
f(x,s,i) \geq \frac{e^{-\gamma s}}{  G_i\left(  \lambda_{\min} s \right) } \Ee \left[ (e^{\gamma \frac{S_1^i}{\lambda_{\max}}} - 1)  \1_{\frac{S_1^i}{\lambda_{\min}} \leq s}\right] :=g(s)
\]
where, with the notations of Section \ref{sec:setting}, $S_1^i = \psi_i(U_1)$ and 
\[
\psi_i(u) = \inf\{ r \geq 0 \: :  {G}_i(r) \geq u \}.
\]
In other word, $S_1$ is a random variable with law $\mu_i$. Still with the same notations, for $x \in M$, we define $S_1^{x,i} = \psi_{x,i}(U_1)$. Then, it is easily seen that
\[
\frac{S_1^i}{\lambda_{\max}} \leq S_1^{x,i}  \leq \frac{S_1^i}{\lambda_{\min}}
\]
which proves the claim since under $\PP_{\varphi_{-s}^i(x), 0, i}$, $S_1$ has the same law as $S_1^{\varphi_{-s}^i(x), i}$. Now, by definition of $\bar t_i$, assumption on $  G_i$ and definition of $\gamma$, 

\[
\lim_{ s \to \frac{\bar{t}_i}{\lambda_{\min}} } \frac{e^{-\gamma s}}{  G_i\left(  \lambda_{\min} s \right) }  = + \infty,
\]
while, by monotone convergence, 

\[
\lim_{ s \to \frac{\bar{t}_i}{\lambda_{\min}} } \Ee \left[ (e^{\gamma \frac{S_1^i}{\lambda_{\max}}} - 1)  \1_{\frac{S_1^i}{\lambda_{\min}} \leq s}\right] = \Ee \left[ (e^{\gamma \frac{S_1^i}{\lambda_{\max}}} - 1) \right] > 0.
\]
Hence, $\lim_{s \to \frac{\bar{t}_i}{\lambda_{\min}}} g(s) = + \infty$. In particular, if $R \geq 0$, there exists $s_0$ such that $g(s) \geq R +1$ for all $s \geq s_0$. Since for all $(x,s,i) \in K_M$, $f(x,s,i) \geq g(s)$, this implies that $\{f \leq R \}$ is a closed set included in $K_M \cap \left( M \times [0,s_0] \times E \right)$, which is a compact subset of $K_M$, hence $\{f \leq R \}$ is compact in $K_M$ and $f$ is a proper map. 

We next prove the bound on $P_t f$. We define the functions $\tilde f, h: K_M \to \R_+$ by setting $\tilde f= f - 1$ and
\[
h(x,s,i)= \mathbb{E}_{\varphi_{-s}^i(x), 0, i}[ e^{\gamma S_1} \1_{S_1 \leq s}],
\]
so that
\[
\tilde f(x,s,i) = \frac{e^{-\gamma s}}{  G_i\left(  \int_{-s}^0 \lambda_i (\varphi^i_u(x)) du \right) }\left( h(x,s,i) - 1 \right).
\]
We prove that $P_t \tilde f \leq e^{- \gamma t} \tilde f $, then since  $P_t f =  P_t \tilde f + 1$, we will have $P_t f \leq e^{- \gamma t}  f+ 1 - e^{- \gamma t}$, hence the result.
Let $t \geq 0$, then
\[
P_t \tilde f(x,s,i) = \tilde f(\phi_t(x,s,i))   G_{(x,s,i)}(t) + \sum_{n=1}^{+\infty} \EE_{x,s,i} \left( \tilde f(Z_t) \1_{N_t=n} \right)
\]
Note that for all $t \geq 0$,
\[
h(\phi_t(x,s,i))= h(x,s,i) + \mathbb{E}_{\varphi_{-s}^i(x), i, 0}[ e^{\gamma S_1} \1_{s<S_1 \leq t+s}],
\]
and recall that 
\[
  G_{(x,s,i)}(t) = \frac{  G_i \left( \int_{-s}^t \lambda_i( \varphi^i_u(x) ) ds \right)}{  G_i \left( \int_{-s}^0 \lambda_i( \varphi^i_u(x) ) ds \right)}.
\]
Hence,
\begin{multline*}
\tilde f(\phi_t(x,s,i))   G_{(x,s,i)}(t) \\
= e^{- \gamma t} \tilde f(x,s,i)+ e^{-\gamma t} \frac{e^{-\gamma s}}{  G_i\left(  \int_{-s}^0 \lambda_i (\varphi^i_u(x)) du \right)}\mathbb{E}_{\varphi_{-s}^i(x), 0, i}[ e^{\gamma S_1} \1_{s<S_1 \leq t+s}]
\end{multline*}
Let $n \geq 1$, then,
\begin{align*}
\EE_{x,s,i} \left( \tilde f(Z_t) \1_{N_t=n} \right) & = \EE_{x,s,i} \left( \tilde f(\varphi_{t - T_n}^{\bm{I}_n}(\bm{X}_n),\bm{I}_n, t - T_n) \1_{T_n \leq t < T_n + S_{n+1}} \right)\\
& = \EE_{x,s,i} \left( \tilde f(\varphi_{t - T_n}^{\bm{I}_n}(\bm{X}_n),\bm{I}_n, t - T_n) \1_{T_n \leq t }   G_{\bm{X}_n, \bm{I}_n}( t - T_n) \right)\\
& =  \EE_{x,s,i} \left( e^{ - \gamma(t - T_n)}(h(\varphi_{t - T_n}^{\bm{I}_n}(\bm{X}_n),\bm{I}_n, t - T_n) - 1) \1_{T_n \leq t } \right)\\
& = a_n - b_n,
\end{align*}
where 
\[
a_n = \EE_{x,s,i} \left( e^{ - \gamma(t - T_n)} h(\varphi_{t - T_n}^{\bm{I}_n}(\bm{X}_n),\bm{I}_n, t - T_n)  \1_{T_n \leq t } \right)
\]
and 
\[
b_n =  \EE_{x,s,i} \left( e^{ - \gamma(t - T_n)} \1_{T_n \leq t } \right).
\]
Let us prove that for all  $n \geq 1$, $b_{n+1} = a_n$.
We have
\begin{align*}
b_n & =  \EE_{x,s,i} \left( e^{ - \gamma(t - T_{n+1})} \1_{T_{n+1} \leq t } \right)\\
& = \EE_{x,s,i} \left( e^{ - \gamma(t - T_{n})} \1_{T_{n} \leq t } e^{ \gamma S_{n+1}} \1_{S_{n+1} \leq t - T_n} \right)\\
& =  \EE_{x,s,i} \left( e^{ - \gamma(t - T_{n})} \1_{T_{n} \leq t } \EE( e^{ \gamma S_{n+1}} \1_{S_{n+1} \leq t - T_n} | \mathcal{F}_{T_n} ) \right)\\
& = \EE_{x,s,i} \left( e^{ - \gamma(t - T_{n})} \1_{T_{n} \leq t } \EE_{\bm{X}_n,\bm{I}_n,0}( e^{ \gamma S_{1}} \1_{S_{1} \leq t - T_n} ) \right)
\end{align*}
where the last equality comes from the strong Markov property. Now, by definition, 
\[
h(\varphi_{t - T_n}^{\bm{I}_n}(\bm{X}_n),\bm{I}_n, t - T_n)) = \EE_{\bm{X}_n,0,\bm{I}_n}( e^{ \gamma S_{1}} \1_{S_{1} \leq t - T_n} ),
\]
which proves that $a_n = b_{n+1}$. Moreover, $b_n \leq \mathbb{P}_{x,s,i} ( T_n \leq t)$, hence $b_n$ converges to $0$ as $n$ goes to infinity. In particular, this proves that 
\[
 \sum_{n=1}^{+\infty} \EE_{x,s,i} \left( \tilde f(Z_t) \1_{N_t=n} \right) =  \sum_{n=1}^{+\infty}( b_{n+1} - b_n) = - b_1.
\]
Thus, we can conclude if we prove that 
\[
 \EE_{x,s,i} \left( e^{ - \gamma(t - T_{1})} \1_{T_{1} \leq t } \right) = e^{-\gamma t} \frac{e^{-\gamma s}}{  G_i\left(  \int_{-s}^0 \lambda_i (\varphi^i_u(x)) du \right)}\mathbb{E}_{\varphi_{-s}^i(x), 0, i}[ e^{\gamma S_1} \1_{s<S_1 \leq t+s}]
\]
But this is true, since by definition, the law of $S_1$ (or equivalently, $T_1$) under $\mathbb{P}_{x,s,i}$ is the  law of $S_1 - s$ under $\mathbb{P}_{\varphi^i_{-s}(x),0,i}$, conditionnaly on $S_1 > s$. 

\section{Proof of Theorem \ref{thm:Doeblin}: existence of small or petite sets under bracket conditions}    \label{ssec:proof_Doeblin} 

%\subsection{Doeblin points and petite sets}
%
%The following two definitions concern classical notions from the theory of Markov processes. 
%
%\begin{definition}
%We say that $z^* \in K$ is a \emph{Doeblin point} if there exist a neighborhood $U$ of $z^*$, a nonzero Borel measure $\nu$ on $K$, and $t^* > 0$ such that 
%$$
%P_{t^*}(z, \cdot) \geq \nu(\cdot), \quad \forall z \in U. 
%$$
%In this case, the set $U$ is called a \emph{small set} with respect to the Markov kernel $P_{t^*}$. 
%\end{definition} 
%
%\begin{definition} 
%The \emph{resolvent} associated with $(P_t)$ is the Markov kernel $R$ on $K$ defined by 
%$$
%R(z, \cdot) := \int_{\R_+} e^{-t} P_t(z, \cdot) \ dt, \quad z \in K. 
%$$
%We call a set $U \in \Bc(K)$ a \emph{petite set} if there exists a nonzero Borel measure $\nu$ on $K$ such that 
%$$
%R(z, \cdot) \geq \nu(\cdot), \quad \forall z \in U. 
%$$
%\end{definition} 
%
%Our goal in this section is to formulate conditions under which $(P_t)$ admits Doeblin points or petite sets. To this end, we introduce the following definitions:
%
%\begin{definition}
%Let $m \geq 1$, and let $\mu$ be a finite measure on $\R_+^m$. Let $\mu^{\textrm{ac}}$ be the absolutely continuous part of $\mu$ with respect to $l^+_{m}$, the Lebesgue measure on $\mathbb{R}_+^m$, and set $h := d \mu^{\textrm{ac}}/dl^+_m$. We say that a point $\tbf_0 \in \R_+^m$ is \emph{regular} for $\mu$ if $h$ is bounded below by a positive constant on a neighbourhood of $\tbf_0$. 
%\end{definition} 

For $T > 0$ and $m \in \N$, let 
\begin{equation}    \label{eq:def_Delta_T_m} 
D_m^T := \left\{\sbf \in (0,\infty)^m: s_1 + \ldots +s_m < T \right\}. 
\end{equation}

\begin{definition}   \label{def:admissible_reg_sub} 
We say that a point $z=(x,s,i) \in K$ has an \emph{admissible regular submersion} if there exist $m \in \N$ and an admissible control sequence $((s_1, \ldots, s_{m+1}, s_{m+2}), (i_1, \ldots, i_{m+1}, i_{m+2}))$ with respect to $(x,s,i)$ such that the following holds: 
\begin{enumerate}
\item Setting $T := s_1 + \ldots + s_{m+1}$, the map 
\begin{equation}    \label{eq:definition_Psi} 
\Psi: D^T_m \to \R^d, \ (v_1, \ldots, v_m) \mapsto \varphi_{T-(v_1 + \ldots + v_m)}^{i_{m+1}} \circ \Phi_{(v_1, \ldots, v_m)}^{(i_1, \ldots, i_m)}(x)
\end{equation} 
is a submersion at $(s_1, \ldots, s_m)$, i.e., the Jacobian matrix $D\Psi(s_1, \ldots, s_m)$ has full rank;
\item The point $s_1$ is $\mu_z$-regular and for all $k= 2, \ldots, m+2$ the point $s_k$ is $\mu_{x_k}^{i_k}$-regular, where $x_k$ is defined in \eqref{eq:xk}.
\end{enumerate}
\end{definition}

Note that for classical PDMPs with exponential switching times, point 2 is always satisfied. 
The following proposition is one of the main ingredients in the proof of Theorem~\ref{thm:Doeblin}. Its proof is postponed to Section~\ref{ssec:proofsubmersion}.

\begin{proposition} \label{prop:submersion}
Suppose that $z \in K$ has an admissible regular submersion with admissible control sequence $((s_1, \ldots, s_{m+2}), (i_1, \ldots, i_{m+2}))$ and $T := s_1 + \ldots + s_{m+1}$. Suppose also that $\Psi$ is defined as in~\eqref{eq:definition_Psi}. Then there exist $\varepsilon, c > 0$, a neighbourhood $U$ of $z$ in $K$ and a neighbourhood $V$ of $(\varphi^{i_{m+2}}_{\varepsilon} \circ \Psi(s_1, \ldots, s_m),\varepsilon)$ in $K^{i_{m+2}}$ such that, for all $z' \in U$, 
\[
P_{T + \varepsilon} (z', A \times \{i_{m+2}\}) \geq c \ \bm{\lambda}_{d+1}(A \cap V), \quad \forall A \in \Bc(K^{i_{m+2}}).  
\]
In particular, $z$ is a Doeblin point and $U$ is a small set. 
\end{proposition}

As we shall see, the notion of an admissible regular submersion is connected with the strong bracket condition. There is also a variant for the weak bracket condition, which we call admissible regular \emph{weak} submersion.

\begin{definition}
We say that a point $z=(x,s,i) \in K$ has an \emph{admissible regular weak submersion} if there exist $m \in \N$ and an admissible control sequence $((s_1, \ldots, s_{m+1}), (i_1, \ldots, i_{m+1}))$ with respect to $(x,s,i)$ such that the following holds: 
\begin{enumerate}
\item The map 
\begin{equation}    \label{eq:def_weak_Psi} 
\Psi: (0,\infty)^m \to \R^d, \ (v_1, \ldots, v_m) \mapsto \Phi_{(v_1, \ldots, v_m)}^{(i_1, \ldots, i_m)}(x)
\end{equation} 
is a submersion at $(s_1, \ldots, s_m)$, i.e., the Jacobian matrix $D\Psi(s_1, \ldots, s_m)$ has full rank;
\item The point $s_1$ is $\mu_z$-regular and for all $k= 2, \ldots, m+1$ the point $s_k$ is $\mu_{x_k}^{i_k}$-regular.
\end{enumerate}
\end{definition}

\begin{proposition}
\label{prop:weaksubmersion}
Suppose that $z \in K$ has an admissible regular weak submersion with admissible control sequence $((s_1, \ldots, s_{m+1}), (i_1, \ldots, i_{m+1}))$ and let $\Psi$ be defined as in~\eqref{eq:def_weak_Psi}. Then there exist $c > 0$, a neighbourhood $U$ of $z$ in $K$ and a neighbourhood $V$ of $(\Psi(s_1, \ldots, s_m),0)$ in $K^{i_{m+1}}$ such that, for all $z' \in U$,
\[
R(z', A \times \{i_{m+1}\}) \geq c \ \bm{\lambda}_{d+1}(A \cap V), \quad \forall A \in \Bc(K^{i_{m+1}}).  
\]
In particular, $U$ is a petite set for $(P_t)_{t \geq 0}$. 
\end{proposition}

We now prove point 2 of Theorem~\ref{thm:Doeblin}. We omit the proof of point 1 as it only requires small changes to the proof of point 2; most importantly, Theorem~4 from~\cite{BH12} and Proposition~\ref{prop:submersion} need to be replaced with Theorem~5 from~\cite{BH12} and Proposition~\ref{prop:weaksubmersion}, respectively. 

By Theorem 4 in \cite{BH12}, which is derived from results in~\cite[Chapter 3]{J97}, the strong bracket condition at a point $x^* \in \R^d$ implies that for all $i,j \in E$, there exist $m > d$ and $\mathbf{i} \in E^{m+1}$ with $i_1 = i$ and $i_{m+1} =j$ such that the following holds: For every $T > 0$ there exists a point $(t_1, \ldots, t_m) \in D^T_m$ such that the map
\[
\Psi: D^T_m \to \R^d, \ (v_1, \ldots, v_m) \mapsto \varphi_{T-(v_1 + \ldots + v_m)}^{i_{m+1}} \circ \Phi_{(v_1, \ldots, v_m)}^{(i_1, \ldots, i_m)}(x^*)
\]
is a submersion at $(t_1, \ldots, t_m)$. Our goal is to prove that under Assumption~\ref{cond:regular_point} (resp. under Assumption~\ref{cond:analyticity}), one can construct from $\Psi$ an admissible regular submersion at $(x^*,0,i)$. Point 2 of Theorem \ref{thm:Doeblin} will then be a direct consequence of Proposition~\ref{prop:submersion}.

Let $i,j \in E$ and let $m$ and $\ibf$ be as above. We first construct an admissible regular submersion under Assumption~\ref{cond:regular_point}. Under this assumption, for every $k \in E$, there are $\mu^k$-regular points arbitrarily close to $0$. This feature allows one to construct the admissible regular submersion in essentially the same way as for PDMPs with exponential switching times, i.e., by choosing $T$ small enough and by adding small jumps: 

By irreducibility of $Q(x^*)$ and continuity of $y \mapsto Q(y)$, there exist $\delta > 0$ and $\tilde{\ibf} \in E^{n+1}$ for some $n \geq m$, such that $(i)$ for all $k =1, \ldots, n$ and $y \in \mathbb{R}^d$ with $\|x^*-y\| < \delta$, $q_{\tilde{i}_k, \tilde{i}_{k+1}}(y) > 0$, and $(ii)$ there is an increasing sequence $r: [\![1, m+1 ]\!] \to [\![1, n+1 ]\!]$ with $r(1) = 1$, $r(m+1)=n+1$ and $\tilde{i}_{r(k)} = i_k$ for all $k \in [\![1, m+1]\!]$. Moreover, by continuity of the flows $\varphi^k$, there exists $\varepsilon_1 > 0$ such that for all $(u_1, \ldots, u_n) \in (0, \varepsilon_1)^n$, one has $\|\Phi^{(\tilde i_1, \ldots, \tilde i_k)}_{(u_1, \ldots, u_k)}(x^*) - x^* \| < \delta$, for all $k \in [\![ 1, n ]\!]$.

%Since the functions $(\lambda_i)_{i \in E}$ are bounded, one has 
%$$
%\lambda_{\max} := \max_{i \in E} \sup_{x \in \R^d} \lambda_i(x) < \infty. 
%$$
Let $T \leq \min\{\frac{\varepsilon_0}{ \lambda_{\max}}; \varepsilon_1\}$, where $\varepsilon_0$ was introduced in Assumption~\ref{cond:regular_point}. The inequality $T \leq \varepsilon_0/ \lambda_{\max}$ implies that for every $t \in (0,T)$ and $(x,k) \in \R^d \times E$, $t$ is $\mu_x^k$-regular and, in particular, $t \in \supp \mu_x^k$. On the closure of $D^T_n$, define 
\begin{equation}   \label{eq:def_tilde_Psi} 
\tilde \Psi: \cl(D^T_n) \to \R^d, \ (v_1, \ldots, v_n) \mapsto \varphi_{T-(v_1 + \ldots + v_n)}^{\tilde i_{n+1}} \circ \Phi_{(v_1, \ldots, v_n)}^{(\tilde i_1, \ldots, \tilde i_n)}(x^*).
\end{equation} 
Note that by construction, there exists a permutation $\sigma$ on $\{1, \ldots, n\}$ such that for 
$$
a: \R^n \to \R^n, \ (v_1, \ldots, v_n) \mapsto (v_{\sigma(1)}, \ldots, v_{\sigma(n)}), 
$$
one has 
$$
\Psi(v_1, \ldots, v_m) = \tilde \Psi(a(v_1, \ldots, v_m, 0, \ldots, 0)), \quad \forall (v_1, \ldots, v_m) \in D^T_m. 
$$
Recall from the beginning of the argument that there exists $(t_1, \ldots, t_m) \in D^T_m$ such that $\Psi$ is a submersion at $(t_1, \ldots, t_m)$. Thus, by continuity of the determinant and of the partial derivatives of $\tilde \Psi$, there exist $t_{m+1}, \ldots, t_n > 0$ such that
$$
a(t_1, \ldots, t_m, t_{m+1}, \ldots, t_n) \in D^T_n
$$
and $\tilde \Psi$ is a submersion at $a(t_1, \ldots, t_m, t_{m+1}, \ldots, t_n)$. By construction, the sequence 
$$
((a(t_1, \ldots, t_m, t_{m+1}, \ldots, t_n), T - \sum_{k=1}^n t_k), (\tilde i_1, \ldots, \tilde i_{n+1}))
$$
is an admissible control sequence with respect to $(x^*, 0, i)$ (see Definition~\ref{def:admissible_cs}). It suffices to add some $(t_{n+2}, \tilde i_{n+2})$ in an admissible regular manner to conclude that $(x^*, 0, i)$ has an admissible regular submersion. 

\medskip

We now prove the existence of an admissible regular submersion at $(x^*, 0, i)$ under Assumption~\ref{cond:analyticity}. The difficulty here is that we cannot add arbitrarily small jumps in an admissible manner. The jumps we need to add may be long, and thus the continuity of the determinant is not sufficient for the submersion property to carry over. This is where analyticity of the vector fields enters the picture: We can transfer the fact that the determinant for a suitable submatrix of the Jacobi matrix $D \Psi$ is not zero at a particular point to the fact that it cannot be identically zero on any given open set. Let us make the argument precise. 
 
Since, under Assumption~\ref{cond:analyticity}, $Q$ is globally irreducible, one can find $\tilde{\ibf} \in E^{n+1}$ for some $n \geq m$ such that $(i)$ for all $k =1, \ldots, n$ and $y \in \mathbb{R}^d$, $q_{\tilde{i}_k, \tilde{i}_{k+1}}(y) > 0$, and $(ii)$ there is an increasing sequence $r: [\![ 1, m+1 ]\!] \to [\![  1, n+1 ]\!]$ with $r(1) = 1$, $r(m+1)=n+1$ and $\tilde{i}_{r(k)} = i_k$ for all $k \in [\![ 1, m+1 ]\!]$. This is a considerably stronger statement than the corresponding one under Assumption~\ref{cond:regular_point}, where positive transition probabilities could only be guaranteed for points $y$ close to $x^*$.

\begin{lemma}
\label{lem:regularmuye}
Assume that for every $e \in E$, there exists a $\mu^e$-regular point $t^r_e > 0$. For $(y, e) \in \R^d \times E$, let
\[
t^r_{y,e} = \inf\{ t \geq 0 \: : \int_0^t \lambda^e(\varphi_u^e(y)) \ du > t^r_e \}.
\]
Set $\tilde x_1 = x^*$, $\tilde t_1 =  t^r_{\tilde x_1, \tilde i_1}$, and for $k =1, \ldots, n$ set $\tilde x_{k+1} = \varphi_{\tilde{t}_k}^{\tilde i_k}(\tilde{x}_k)$ and $\tilde t_{k+1} = t^r_{\tilde{x}_{k+1}, \tilde i_{k+1}}$. Given any sequence $(\hat t_1, \ldots, \hat t_{n+1}) \in \R_+^{n+1}$, define $(\hat x_1, \ldots, \hat x_{n+1})$ recursively by setting 
$\hat x_1 = x^*$ and $\hat x_{k+1} = \varphi^{\tilde i_k}_{\hat t_k} (\hat x_k)$.
Then 
\begin{enumerate}
\item For all $e \in E$ and $y \in \R^d$, $t_{y,e}^r$ is $\mu_y^e$-regular and 
\[
\frac{t^r_e}{\lambda_{\max}} \leq t^r_{y,e} \leq \frac{t^r_e}{\lambda_{\min}}; 
\]
\item There exists $\varepsilon_1 > 0$ such that $\hat{t}_k \in (\tilde t_{k} - \varepsilon_1, \tilde t_{k} + \varepsilon_1)$ for all $k = 1, \ldots, n+1$ implies that $\hat t_k$ is $\mu_{\hat x_k}^{\tilde i_k}$-regular for all $k = 1, \ldots, n+1$.
\end{enumerate}
\end{lemma}

\begin{proof}
The chain of inequalities $t^r_e/\lambda_{\max} \leq t^r_{y,e} \leq t^r_e/\lambda_{\min}$ is an immediate consequence of the boundedness of $(\lambda^e)_{e \in E}$ and the definition of $t^r_{y,e}.$ 
Since $\int_0^{t^r_{y,e}} \lambda^e(\varphi_u^e(y)) \ du = t^r_e$ and since the absolutely continuous component of $\mu_y^e$ has density $h_y^e$ satisfying 
$$
h_y^e(t) = \lambda^e(\varphi^e_t(y)) g^e \biggl(\int_0^t \lambda^e(\varphi^e_u(y)) \ du \biggr), 
$$
with $g^e$ the density of the absolutely continuous component of $\mu^e$, the point $t^r_{y,e}$ is $\mu_y^e$-regular. In particular, for $k = 1, \ldots, n+1$, $\tilde t_k = t^r_{\tilde x_k, \tilde i_k}$ is $\mu_{\tilde x_k}^{\tilde i_k}$-regular. Point 2 then follows from our standing assumption that $(\lambda^e)$ are bounded away from zero and from the continuity of $(\lambda^e)$ and $(\varphi^e)$. 
\end{proof}

Under Assumption~\ref{cond:analyticity}, for every $e \in E$, there exists a $\mu^e$-regular point $t^r_e > 0$. With the notation of Lemma~\ref{lem:regularmuye}, set
\[
T= \frac{\varepsilon_1}{2} +\sum_{k=1}^{n+1}  \tilde t_k.
\]
Let $\tilde \Psi$ be defined as in~\eqref{eq:def_tilde_Psi}. By the same argument as under Assumption~\ref{cond:regular_point}, there exists a point $(t_1, \ldots, t_n) \in D^T_n$ at which $\tilde \Psi$ is a submersion. Without loss of generality, one can assume that 
\[
\det \left( \partial_{v_1} \tilde \Psi, \ldots, \partial_{v_d} \tilde \Psi \right) \big\vert_{(t_1, \ldots, t_n)} \neq 0.
\]
Now we make use of the assumption that $(F^e)$ are analytic. By the Cauchy--Kovalevskaya Theorem, the map $(x,t) \mapsto \varphi_t^e(x)$ is analytic for every $e \in E$. Since the determinant is also analytic, the map $\det( \partial_{v_1} \tilde \Psi, \ldots, \partial_{v_d} \tilde \Psi)$ is analytic and nonzero at the point $(t_1, \ldots, t_n)$. Hence, it is not identically zero on the open set 
\[
\mathcal{U} = \prod_{k=1}^n\left(\tilde t_k - \frac{\varepsilon_1}{2n}, \tilde t_k + \frac{\varepsilon_1}{2n}\right).
\]
Pick some $(\hat t_1, \ldots, \hat t_n)$ in $\mathcal{U}$ where $\det (\partial_{v_1} \tilde \Psi, \ldots, \partial_{v_d} \tilde \Psi)$ is nonzero, i.e., $\tilde \Psi$ is a submersion at $(\hat t_1, \ldots, \hat t_n)$. Note that by definition of $T$ and $\mathcal{U}$, we have indeed $\hat t_1 + \ldots + \hat t_n < T$. Moreover, letting $\hat t_{n+1} = T - (\hat t_1 + \ldots + \hat t_n)$, one has $\hat t_{n+1} \in (\tilde t_{n+1}, \tilde t_{n+1} + \varepsilon_1)$. Then Lemma \ref{lem:regularmuye} ensures that, for all $k = 1, \ldots, n+1$, $\hat t_k$ is a $\mu_{\hat x_k}^{\tilde i_k}$-regular point, where $\hat x_k$ is constructed from $(\hat t_1, \ldots, \hat t_{k-1})$ as explained in the lemma.  
This and the fact that $q_{\tilde i_k, \tilde i_{k+1}}(y) > 0$ for $k = 1, \ldots, n$ and for all $y \in \R^d$ imply that $((\hat t_1, \ldots, \hat t_{n+1}),(\tilde i_1, \ldots, \tilde i_{n+1}))$ is an admissible control sequence with respect to $(x^*,0,i)$. 
It suffices to add some $(\hat t_{n+2}, \tilde i_{n+2})$ in an admissible regular manner to conclude that $(x^*,0,i)$ has an admissible regular submersion. This concludes the proof of Theorem \ref{thm:Doeblin}, point 2, under Assumption~\ref{cond:analyticity}.

\subsection{Proof of Proposition \ref{prop:submersion}}
\label{ssec:proofsubmersion}
 Let $z_0 = (x_0, s_0, i_1) \in K$ be a point having an admissible regular submersion. For the reader's convenience, we recall that in this case there exist $m \in \N$ and an admissible control sequence $((s_1, \ldots, s_{m+2}), (i_1, \ldots, i_{m+2}))$ with respect to $(x_0, s_0, i_1)$ such that the following holds: 
\begin{enumerate}
\item Setting $T = s_1 + \ldots + s_{m+1}$, the map 
$$
\Psi: D^T_m \to \R^d, \ (v_1, \ldots, v_m) \mapsto \varphi_{T-(v_1 + \ldots + v_m)}^{i_{m+1}} \circ \Phi_{(v_1, \ldots, v_m)}^{(i_1, \ldots, i_m)}(x_0)
$$
is a submersion at $(s_1, \ldots, s_m)$;
\item The point $s_1$ is $\mu_{z_0}$-regular and for all $k= 2,\ldots, m+2$ the point $s_k$ is $\mu_{x_k}^{i_k}$-regular, where 
$$
x_k := \Phi_{(s_1, \ldots, s_{k-1})}^{(i_1, \ldots, i_{k-1})}(x_0). 
$$
\end{enumerate}
For $\varepsilon \geq  0$, consider the map
\begin{displaymath}
\begin{array}{lrcl}
\phi^{\varepsilon} : & D_{m+1}^{T+ \varepsilon} \times \R^d & \longrightarrow & \R^d \\
    & (v_1, \ldots, v_m, r,x) & \longmapsto & \varphi^{i_{m+2}}_r \circ \varphi^{i_{m+1}}_{T + \varepsilon - \sum_{i=1}^m v_i - r} \circ \Phi_{(v_1, \ldots, v_m)}^{(i_1, \ldots, i_m)}(x)=: \phi^{\varepsilon}_{(x,r)}(\vbf). 
\end{array}
\end{displaymath}
Note that, with $\varepsilon = 0$, we have $\phi^0_{(x_0,0)} = \Psi$. Hence, by assumption on $\Psi$ and continuity of the determinant, for every $\varepsilon > 0$ small enough, $(D \phi_{(x_0,\varepsilon)}^{\varepsilon})_{\tbf_0}$ has full rank for $\tbf_0 = (s_1, \ldots, s_m)$. 

Let $t = T + \varepsilon$ for such $\varepsilon$. With the notation from Section~\ref{sec:PDP}, define the events 
$$
B = \{\bm{I}_1 = i_2, \ldots, \bm{I}_{m+1} = i_{m+2}\} \quad \text{and} \quad N = \{T_{m+1} \leq t < T_{m+2} \} \cap B.  
$$
For a starting point $z = (x,s,i_1) \in K$, one has on the event $N$ the identity 
$$
X_t = \phi^{\varepsilon} \left(S_1, \ldots, S_m, t - \sum_{i=1}^{m+1} S_i, x \right)
$$
because $T + \varepsilon - \sum_{i=1}^m S_i - (t - \sum_{i=1}^{m+1} S_i) = S_{m+1}$. Define the random variables $\tilde S_1, \ldots, \tilde S_{m+2}$ as 
$$
\tilde S_1 := \psi_z(U_1), \quad \quad \tilde S_k := \psi^{i_k}_{\Phi_{(\tilde S_1, \ldots, \tilde S_{k-1})}^{(i_1, \ldots, i_{k-1})}(x)}(U_k), \quad 2 \leq k \leq m+2, 
$$
and notice that $\tilde S_k = S_k$, $1 \leq k \leq m+2$, on the event $B$. Define also
\[
\tilde R = t - \sum_{i=1}^{m+1} \tilde S_i, \quad \tilde S_{x,s} = (\tilde S_1, \ldots, \tilde S_m, \tilde R) \quad \text{and} \quad \tilde \phi^{\varepsilon}(v_1, \ldots, v_m, r,x) := (\phi^{\varepsilon}(v_1, \ldots, v_m, r,x), r). 
\]
Then, for $A \in \Bc(K^{i_{m+2}})$, 
\begin{equation}      \label{eq:P_t_lb_A_N} 
P_t(z, A \times \{i_{m+2}\}) \geq \mathbb{P}_z(Z_t \in A \times \{i_{m+2}\}, N) 
= \mathbb{P}_z \left(\tilde \phi^{\varepsilon} \left(\tilde S_{x,s}, x \right) \in A, N \right).  
\end{equation}
Fix a positive constant $\delta_1 < \varepsilon/(m+1)$ and define 
$$
\Delta = \left\{(r_1, \ldots, r_{m+1}) \in (0, \infty)^{m+1}: \ \max_{i = 1, \ldots, m+1} \lvert r_i - s_i \rvert < \delta_1 \right\}. 
$$
Since $t = T + \varepsilon = \sum_{i=1}^{m+1} s_i + \varepsilon$, one has $t > \sum_{i=1}^{m+1} \tilde s_i$ whenever $(\tilde s_1, \ldots, \tilde s_{m+1}) \in \Delta$. Hence, with $\mathbf{S} = (\tilde S_1, \ldots, \tilde S_{m+1})$ and keeping in mind that $\tilde S_k = S_k$ on $B$, the right-hand side of~\eqref{eq:P_t_lb_A_N} is greater than or equal to 
\begin{equation}    \label{eq:P_z_4_events}
\mathbb{P}_z \biggl(\tilde \phi^{\varepsilon} \left(\tilde S_{x,s}, x \right) \in A, \mathbf{S} \in \Delta, t < \sum_{i=1}^{m+2} \tilde S_i, B \biggr). 
\end{equation} 
On the event $\mathbf{S} \in \Delta$, if $\tilde S_{m+2} > 2 \varepsilon$ then $t < \sum_{i=1}^{m+2} \tilde S_i$. The expression in~\eqref{eq:P_z_4_events} is therefore greater than or equal to 
\begin{equation}   \label{eq:P_z_4_events_alt} 
\mathbb{P}_z \left( \tilde \phi^{\varepsilon} \left( \tilde S_{x,s}, x \right) \in A, \mathbf{S} \in \Delta, \tilde S_{m+2} > 2 \varepsilon, B \right). 
\end{equation} 
Now, 
\begin{equation}   \label{eq:P_G_identity} 
\mathbb{P}_z\left(\tilde S_{m+2} > 2 \varepsilon \big\vert \tilde \phi^{\varepsilon} \left(\tilde S_{x,s}, x \right) \in A, \mathbf{S} \in \Delta, B \right) = G_{\Phi_{(\tilde S_1, \ldots, \tilde S_{m+1})}^{(i_1, \ldots, i_{m+1})}(x)}^{i_{m+2}}(2 \varepsilon). 
\end{equation} 
Since $G^{i_{m+2}}(0) = 1$ and since $G^{i_{m+2}}$ is right-continuous, there exist $\tilde \varepsilon > 0$ and $\delta_2 > 0$ such that $G^{i_{m+2}}(\tilde \varepsilon) \geq \delta_2$. By continuity of the map $(y, \varepsilon) \mapsto \int_0^{\varepsilon} \lambda^{i_{m+2}}(\varphi_u^{i_{m+2}}(y)) du$, there exists a neighbourhood $V_1$ of $x_{m+2} := \Phi_{(s_1, \ldots,s_{m+1})}^{(i_1, \ldots, i_{m+1})}(x_0)$ such that for $\varepsilon > 0$ small enough and for every $y \in V_1$, $\int_0^{2 \varepsilon} \lambda^{i_{m+2}}(\varphi^{i_{m+2}}_u(y)) du < \tilde \varepsilon$ and thus $G_y^{i_{m+2}}(2 \varepsilon) \geq \delta_2$. The required smallness of $\varepsilon$ does not depend on $\delta_1$, which is important because we chose $\delta_1 < \varepsilon/(m+1)$ earlier in the proof. One can furthermore choose a neighbourhood $V_0$ of $x_0$ and suppose without loss of generality that $\delta_1$ is so small that, for every $x \in V_0$, on the event $\mathbf{S} \in \Delta$, $\Phi_{(\tilde S_1, \ldots, \tilde S_{m+1})}^{(i_1, \ldots, i_{m+1})}(x)$ belongs to $V_1$ and thus 
$$
G_{\Phi_{(\tilde S_1, \ldots, \tilde S_{m+1})}^{(i_1, \ldots, i_{m+1})}(x)}^{i_{m+2}}(2 \varepsilon) \geq \delta_2. 
$$
In light of~\eqref{eq:P_G_identity}, we have proven for $\varepsilon$ as above and for $z = (x,s, i_1) \in K$ with $x \in V_0$ that the expression in~\eqref{eq:P_z_4_events_alt} is greater than or equal to
$$
\delta_2 \mathbb{P}_z \left(\tilde \phi^{\varepsilon} \left( \tilde S_{x,s}, x \right) \in A, \mathbf{S} \in \Delta, B \right). 
$$
Let $\mathcal{F}_U$ be the $\sigma$-field generated by the random variables $(U_i)_{i \geq 1}$. Write 
\begin{equation}      \label{eq:double_exp_cond} 
\mathbb{P}_z \left(\tilde \phi^{\varepsilon} \left(\tilde S_{x,s}, x \right) \in A, \mathbf{S} \in \Delta, B \right) = \mathbb{E}_z \left[ \id \left(\tilde \phi^{\varepsilon}(\tilde S_{x,s}, x) \in A, \mathbf{S} \in \Delta \right) \ \mathbb{E}_z[\id_B \vert \mathcal{F}_U] \right]
\end{equation}   
and observe that  
\[
\mathbb{E}_z[ \id_B | \mathcal{F}_U ] = \prod_{j=1}^{m+1} q_{i_j,i_{j+1}} \left(\Phi_{(\tilde S_1, \ldots, \tilde S_j)}^{(i_1, \ldots, i_j)}(x) \right). 
\]
Since $((s_1, \ldots, s_{m+2}), (i_1, \ldots, i_{m+2}))$ is an admissible control sequence for $z_0 =(x_0, s_0, i_1)$, continuity of the flows and of the maps $q_{j,k}$ implies that there exist a neighbourhood $J_0$ of $x_0$ and a constant $\delta_3 > 0$ such that for every $x \in J_0$ and for every $(\tilde s_1, \ldots, \tilde s_{m+1}) \in \R_+^{m+1}$, the following implication holds:  
\[
(\tilde s_1, \ldots, \tilde s_{m+1}) \in \Delta \quad \Longrightarrow \quad \prod_{j=1}^{m+1} q_{i_j,i_{j+1}} \left(\Phi_{(\tilde s_1, \ldots, \tilde s_j)}^{(i_1, \ldots, i_j)}(x) \right) \geq \delta_3. 
\]
Hence, if $z = (x,s,i_1)$ satisfies $x \in J_0$, the expression in~\eqref{eq:double_exp_cond} is greater than or equal to
\[
\delta_3 \mathbb{P}_z\left( \tilde \phi^{\varepsilon}(\tilde S_{x,s},x) \in A,  \mathbf{S} \in \Delta\right). 
\]
So far, we have shown that for $z=(x,s,i_1)$ with $x \in J_0 \cap V_0$ and for $t = T + \varepsilon$ with $\varepsilon > 0$ as above, 
\begin{equation}
\label{eq:minorPtdelta2delta3}
P_t(z, A \times \{i_{m+2}\} ) \geq \delta_2  \delta_3 \mathbb{P}_z\left( \tilde \phi^{\varepsilon}(\tilde S_{x,s},x) \in A,  \mathbf{S} \in \Delta \right). 
\end{equation}
The law of $\mathbf{S}$ is given by $\nu^S_{x,s}$ defined as
\[
\nu^S_{x,s}(B) = \int_{\R_+} \ldots \int_{\R_+} \1_B(r_1, \ldots, r_{m+1}) \ d\mu_{\Phi^{(i_1, \ldots, i_m)}_{(r_1, \ldots, r_m)}(x)}^{i_{m+1}}(r_{m+1}) \ldots  d \mu_{x,s,i_1}(r_1).
\]
Since $((s_1, \ldots, s_{m+2}), (i_1, \ldots, i_{m+2}))$ is an admissible control sequence for $z_0=(x_0, s_0, i_1)$, $s_0 + s_1$ lies in the support of $\mu_{\varphi^{i_1}_{-s_0}(x_0)}^{i_1}$. From this it follows that $s_1$ lies in the support of $\mu_{z_0}$. In addition, for $k = 2, \ldots, m+1$, $s_k$ lies in the support of $\mu_{x_k}^{i_k}$, where $x_k := \Phi_{(s_1, \ldots, s_{k-1})}^{(i_1, \ldots, i_{k-1})}(x_0)$.

Lemma \ref{lem:supportmuz} in Subsection~\ref{ssec:F_mu_Q_implies_P} and continuity of the flows imply that there exist a neighbourhood $W$ of $(x_0,s_0)$ in $K^{i_1}$ and a constant $c_1$ such that, for every $(x,s) \in W$, 
\begin{equation*}
C(x,s) := \mathbb{P}_{x,s,i_1}\left(\mathbf{S} \in \Delta \right) = \nu^S_{x,s} \left( \prod_{i=1}^{m+1} ( s_i - \delta_1, s_i + \delta_1 )\right) \geq c_1 > 0.
\end{equation*}
For $z = (x,s,i_1)$ with $(x,s) \in W$, let $\tilde T_{x,s}$ be a random variable with law $\mathbb{P}_z(\tilde S_{x,s} \in \cdot \vert \mathbf{S} \in \Delta)$. Then 
\[
\mathbb{P}_z\left( \tilde \phi^{\varepsilon}(\tilde S_{x,s},x) \in A,  \mathbf{S} \in \Delta \right) = \mathbb{P}_z\left( \tilde \phi^{\varepsilon}(\tilde T_{x,s},x) \in A  \right) C(x,s) \geq c_1 \mathbb{P}_z \left(\tilde \phi^{\varepsilon} \left(\tilde T_{x,s}, x \right) \in A \right). 
\]
Together with~\eqref{eq:minorPtdelta2delta3}, one obtains for all $(x,s) \in W$ with $x \in J_0 \cap V_0$ the estimate 
\begin{equation}
\label{eq:minorPtdelta2delta3c1}
P_t(z, A \times \{i_{m+2}\} )  \geq c_1 \delta_2  \delta_3 \PP_z \left( \tilde \phi^{\varepsilon}(\tilde T_{x,s},x) \in A  \right). 
\end{equation}
Now we will use the assumption that $z_0$ admits an admissible regular submersion to conclude. First, we need the following strengthening of the definition of a ($\mu$-)regular point:
\begin{definition}
Let $(M,d)$ a metric space and $O$ an open subset of $M$. Let $(\nu_y)_{y \in O}$ be a family of measures on $\R^m$ for some $m \geq 1$. Let $h_y$ be the density of the absolutely continuous part of $\nu_y$. We say that $\tbf_0 \in \R^m$ is uniformly regular for $(\nu_y)_{y \in O}$ if there exist a positive constant $c$ and a neighbourhood $V$ of $\tbf_0$ such that
\[
\inf_{ \tbf \in V, y \in O} h_y(\tbf) \geq c. 
\]
\end{definition}
The following key lemma is a slight modification of Lemma 6.3 in \cite{BMZIHP}.
\begin{lemma}
\label{lem:submersion}
Let $m, d$ be two positive integers, $m \geq d$. (In what follows, our notation is adapted to the more complicated case $m > d$. The case $m=d$ can be treated by making simple modifications -- in particular, the variable $\vbf$ introduced below disappears.) Let $\mathcal{D}$ be a nonempty open subset of $\R^d \times \R \times \R^{m-d}$ and let 
\[
\phi :  \mathcal{D} \times \R^d \to \R^d, \ (\ubf, r, \vbf, x) \mapsto \phi(\ubf, r, \vbf, x) = \phi_x(\ubf,r, \vbf) = \phi_{(x,r)}(\ubf, \vbf)
\]
be a $C^1$ map. Let $(T_{x,s})_{(x,s) \in \R^d \times \R_+}$ be a family of random variables with values in $\mathcal{D}$. 
Assume that
\begin{enumerate}
\item for some $x_0 \in \R^d$ and $(\ubf_0, r_0, \vbf_0) \in \mathcal{D}$, $(D \phi_{(x_0, r_0)})_{(\ubf_0, \vbf_0)} : \R^m \to \R^d$ has full rank $d$;
\item for some $s_0 \geq 0$ and for some neighbourhood $O$ of $(x_0, s_0)$ in $\R^d \times \R_+$, the point $(\ubf_0,r_0,\vbf_0)$ is uniformly regular for the laws of $(T_{(x, s)})_{(x,s) \in O}$. 
 \end{enumerate} 
Let $\tilde \phi( \ubf, r , \vbf, x) = ( \phi(\ubf, r, \vbf, x) , r) \in \R^d \times \R$. Then, there exist a constant $c > 0$, a neighbourhood $J \times I$ of $(x_0,s_0)$ in $\R^d \times \R_+$ and a neighbourhood $I_1$ of $\tilde \phi(\ubf_0, r_0, \vbf_0, x_0)$ such that
\[
\PP\left( \tilde{\phi}(T_{(x,s)}, x) \in \cdot \right) \geq c \bm{\lambda}_{d+1}( \cdot \cap I_1), \quad \forall (x,s) \in J \times I. 
\]
\end{lemma}

\begin{proof}
Since $(D \phi_{(x_0, r_0)})_{(\ubf_0, \vbf_0)}$ has full rank $d$, we can assume without loss of generality that the first $d$ columns of $(D \phi_{(x_0, r_0)})_{(\ubf_0, \vbf_0)}$ are linearly independent. In particular, the derivative at $\ubf_0$ of the map $\psi_{x,r,\vbf} : \ubf \mapsto \phi_{(x,r)}( \ubf, \vbf)$ is invertible for $(x,r,\vbf) =(x_0, r_0, \vbf_0)$. We now define the map 
\begin{displaymath}
\begin{array}{lrcl}
f : & \mathcal{D} \times \R^d & \longrightarrow & \R^d \times \R \times \R^{m-d} \\
    & (\ubf,r,\vbf,x) & \longmapsto & (\phi_x(\ubf, r, \vbf), r, \vbf) = f_x(\ubf,r, \vbf).
\end{array}
\end{displaymath}
The Jacobian matrix of $f_{x_0}$ at the point $(\ubf_0,r_0, \vbf_0)$ is invertible: Indeed, we have
\[
(D f_{x})_{\ubf, r, \vbf} = 
\begin{pmatrix}
(D \psi_{x,r,\vbf})_{\ubf} & * & *\\
0 & 1 & 0\\
0 & 0 & I_{m-d}
\end{pmatrix}
\]
and thus $\det (D f_{x})_{\ubf, r, \vbf} = \det (D \psi_{x,r,\vbf})_{\ubf}$, which is nonzero for $x=x_0$, $r=r_0$, $\vbf = \vbf_0$ and $\ubf = \ubf_0$. Hence, by Lemma 6.2 in \cite{BMZIHP}, for any neighbourhood $W \subset \mathcal{D}$ of $(\ubf_0, r_0, \vbf_0)$ there exist a neighbourhood $J \subset \R^d$ of $x_0$, nonempty open sets $I_1 \subset \R^{d+1}$ and $I_2 \subset \R^{m-d}$, and open neighbourhoods $(W_x)_{x \in J}$ of $(\ubf_0, r_0, \vbf_0)$, contained in $W$, such that $f_x$ maps $W_x$ diffeomorphically to $I_1 \times I_2$ for every $x \in J$. 
By continuity of $\det (D f_{x})_{\ubf, r, \vbf}$, there exist $c_1 > 0$, a neighbourhood $J_1$ of $x_0$ and a neighbourhood $W_1$ of $(\ubf_0, r_0, \vbf_0)$ such that 
$$
\lvert \det (D f_x)_{\ubf, r, \vbf} \rvert^{-1} \geq c_1, \quad \forall x \in J_1, \ (\ubf, r, \vbf) \in W_1. 
$$
And since $(\ubf_0, r_0, \vbf_0)$ is  uniformly regular for the laws of $(T_{x, s})_{(x,s) \in O}$, there exist $c_2 > 0$ and a neighbourhood $W_2$ of $(\ubf_0, r_0, \vbf_0)$ such that 
$$
\inf_{(\ubf, r, \vbf) \in W_2, (x,s) \in O} h_{x,s}(\ubf, r, \vbf) \geq c_2, 
$$
where $h_{x,s}$ denotes the density for the absolutely continuous component of the law of $T_{x,s}$ with respect to $\bm{\lambda}_{m+1}$. Let $J_2$ be a neighbourhood of $x_0$ and let $I$ be a neighbourhood of $s_0$ in $\R_+$ such that $J_2 \times I \subset O$. For the neighbourhood $W := W_1 \cap W_2$ of $(\ubf_0, r_0, \vbf_0)$, there exist a neighbourhood $J_3$ of $x_0$, nonempty open sets $I_1 \subset \R^{d+1}$ and $I_2 \subset \R^{m-d}$, and open neighbourhoods $(W_x)_{x \in J_3}$ of $(\ubf_0, r_0, \vbf_0)$, contained in $W$, such that $f_x$ maps $W_x$ diffeomorphically to $I_1 \times I_2$ for every $x \in J_3$. Set $J = J_1 \cap J_2 \cap J_3$. Then, for every $(x,s) \in J \times I$ and for every $(\ubf, r, \vbf) \in W_x$, one has 
\[
h_{x,s}(\ubf,r, \vbf) | \det (D f_{x})_{\ubf, r, \vbf} |^
{-1} \geq c_2 c_1 > 0.
\]
The rest of the proof follows the lines of the proof of Lemma 6.3 in \cite{BMZIHP} and we omit it. 
\end{proof}

Let us return to the proof of Proposition~\ref{prop:submersion}. Recall that for $(x,s) \in W$ and $z = (x,s,i_1)$, $\tilde T_{x,s}$ is a random variable with law $\mathbb{P}_z(\tilde S_{x,s} \in \cdot \vert \mathbf{S} \in \Delta)$. Let us write $\tilde T_{x,s} = (U_{x,s}, V_{x,s}, R_{x,s})$, where $U_{x,s}$, $V_{x,s}$ and $R_{x,s}$ take on values in $\R^d$, $\R^{m-d}$ and $\R$, respectively. To complete the proof of Proposition~\ref{prop:submersion}, it is enough to show that $T_{x,s} := (U_{x,s}, R_{x,s}, V_{x,s})$ and $\phi(\ubf, r, \vbf, x) := \phi^{\varepsilon}(\ubf, \vbf, r, x)$ satisfy the assumptions of Lemma \ref{lem:submersion}. This is done in the next four lemmas, with the main argument in Lemma~\ref{lm:check_lemma_submersion}.

\begin{lemma}
\label{lem:density}

Let $\mu$ a probability measures on $\R_+$, $m \geq 2$, and for all $k =1, \ldots, m-1$, let $(\mu_{(r_1, \ldots, r_k)})_{(r_1, \ldots, r_k) \in \R^k_+}$ a family of probability measures on $\R_+$. We define a measure $\nu$ on $\R_+^m$ by setting, for every $B \in \mathcal{B}(\R_+^m)$,
\begin{equation}
\label{eq:nu}
\nu(B) = \int_{\R_+} \ldots \int_{\R_+} \1_B(r_1, \ldots, r_m) d\mu_{(r_1, \ldots,r_{m-1})}(r_m) \ldots d \mu_{r_1}(r_2) d \mu(r_1) 
\end{equation}
For each  $k =1, \ldots, m-1$ and $(r_1, \ldots, r_k) \in \R^k_+$, we let $h_{(r_1, \ldots, r_k)}$ be the density of $\mu^{\textrm{ac}}_{(r_1, \ldots, r_k)}$ with respect to the Lebesgue measure on $\R_+$ and $H$ be the density of $\nu^{\textrm{ac}}$. Then, 
\[
H(r_1, \ldots, r_m) = h(r_1)h_{r_1}(r_2) \ldots h_{(r_1, \ldots, r_{m-1})}(r_m)
\]
\end{lemma}
 
\begin{proof}
We prove the lemma for $m=2$, a straightforward induction leads the result for any $m$. 
Let $B \in \mathcal{B}(\R_+^2)$, then
\[
\nu(B) = \nu_1(B) + \nu_2(B) + \nu_3(B),
\]
where
\[
\nu_1(B) = \int_{B} H(r_1, r_2) d r_2 dr_1, 
\]
and 
\[
\nu_2(B) = \int_{\R_+} \int_{\R_+} \1_B(r_1,r_2) d \mu^s_{r_1}(r_2) h(r_1) d r_1
\]
and
\[
\nu_3(B) = \int_{\R_+} \int_{\R_+} \1_B(r_1,r_2) d \mu_{r_1}(r_2) d \mu^s (r_1).
\]
Thus, it suffices to show that $\nu_2$ and $\nu_3$ are singular and by unicity of the Lebesgue decomposition, we will have $\nu^{ac} = \nu_1$. We start by proving that $\nu_3$ is singular. Since $\mu^s$ is singular, there exists a borel set $F$  such that $\mu^s(F) = 0$ and $\bm{\lambda}_1( \R_+ \setminus F) = 0$. Consider the set $B = F \times \R_+$, then $\nu_3(B) = \mu^s(F) = 0$ and $\bm{\lambda}_2 (\R_+^2 \setminus B) =  0$, hence $\nu_3$ is singular with respect to $\bm{\lambda}_2$. To conclude, we show that $\nu_2$ is singular with respect to $\bm{\lambda}_2$. Since for all $r_1 \in \R_+$, $\mu_{r_1}^S$ is singular, there exists $F_{r_1}$ such that $\mu^s_{r_1}(F_{r_1}) = 0$ and $\bm{\lambda}_1(\R_+ \setminus F_{r_1}) = 0$. Set $ B = \{ (r_1, r_2) \in \R_+^ 2 \: : r_2 \in F_{r_1} \}$. Then, 
\[
\nu_2(B) = \int_{\R_+} \mu^s_{r_1}(F_{r_1}) h(r_1) d r_1 = 0,
\] 
and 
\[
\bm{\lambda}_2 ( \R_+^2 \setminus B) = \int_{\R_+} l_1( \R_+ \setminus F_{r_1}) d r_1 = 0,
\]
hence $\nu_2$ is also singular with respect to $\bm{\lambda}_2$.
\end{proof} 
For $k = 1, \ldots, m+1$, let $\Delta_k = \{ r \in \R_+ \: : | r - s_k| < \delta_1 \}$. Note that $\Delta  = \cap_{k=1}^{m+1} \Delta_k$.

\begin{lemma}
\label{lem:unifregular}
For $(x,s,i_1) \in K$, a sequence  $(i_1, \ldots, i_{m+2}) \in E^{m+1}$, $m \geq 1$ and $(r_1, \ldots, r_{m_+1}) \in \R_+^{m+1}$, let $\nu_{x,s}^{\Delta}$ the measure defined by \eqref{eq:nu} for  $\mu = \mu_{x,s}^{i_1}( \cdot \cap \Delta_1)$ and $\mu_{(r_1, \ldots, r_k)} = \mu_{\Phi^{(i_1, \ldots, i_k)}_{(r_1, \ldots, r_k)}(x)}^{i_{k+1}}( \cdot \cap \Delta_k)$.  Let $\sbf = (s_1, \ldots, s_{m+1}) \in \R_+^{m+1}$ and $z_0 = (x_0,s_0, i_1) \in K$ and set for $k=1, \ldots, m$, $x_{k+1} = \Phi^{(i_1, \ldots, i_k)}_{(s_1, \ldots, s_k)}(x_0)$. Then, if $s_1$ is regular for $\mu_{z_0}$, and for $k = 1, \ldots, m$, $s_{k+1}$ is regular for $\mu_{x_{k+1}}^{ i_{k+1}}$, then  $\sbf$ is uniformly  regular for $(\nu_{x,s}^{\Delta})_{(x,s,i_1) \in V_0}$, where $V_0$ is a neighbourhood of $(x_0, s_0, i_1)$ in $K$.   
\end{lemma}
\begin{proof}
For $(x,s,i_1) \in K$, let $h_{x,s}^{\Delta}$ the density of the absolute continuous part of $\mu$ and $h_{x,r_1, \ldots, r_k}^{\Delta}$ the density of the absolute continous part of $\mu_{(r_1, \ldots, r_k)}$. Then, it is easily checked that 
\begin{equation}
\label{eq:hx}
h_{x,s}^{\Delta}(t) = \frac{\lambda_{i_1}( \varphi_t^{i_1}(x)) g_{i_1}\left( \int_{-s}^{t}\lambda_{i_1}( \varphi_u^{i_1}(x)) du \right) }{  G_{i_1} \left(\int_{-s}^{0}\lambda_{i_1}( \varphi_u^{i_1}(x)) du \right) }\1_{\Delta_1}(t),
\end{equation}
and
\[
h_{x,r_1, \ldots, r_k}^{\Delta}(t) = g_{i_{k+1}}\left(\int_{0}^{t}\lambda_{i_{k+1}}( \varphi_u^{i_{k+1}}\circ \Phi^{(i_1, \ldots, i_k)}_{(r_1, \ldots, r_k)}(x)) du \right)\1_{\Delta_{k+1}}(t),
\]
where for all $i \in E$, $g_i$ denotes the density of the absolute continuous part of $\mu_i$. Note in particular that
\begin{equation}
\label{eq:minorh}
h^{\Delta}_{x,s}(t) \geq \lambda_{i_1}( \varphi_t^{i_1}(x)) g_{i_1}\left( \int_{-s}^{t}\lambda_{i_1}( \varphi_u^{i_1}(x)) du \right)\1_{\Delta_1}(t).
\end{equation}
  Now, since $s_1$ is regular for $\mu_{(x_0, s_0, i_1)}$, there exists a neighbourhood $W_1$ of $s_1$  and a positive constant $c_1$ such that, for all $t \in W_1$, $h_{x_0, s_0, i_1}(t) \geq c_1$. This is equivalent to
\[
\frac{\lambda_{i_1}( \varphi_t^{i_1}(x_0)) g_{i_1}\left( \int_{-s_0}^{t}\lambda_{i_1}( \varphi_u^{i_1}(x_0)) du \right)}{   G_{i_1} \left(\int_{-s_0}^{0}\lambda_{i_1}( \varphi_u^{i_1}(x_0)) du \right)} \geq c_1, 
\]
for all $t \in W_1$. Setting  $c_1' = \frac{c_1}{\lambda_{\max}}    G_{i_1} \left(\int_{-s_0}^{0}\lambda_{i_1}( \varphi_u^{i_1}(x_0)) du \right)$, the above inequality implies that for all $t \in W_1$,
\[
g_{i_1}\left( \int_{-s_0}^{t}\lambda_{i_1}( \varphi_u^{i_1}(x_0)) du \right) \geq c_1'.
\]
Let
\[
D = \{ \int_{-s_0}^{t}\lambda_{i_1}( \varphi_u^{i_1}(x_0)) du, \, t \in W_1 \}.
\]
Then, for all $v \in D$, $g_{i_1}(v) \geq c_1$. By continuity of $(x,s,t) \mapsto \int_{-s}^t \lambda_{i_1}( \varphi_u^{i_1}(x)) du$, and reducing $W_1$ if necessary, there exists a neighbourhood $V_0$ of $(x_0,s_0)$ in  $K_{i_1}$ such that, for all $(x,s,t) \in V_0 \times W_1$, $\int_{-s}^t \lambda_{i_1}( \varphi_u^{i_1}(x)) du$ lies in $D$ and thus  for all $(x,s,t) \in V_0 \times W_1$,
\[
g_{i_1}\left( \int_{-s}^{t}\lambda_{i_1}( \varphi_u^{i_1}(x)) du \right) \geq c_1'.
\]
Without loss of generality, we can assume that $W_1 \subset \Delta_1$. Hence, the above inequality and Equation \eqref{eq:minorh} yield that for all $(x,s,t) \in V_0 \times W_1$, $h_{x,s}^{\Delta}(t) \geq c_1''$, with $c_1'' = \lambda_{\min} c_1'$.
%Since $K$ is open in $\R^d \times \R_+ \times E$, there exists a neighbourhood $V_0$ of $(x_0,s_0, i_1)$ such that $\bar V_0$, the closure of $V_0$, is included in $K$. Set 
%\[
%t^* = \max_{(x,s,i_1) \in \bar V_0} \int_{-s}^{0}\lambda_{i_1}( \varphi_u^{i_1}(x)) du.
%\]
%Note that since $\bar V_0 \subset K$, $t^*   < \bar t_{i_1}$. Therefore, $\bar G_{i_1}( t^*) > 0$ and, for all $(x,s,i_1) \in V_0$ and $t \in W_1$, 
%\begin{align*}
%\lambda_{i_1}( \varphi_t^{i_1}(x_0)) g_{i_1}\left( \int_{-s_0}^{t}\lambda_{i_1}( \varphi_u^{i_1}(x_0)) du \right) & \geq c_1 \bar G_{i_1} \left(\int_{-s_0}^{0}\lambda_{i_1}( \varphi_u^{i_1}(x_0)) du \right)\\
%& \geq c_1 \bar G_{i_1}( t^*) := 2 c_1' > 0
%\end{align*}
%Hence, by continuity of the flow $(u,x) \mapsto \varphi^{i_1}_u(x)$, we may assume without loss of generality that for all $(x,s,i_1) \in V_0$ and $t \in W_1$, $h_{x,s}(t) \geq c_1'$.
 Next, since $s_2$ is a regular point for $\mu^{i_2}_{\varphi^{i_1}_{s_1}(x_0)}$, there exists a neighbourhood $W_2$ of $s_2$ and a positive constant $c_2$ such that, for all $t \in W_2$, $h_{x_0, s_1}(t) \geq c_2$, which by the same argument as above implies that for all $(x,s) \in V_0$ and $(r_1, r_2) \in W_1 \times W_2$, $h^{\Delta}_{x, r_1}(r_2) \geq  c'_2$ for some constant $c_2'$. Keeping on this procedure, we can prove that there exists a neighbourhood $V$ of $\sbf$ such that, for all $(r_1, \ldots, r_{m+1}) \in V$ and all $(x,s) \in V_0$, 
\[
h^{\Delta}_{x,s}(r_1) h_{x,s,r_1}(r_2) \ldots h^{\Delta}_{x,s,r_1, \ldots, r_{m}}(r_{m+1})  \geq c
\]
for some positive constant $c$. This concludes the proof by Lemma \ref{lem:density}. 
\end{proof}
\begin{lemma}
\label{lem:pushforward}
Let $(M,d)$ a metric space and $O$ an open subset of $M$. Let $(\nu_y)_{y \in O}$ a family of measures on $\R^m$ for some $m \geq 1$. Let $B$ an open set of $\R_m$ and $f : B \to \R^m$ a $C^1$ map such that the Jacobian matrix of $f$ is invertible in every point of $B$. Let $\nu_y \circ f^{-1}$ be the pushforward of $\nu_y$ by $f$. Then, if $\tbf \in B$ is uniformly  regular for $(\nu_y)_{y \in O}$,  $f(\tbf)$ is uniformly regular for $(\nu_y \circ f^{-1})_{y \in O}$. 
\end{lemma}

\begin{proof}
Set $\mu_y = \nu_y \circ f^{-1}$. Then $\mu_y^{ac} \geq \nu_y^{ac} \circ f^{-1}$ and therefore, it is sufficient to prove that $f(\tbf)$ is uniformly regular for $\nu_y^{ac} \circ f^{-1}$. By Lemma 3 in \cite{BH12}, we have for all $\sbf \in \R^m$,
\[
\frac{d \nu_y^{ac} \circ f^{-1}}{d \lambda_m} ( \sbf) = \sum_{ \sbf' \in B \: : f(\sbf') = \sbf} | \det Df(\sbf') |^{-1} \frac{d \nu_y^{ac}}{d \lambda_m} ( \sbf').
\] 
In particular, for all $\sbf \in B$, 
\[
\frac{d \nu_y^{ac} \circ f^{-1}}{d \lambda_m} ( f(\sbf)) \geq  | \det Df(\sbf) |^{-1} \frac{d \nu_y^{ac}}{d \lambda_m} ( \sbf).
\] 
Now, since $\tbf$ is uniformly regular for $(\nu_y)_{y \in O}$, there exists a neigbhourhood $V$ of $\tbf$ such that, for $\sbf \in V$ and all $y \in O$, 
\[
\frac{d \nu_y^{ac}}{d \lambda_m} ( \sbf) \geq c
\]
for some positive constant $c$. Next, since $f$ is $C^1$, there exists a positive constant $c'$ such that, for all $\sbf \in V$, $| \det Df(\sbf) |^{-1}  \geq c'$. This implies that for all $\sbf \in V$, 
\[
\frac{d \nu_y^{ac} \circ f^{-1}}{d \lambda_m} ( f(\sbf)) \geq c'c,
\]
which concludes the proof. 
\end{proof}
We now conclude the proof of Proposition \ref{prop:submersion}.
\begin{lemma}        \label{lm:check_lemma_submersion}
For $\varepsilon > 0$ small enough, $\tilde{T}_{x,s}$ and $\phi^{\varepsilon}$ satisfy the assumptions of Lemma \ref{lem:submersion} for $x=x_0$, $s=s_0$, $r_0 = \varepsilon$, and $\tbf_0 = (s_1, \ldots, s_m)$.
\end{lemma}
\begin{proof}
We have already proven that for $\varepsilon > 0$ small enough,  $(D \phi_{(x_0,\varepsilon)}^{\varepsilon})_{\tbf_0}$ has full rank for $\tbf_0 = (s_1, \ldots, s_m)$. We fix such a small $\varepsilon$, and we set $r_0 = \varepsilon$. Let us prove that the point $(\tbf_0, r_0)$ is uniformly regular for $(\tilde{T}_{x,s})_{(x,s) \in O}$ for some neighbourhood $O$ of $(x_0, s_0)$. Let $f : \R^{m+1} \to \R^{m+1}$ be the map defined by $f(r_1, \ldots, r_{m+1}) = ( r_1, \ldots, r_m, t - \sum_{i=1}^{m+1} r_i)$. Then $\tilde S_{x,s} =  f( \mathbf{S})$ and we have
\begin{align*}
\PP( \tilde T_{x,s} \in B) & =  \frac{\Pp_z( \tilde S_{x,s} \in B, \mathbf{S} \in \Delta)}{\Pp_z( \mathbf{S} \in \Delta)} \\ 
& = \frac{1}{C(x,s)} \Pp_z( f(\mathbf{S}) \in B, \mathbf{S} \in \Delta)\\
& = \frac{1}{C(x,s)} \nu_{x,s}^{S, \Delta} \circ f^{-1} (B),
\end{align*}
where
\[
\nu_{x,s}^{S, \Delta}( \cdot) = \nu_{x,s}^{S}( \cdot \cap \Delta).
\]

 Let $\sbf = (s_1, \ldots, s_{m+1})$ Note that $\nu_{x, s}^{\mathbf{S}, \Delta}$ can be written as in Lemma \ref{lem:density}, with $\mu= \mu_{x,s}^{i_1}(\cdot \cap \Delta_1)$ and, for all $k =1, \ldots, m$, $\mu_{(r_1, \ldots, r_k)} = \mu_{\Phi^{(i_1, \ldots, i_k)}_{(r_1, \ldots, r_k)}(x)}^{i_{k+1}}(\cdot \cap \Delta_{k+1})$.  By assumption, $s_1$ is regular for $\mu_{x_0, s_0}^{i_1}$ and $s_k$ is regular for $\mu_{x_{k+1}, i_{k+1}}$. Thus, Lemma \ref{lem:unifregular} implies that $\sbf$ is uniformly regular  for $(\nu^{\mathbf{S},\Delta}_{x,s})_{x,s \in O}$ for some neighbourhood $O$ of $(x_0, s_0)$. Furthermore, by definition of $f$; $f(\sbf) = (\tbf_0, r_0)$ (recall that $t = \sum_{i=1}^{m+1} s_i + \varepsilon$). It is also straightforward that for any $\sbf' \in \Delta$, $|\det Df(\sbf')|^{-1} = 1$. Thus, Lemma \ref{lem:pushforward} implies that   $f(\sbf) = (\tbf_0, r_0)$ is uniformly regular for $(\nu^{S, \Delta}_{x,s} \circ f^{-1})_{(x,s) \in O}$, which implies that $(\tbf_0, r_0)$ is uniformly regular  for $(\tilde{T}_{x,s})_{(x,s) \in O}$.
\end{proof}

\subsection{Proof of Proposition \ref{prop:weaksubmersion}}
The proof is similar to the proof of Theorem \ref{prop:submersion}, except that we change the deterministic terminal time $t$ by some random $T$, distributed according to an exponential law of parameter $1$. 
Let $z_0 = (x_0, s_0, i_1) \in K$ a point having an admissible regular submersion. Recall that there are $m \in \N$ and an admissible control sequence $((s_1, \ldots, s_{m+1}), (i_1, \ldots, i_{m+1}))$ with respect to $(x_0,s_0,i_1)$ such that the following holds: 
\begin{enumerate}
\item the map 
$$
\Psi: \R^m \to \R^d, \ (v_1, \ldots, v_m) \mapsto \Phi_{(v_1, \ldots, v_m )}^{(i_1, \ldots, i_m)}(x)
$$
is a submersion at $(s_1, \ldots, s_m)$, i.e. the Jacobian matrix $D\Psi(s_1, \ldots, s_m)$ has full rank;
\item The point $s_1$ is regular for $\mu_z$ and for all $k= 2 \ldots m+1$, the point $s_k$ is regular for $\mu_{x_k,i_k}$, where $x_k$ is defined in \eqref{eq:xk}.
\end{enumerate}
In order to apply Lemma \ref{lem:submersion}, we introduce the map
$$
\phi: \R^m \times \R \times \R^d \to \R^d, \ (v_1, \ldots, v_m, v_{m+1}, x) \mapsto \Phi_{(v_1, \ldots, v_m, v_{m+1})}^{(i_1, \ldots, i_m, i_{m+1})}(x).
$$
Note that $\phi_{x_0, 0} = \Psi$, hence $(D\phi_{x_0, 0})_{\mathbf{t}_0}$ has full rank for $\mathbf{t}_0 = (s_1, \ldots, s_m)$. 
For $t > 0$, let $\mathbf{E}_t$ the event  that the process has jumped exactly $m$ times before time $t$ and that the sequence of postjump locations $(B_1, \ldots, B_{m})$ is exactly equal to $(i_2, \ldots, i_{m+1})$. Then, on the event $\mathbf{E}_t$, one has
\[
(X_t, \tau_t, I_t) = \left( \varphi^{i_{m+1}}_{t - \sum_{i=1}^{m} S_i} \circ  \Phi_{(S_1, \ldots, S_m)}^{(i_1, \ldots, i_m)}(x), t - \sum_{i=1}^{m} S_i, i_{m+1} \right) 
\]
In particular, for $\tilde \phi ( r_1, \ldots, r_m, r_{m+1},x )  = (\phi(r_1, \ldots, r_m, r_{m+1},x), r_{m+1})$, one has $(X_t, \tau_t) = \tilde \phi( S_1, \ldots, S_m, t -  \sum_{i=1}^{m} S_i,x)$. Hence,
\begin{align*}
R((x,s,i), A \times \{i_{m+1}\} & =
\int_0^{+\infty} \mathbb{P}_{x,s,i} (Z_t \in A \times \{i_{m+1}\}) e^{-t} dt \\
& \geq \int_0^{+\infty} \mathbb{P}_{x,s,i} (Z_t \in A \times \{i_{m+1}\}; \mathbf{E}_t) e^{-t} dt \\
& =\int_0^{+\infty} \mathbb{P}_{x,s,i} ( \tilde \phi( S_1, \ldots, S_m, t -  \sum_{i=1}^{m} S_i,x) \in A;  \mathbf{E}_t) e^{-t} dt
\end{align*}
With the notations and reasoning of the proof of Theorem \ref{prop:submersion}, we can thus prove that there exists $\delta_1, \delta_2, \delta_3, \varepsilon > 0$ such that $m \delta_1 < \varepsilon$ and 
\[
R((x,s,i), A \times \{i_{m+1}\} \geq \delta_2 \delta_3 \int_{t_m + m \delta_1}^{t_m + \varepsilon} \mathbb{P}_{x,s,i} ( \tilde \phi( \tilde S_1, \ldots, \tilde S_m, t -  \sum_{i=1}^{m} S_i,x) \in A;  \max_{i = 1, \ldots, m} | \tilde S_i - s_i| < \delta_1 ) e^{-t} dt
\]
where $t_m = \sum_{i=1}^m s_i$
Now, let $T$ be a random variable, independant of the $(U_i)_{i \geq 1}$ and the $(V_i)_{i \geq 1}$, with an exponential law of parameter $1$. Then, we have 
\begin{multline*}
\int_{t_m + m \delta_1}^{t_m + \varepsilon} \mathbb{P}_{x,s,i} ( \tilde \phi( \tilde S_1, \ldots, \tilde S_m, t -  \sum_{i=1}^{m} S_i,x) \in A;  \max_{i = 1, \ldots, m} | \tilde S_i - s_i| < \delta_1 ) e^{-t} dt = \\  \PP ( \tilde \phi( \tilde S_1, \ldots, \tilde S_m, T -  \sum_{i=1}^{m} S_i,x) \in A;  \max_{i = 1, \ldots, m} | \tilde S_i - s_i| < \delta_1; T \in [t_m + m \delta_1, t_m + \varepsilon]).
\end{multline*}
Let $\mathbf{S} = (\tilde S_1, \ldots, \tilde S_m)$,  $\tilde R_{x,s} = T - \sum_{i=1}^{m} S_i$ and set $\tilde S_{x,s} = (\mathbf{S},\tilde R_{x,s})$. Let $  g : \R^{m} \times \R \to \R^{m+1}$ the map defined by $  g(r_1, \ldots,r_m, t) = ( r_1, \ldots, r_m, t - \sum_{i=1}^{m} r_i)$. Then, $\tilde S_{x,s} =   g( \mathbf{S},T)$. We also introduce the set $\Delta$ defined by
\[
\Delta = \{ (r_1, \ldots, r_{m}) \in \R_+^{m} \: : \max_{i=1 \ldots, m} | r_i - s_i| < \delta_1 \} \times [t_m + m \delta_1, t_m + \varepsilon]
\]
and we let $g$ be the restriction of $  g$ to $\Delta$. 
Now, let $\tilde T_{x,s}=(T_{x,s}, R_{x,s})$ be a random variable with the law of $\tilde S_{x,s}$ contionned on the event that $(\mathbf{S},T) \in \Delta$. Then, for all Borel set $B$, we have

\[
\PP( \tilde T_{x,s} \in B)  = \frac{1}{C(x,s)} (\nu_{x,s}^S \otimes \xi )\circ g^{-1} (B),
\]
where $\nu_{x,s}^S$ is the law of $\mathbf{S}$ (given in the proof of Theorem \ref{prop:submersion}, with $m+1$ instead of $m$), $\xi$ is the law of $T$ and $C(x,s)$ is the probability that $(\mathbf{S},T) \in \Delta$. Now, Lemma \ref{lem:unifregular} implies that $\mathbf{t}_0$ is uniformly regular for $(\nu_{x,s}^S)_{(x,s)\in O}$, for some  neighbourhood $O$ of $(x_0, s_0)$ hence $(\tbf_0, t_m)$ is uniformly regular for $(\nu_{x,s}^S \otimes \xi)_{(x,s)\in O}$ since every point is regular for $\xi$. Moreover, $g(\tbf_0,t_m)= (\tbf_0, 0)$, thus Lemma \ref{lem:pushforward} implies that $(\tbf_0, 0)$ is uniformly regular for $(\tilde T_{x,s})_{(x,s)\in O}$. This concludes the proof of Proposition \ref{prop:weaksubmersion} by Lemma \ref{lem:submersion}.

\section{Proof of Proposition~\ref{prop:ac_preserve}: absolute continuity of the invariant measure}    \label{ssec:proof_ac} 
The proof of Proposition~\ref{prop:ac_preserve} relies on the following lemma. 

\begin{lemma}
\label{lem:psiC1}
Suppose that the assumptions of Proposition~\ref{prop:ac_preserve} hold. Let 
\[
\Xi = \{ {g}_n^i(a_n^i): \:  i \in E, n \in N^i \}
\]
and $\Xi^{\mathrm{c}} = (0,1) \setminus \Xi$.
 Then, for all $i \in E$, the map $(x,u) \mapsto \psi^i_x(u)$ is $C^1$ on $\R^d \times \Xi^{\mathrm{c}}$. Moreover, $\partial_u \psi_x^i(u) < 0$ for every $(x,u) \in \R^d \times \Xi^{\mathrm{c}}$.
\end{lemma}

\bpf
Fix $i \in E$, and for $(x,t) \in \mathbb{R}^d \times [0, +\infty)$, let 
\begin{equation}     \label{eq:def_gamma} 
\gamma(x,t) = \gamma_x(t) = \int_0^t \lambda^i( \varphi^i_u(x)) du.
\end{equation} 
Then, by definition, $  G_x^i =   G^i \circ \gamma_x$. Moreover, since $\lambda^i \geq \lambda_{\min}$, $\gamma_x :(0,+\infty) \to (0,+\infty)$ is invertible. Set $a_{n,x}^i = \gamma_x^{-1}(a_n^i)$, $b_{n,x}^i = \gamma_x^{-1}(b_n^i)$ and $  g_{n,x}^i =   g_n^i \circ \gamma_x$. Then it is easily seen that $  G^i_x$ satisfies Assumption \ref{hyp:barGC1} with $a_{n,x}^i$, $b_{n,x}^i$ and $  g_{n,x}^i$. Hence, on $(a_{n,x}^i, b_{n,x}^i)$, $  G^i_x$ is invertible, with continuously differentiable inverse $(  g^i_{n,x})^{-1} = \gamma_x^{-1} \circ (  g_n^i)^{-1}$. Note that $  g^i_{n,x}(a^i_{n,x}) =   g_n^i(a_n^i)$ and $  g^i_{n,x}(b^i_{n,x}) =   g_n^i(b_n^i)$. Note also that by definition of $a_n^i$ and $b_n^i$, we have $  g_n^i(b_n^i) =   g_{n+1}^i(a_{n+1}^i)$. Therefore, $\psi_x^i$ is equal to $(  g_{n,x}^i)^{-1}$ on $(  g_n^i(b_n^i),   g_n^i(a_n^i))$ and jumps at $  g_n^i(a_n^i)$ from $a_{n,x}^i$ to $b_{n-1,x}^i$ with the convention $b_{0,x}^i = 0$. In particular, $\psi^i_x$ is continuously differentiable on $\Xi^{\mathrm{c}}$, with negative derivative because $  g_n^i$ has negative derivative.

To conclude the proof of the lemma, we show that $(x,u) \mapsto \psi^i_{x}(u)$ is $C^1$ on $\R^d \times \Xi^{\mathrm{c}}$. Given a point $(x_0, u_0) \in \R^d \times \Xi^{\mathrm{c}}$, there exists $n \in N^i$ such that $  g_n^i(b_n^i) < u_0 <   g_n^i(a_n^i)$. Hence we can choose a neighbourhood $V$ of $(x_0, u_0)$ such that for all $(x,u) \in V$, $  g_n^i(b_n^i) < u <   g_n^i(a_n^i)$. Therefore, for all $(x,u) \in V$,
\[
\psi^i_x(u)  =  (  g_{n,x}^i)^{-1}(u) = \gamma_x^{-1} \circ (  g_n^i)^{-1}(u) = \eta(x, (  g_n^i)^{-1}(u)), 
\]
where $\eta: (x,t) \mapsto \gamma_x^{-1}(t)$. Since $(  g_n^i)^{-1}$ is $C^1$, it suffices to show that $\eta$ is $C^1$ on $\R^d \times (0, + \infty)$.  Let $V_1 := \R^d \times (0,\infty)$ and $V_2 := (0, \infty)$, and define $h: V_1 \times V_2 \to \R$ as 
$$
h((x,u),t) := \gamma(x,t) - u. 
$$
Since  $\lambda^i$ is $C^1$ by assumption, $h \in C^1(V_1 \times V_2)$ as is required to invoke the implicit function theorem. Moreover, for all $(x, u) \in V_1$, one has $ \gamma(x, \gamma_{x}^{-1}(u)) = u$. If we fix $(x^*, u^*) \in \R^d \times (0, \infty)$ and set $t^* := \gamma_{x^*}^{-1}(u^*)$, one has $h((x^*, u^*), t^*) = 0$. 
Moreover, 
$$
\partial_t h((x,u), t) =  \lambda_i(\varphi^i_t(x)) > 0. 
$$
The implicit function theorem implies that there exists a $C^1$ map $\theta$ defined on a neighbourhood $V \subset V_1$ of $(x^*, u^*)$ as well as a neighbourhood $W \subset V_1 \times V_2$ of $((x^*, u^*), t^*)$ such that 
$$
\{((x,u),t) \in W: \ h((x,u),t) = 0\} = \{((x,u),t) \in V_1 \times V_2: (x,u) \in V, \theta(x,u) = t\}. 
$$
This means that for all $(x,u) \in V$, $\gamma_x(\theta(x,u)) = u$ and thus 
$$
\eta(x,u) = \gamma^{-1}_x(u)=\gamma^{-1}_x(\gamma_x(\theta(x,u))) = \theta(x,u), \quad \forall (x,u) \in V. 
$$
Thus, $\eta$ is $C^1$ and the lemma is proved.
\epf

\begin{remark}\label{rm:snon0}    \rm 
For $z = (x,s,i) \in K$ and $t \geq 0$, let $\gamma_z(t) := \gamma(\varphi^i_{-s}(x), t+s)$, with $\gamma$ defined as in~\eqref{eq:def_gamma}. The proof of Lemma \ref{lem:psiC1} can easily be adapted to prove that, under the assumptions of Proposition~\ref{prop:ac_preserve}, for all $z \in K$, the map $\psi_z$ is $C^1$ on $\Xi^{\mathrm{c}}_z$, where 
$$
\Xi_z = \{  g_n^i(a^i_n) /   G_i \circ \gamma_z(0): \: n \in N^i\}. 
$$
It suffices to replace $\gamma_x$ and $  g^i_{n,x}$ with $\gamma_z$ and $  g_{n,z} =   g^i_n \circ \gamma_z /   G_i \circ \gamma_z(0)$, respectively.  
\end{remark}

In several places throughout the proof of Proposition~\ref{prop:ac_preserve}, we will use the following fact from measure theory (see, e.g., Proposition 4.4 in~\cite{Davydov}). 

\begin{lemma}      \label{lm:Davydov} 
Let $m \in \N$, let $B \subset \R^m$ be open, and let $f: B \to \R^m$ be a Borel measurable function such that for $\bm{\lambda}_m$-almost every $\tbf \in B$, $f$ is differentiable at $\tbf$ and $\det Df(\tbf) \neq 0$. If $\nu$ is a $\sigma$-finite Borel measure on $\R^m$ such that $\nu \ll \bm{\lambda}_m$, then $\nu f^{-1} \ll \bm{\lambda}_m$, where $\nu f^{-1}$ denotes the pushforward of $\nu$ under $f$. 
\end{lemma} 

Let $\mu$ be a finite measure on $(K, \Bc(K))$ such that $\mu \ll \bm{L}$ and let $t > 0$. Let $A \in \Bc(K)$ such that $\bm{L}(A) = 0$. One has  
\begin{align} \label{eq:mu_P_t_expansion} 
\mu P_t(A) =& \int_K P_t(z,A) \ \mu(dz) = \sum_{m=1}^{\infty} \int_K \Pp_z(Z_t \in A, T_{m-1} \leq t < T_m) \ \mu(dz) \\
=& \int_K \id_A(\varphi_t^i(x), s+t, i) \Pp_z(t < T_1) \ \mu(dz) + \sum_{m \geq 2} \int_K \Pp_z(Z_t \in A, T_{m-1} \leq t < T_m) \ \mu(dz). \notag 
\end{align}
In order to show $\mu P_t(A) = 0$ and thus $\mu P_t \ll \bm{L}$, it is then enough to prove the following statements: 
\begin{enumerate}[{\bf (i)}] 
\item $$ \int_K \id_A(\varphi_t^i(x), s+t, i) \Pp_z(t < T_1) \ \mu(dz) = 0; $$
\item $$ \int_K \Pp_z(Z_t \in A, T_{m-1} \leq t < T_m) \ \mu(dz) = 0, \quad \forall m \geq 2. $$
\end{enumerate} 
Let us first show claim (i), which requires neither Assumption~\ref{hyp:barGC1} nor differentiability of $(\lambda_i)_{i \in E}$. Since $\Pp_z(t < T_1) \leq 1$ for all $z$, the integral on the left is dominated by 
$$
 \int_K \id_A(\varphi_t^i(x), s+t, i) \ \mu(dz) = \mu(\{(x,s,i)  \in K: \ (\varphi_t^i(x), s+t, i) \in  A\}).
$$  
Since $\mu \ll \bm{L}$, the expression on the right-hand side vanishes if
\begin{equation}   \label{eq:psi_null} 
\bm{L}(\{(x,s,i) \in K: \ (\varphi_t^i(x), s+t, i) \in A\}) = 0. 
\end{equation} 
There is no loss of generality in assuming that $A = \tilde A \times \{j\}$ for some $j \in E$ and $\tilde A \in \Bc(K_j)$. Then the left-hand side of~\eqref{eq:psi_null} is dominated by $\bm{\lambda}_{d+1} f^{-1}(\tilde A)$, where $\bm{\lambda}_{d+1} f^{-1}$ denotes the pushforward of $\bm{\lambda}_{d+1}$ under the map
$$
f: \R^d \times \R \to \R^d \times \R, \ (x,s) \mapsto (\varphi_t^j(x), s+t). 
$$
One has 
$$
D f(x,s) = \begin{pmatrix}
                 D \varphi_t^j(x) & 0 \\
                 0 & 1 
                 \end{pmatrix}, 
$$
so 
$$
\det D f(x,s) = \det D \varphi_t^j(x) \neq 0 
$$
as $\varphi_t^j$ is a diffeomorphism. By Lemma~\ref{lm:Davydov}, $\bm{\lambda}_{d+1} f^{-1} \ll \bm{\lambda}_{d+1}$, which implies $\bm{\lambda}_{d+1} f^{-1}(\tilde A) = 0$. This completes the proof of~\eqref{eq:psi_null} and thus of claim (i). 

\medskip

Now we show claim (ii). Fix $m \geq 2$. The integral on the left-hand side of (ii) can be written as  
$$
\sum_{(i_2, \ldots, i_m) \in E^{m-1}} \int_K \Pp_z(Z_t \in A, T_{m-1} \leq t < T_m, B(i_2, \ldots, i_m)) \ \mu(dz), 
$$
where 
$$
B(i_2, \ldots, i_m) := \{B_1 = i_2, \ldots, B_{m-1} = i_m\}. 
$$
It is therefore enough to show 
\begin{equation}     \label{eq:m_gto_i_exp} 
\int_K \Pp_z(Z_t \in A, T_{m-1} \leq t < T_m, B(i_2, \ldots, i_m)) \ \mu(dz) = 0
\end{equation} 
for $(i_2, \ldots, i_m) \in E^{m-1}$ fixed. 

For $z = (x,s,i) \in K$, we define the random variables $\tilde S_1^z, \ldots, \tilde S^z_{m-1}$ as in the proof of Proposition~\ref{prop:acces}, the only difference being that we make their dependence on $z$ now explicit in the notation: 
$$
\tilde S_1^z := \psi_z(U_1), \quad \tilde S_k^z := \psi^{i_k}_{\Phi_{(\tilde S_1^z, \ldots, \tilde S_{k-1}^z)}^{(i, i_2, \ldots, i_{k-1})}(x)}(U_k), \quad 2 \leq k \leq m-1. 
$$
Let $\nu_z$ denote the law of $(\tilde S_1^z, \ldots, \tilde S_{m-1}^z)$. 

\begin{lemma}       \label{lm:ac_nu_z} 
Under the assumptions of Proposition~\ref{prop:ac_preserve}, one has $\nu_z \ll \bm{\lambda}_{m-1}$ for every $z \in K$. 
\end{lemma} 

We postpone the proof of Lemma~\ref{lm:ac_nu_z} to the end of this section. Assuming as in the proof of claim (i) that $A = \tilde A \times \{j\}$ for $j \in E$ and $\tilde A \in \Bc(K_j)$, we have 
\begin{equation}       \label{eq:P_J_M_B}
\Pp_z(Z_t \in A, T_{m-1} \leq t < T_m, B(i_2, \ldots, i_m)) \leq \nu_z(J_{x,i}), 
\end{equation}
where 
$$
J_{x,i} := \{\sbf \in \R_+^{m-1}: \ (\varphi^j_{t-(s_1 + \ldots + s_{m-1})}(\Phi_{(s_1, \ldots, s_{m-1})}^{(i, i_2, \ldots, i_{m-1})}(x)), t - (s_1 + \ldots + s_{m-1})) \in \tilde A\}. 
$$
We will now show that $\nu_z(J_{x,i}) = 0$ for $\bm{L}$-almost every $z \in K$. Since $\mu \ll \bm{L}$, this implies~\eqref{eq:m_gto_i_exp} and thus completes the proof of Proposition~\ref{prop:ac_preserve}. In general, it is not true, however, that $\nu_z(J_{x,i}) = 0$ for {\bf every} $z \in K$, as can be seen by considering the constant vector fields $F_1 \equiv (1,0)^{\top}$ and $F_2 \equiv (0,1)^{\top}$ on $\R^2$. If $\tilde A = \{((x_1, x_2), s) \in K_j: x_1 + x_2 = t\}$, one has $\bm{L}(\tilde A \times \{j\}) = 0$ while $\nu_{{\bf 0}, s,i}(J_{{\bf 0}, i})$ may be nonzero because every trajectory of $X$ starting at the origin ${\bf 0} = (0,0)$ ends up on the line $x_1 + x_2 = t$ at time $t$. 

\medskip 

For $i \in E$ fixed, we first show that $\bm{\lambda}_{m-1}(J_{x,i}) = 0$ for $\bm{\lambda}_d$-almost every $x \in \R^d$. 
For $t > 0$ and an integer $n \geq 1$, let $D^t_n$ be defined as in~\eqref{eq:def_Delta_T_m} and let $\cl(D^t_n)$ denote its closure. One has 
\begin{equation}    \label{eq:cases_lambda} 
\int_{\R^d} \bm{\lambda}_{m-1}(J_{x,i}) \ \bm{\lambda}_d(dx)  = \begin{cases}
                             (\bm{\lambda}_d \otimes \bm{\lambda}_{m-1}) g^{-1}(\tilde A), & \quad m = 2, \\
                             (\bm{\lambda}_d \otimes \bm{\lambda}_{m-1}) g^{-1}(\tilde A \times \R^{m-2}_+), & \quad m > 2,  
                             \end{cases} 
\end{equation}
where 
\begin{align*}
g:  \R^d \times \cl(\Delta^t_{m-1}) \to& \R^d \times \R_+ \times \R_+^{m-2},  \\
(x,\sbf) \mapsto& (\varphi^j_{t-(s_1 + \ldots + s_{m-1})}(\Phi_{(s_1, \ldots, s_{m-1})}^{(i, i_2, \ldots, i_{m-1})}(x)), t-(s_1 + \ldots + s_{m-1}), s_2, \ldots, s_{m-1}). 
\end{align*}
In the case $m=2$, the components $s_2, \ldots, s_{m-1}$ are not present. Since $\bm{\lambda}_{d+1}(\tilde A) = 0$ and since 
$$
\det Dg(x,\sbf) = -\det D (\varphi^j_{t-(s_1 + \ldots + s_{m-1})} \circ \Phi_{(s_1, \ldots, s_{m-1})}^{(i, i_2, \ldots, i_{m-1})})(x) \neq 0
$$
by the diffeomorphism property of the flows, Lemma~\ref{lm:Davydov} implies that the right-hand side of~\eqref{eq:cases_lambda} is zero. Thus, 
$$
\int_{\R^d} \bm{\lambda}_{m-1}(J_{x,i}) \ \bm{\lambda}_d(dx) = 0 
$$
and $\bm{\lambda}_{m-1}(J_{x,i}) = 0$ for $\bm{\lambda}_d$-almost every $x \in \R^d$. 
Granting Lemma~\ref{lm:ac_nu_z}, we have $\nu_z \ll \bm{\lambda}_{m-1}$ for every $z \in K$, so $\nu_z(J_{x,i}) = 0$ for $\bm{L}$-almost every $z \in K$. 
\epf 

\bigskip 

To conclude, we give a proof of Lemma~\ref{lm:ac_nu_z}. 

\bigskip 

\bpf[Proof of Lemma~\ref{lm:ac_nu_z}] Fix $z = (x,s,i) \in K$ and define the ($z$-dependent) functions $\tilde \sigma_1, \ldots, \tilde \sigma_{m-1}$ as in the proof of Proposition~\ref{prop:acces}. Since  
$$
\tilde S^z_k = \tilde \sigma_k(U_1, \ldots, U_k), \quad 1 \leq k \leq m-1, 
$$
one has $\nu_z = \bm{\lambda}_{m-1} H^{-1}$, where 
\begin{equation}      \label{eq:def_h} 
H: (0, 1]^{m-1} \to \R^{m-1}, \ (u_1, \ldots, u_{m-1}) \mapsto (\tilde \sigma_1(u_1), \tilde \sigma_2(u_1, u_2), \ldots, \tilde \sigma_{m-1}(u_1, \ldots, u_{m-1})). 
\end{equation} 
We claim that $H$ is differentiable on  $\Xi^{\mathrm{c}}_z \times (\Xi^{\mathrm{c}})^{m-2}$ and that for all $(u_1, \ldots, u_{m-1}) \in \Xi^{\mathrm{c}}_z \times (\Xi^{\mathrm{c}})^{m-2}$, $\det D H(u_1, \ldots, u_{m-1})$ is nonzero, which yields by Lemma~\ref{lm:Davydov} that $\bm{\lambda}_{m-1} H^{-1} \ll \bm{\lambda}_{m-1}$ (note that $\Xi_z$ and $\Xi$ are countable, so $H$ is indeed  differentiable $\bm{\lambda}_{m-1}$-almost everywhere). Hence, Lemma \ref{lm:ac_nu_z} will be proven if we establish the claim. In order to show that $H$ is differentiable on $\Xi^{\mathrm{c}}_z \times (\Xi^{\mathrm{c}})^{m-2}$, we need to prove differentiability of $\tilde \sigma_k$ for $1 \leq k \leq m-1$ on $\Xi^{\mathrm{c}}_z \times (\Xi^{\mathrm{c}})^{k-1}$, which can be done by induction. In the base case $k=1$, one has $\tilde \sigma_1(u_1) = \psi_z(u_1)$, which is $C^1$ on $\Xi^{\mathrm{c}}_z$ (see Remark~\ref{rm:snon0}).

In the induction step, we have for $k > 1$ 
$$
\tilde \sigma_k(u_1, \ldots, u_k) = \beta(\alpha_k(u_1, \ldots, u_k)), 
$$
where  
$$
\alpha_k: (0,1)^k \to \R^d \times (0,1), \ (u_1, \ldots, u_k) \mapsto \left(\Phi^{(i,\ldots, i_{k-1})}_{(\tilde \sigma_1(u_1), \ldots, \tilde \sigma_{k-1}(u_1, \ldots, u_{k-1}))}(x), u_k \right)
$$
and 
$$
\beta: \R^d \times (0,1) \to \R, \ (x,u) \mapsto \psi^{i_k}_x(u). 
$$
By induction hypothesis, for $1 \leq j \leq k-1$, $\tilde \sigma_j$ is differentiable on $\Xi^{\mathrm{c}}_z \times (\Xi^{\mathrm{c}})^{j-1}$. Together with smoothness of $(F^i)_{i \in E}$, this implies that $\alpha_k$ is differentiable on $\Xi^{\mathrm{c}}_z \times (\Xi^{\mathrm{c}})^{k-1}$. Note that $\alpha_k$ maps $\Xi^{\mathrm{c}}_z \times (\Xi^{\mathrm{c}})^{k-1}$ to $\R^d \times \Xi^{\mathrm{c}}$. Moreover, $\beta$ is $C^1$ on $\R^d \times \Xi^{\mathrm{c}}$ by Lemma \ref{lem:psiC1}. Hence, $\tilde \sigma_k$ is differentiable on $\Xi^{\mathrm{c}}_z \times (\Xi^{\mathrm{c}})^{k-1}$. 

For all $(u_1, \ldots, u_{m-1}) \in \Xi^{\mathrm{c}}_z \times (\Xi^{\mathrm{c}})^{m-2}$, the Jacobian matrix $D H(u_1, \ldots, u_{m-1})$ is triangular, so 
\begin{equation}    \label{eq:det_D_h} 
\det D H(u_1, \ldots, u_{m-1}) = \prod_{k=1}^{m-1} \partial_{u_k} \tilde \sigma_k(u_1, \ldots, u_k) = \psi_z'(u_1) \prod_{k=2}^{m-1} \left(\psi_{x_k}^{i_k} \right)'(u_k),  
\end{equation} 
where 
$$
x_k := \Phi^{(i, \ldots, i_{k-1})}_{(\tilde \sigma_1(u_1), \ldots, \tilde \sigma_{k-1}(u_1, \ldots, u_{k-1}))}(x), \quad 2 \leq k \leq m-1. 
$$
By Lemma~\ref{lem:psiC1} and Remark~\ref{rm:snon0}, all the terms in the product are negative, hence the right-hand side of \eqref{eq:det_D_h} is nonzero and the claim is proved. 
\epf

\section{Proofs Propositions \ref{prop:Fiaccessible} and \ref{prop:acces}: characterisation of accessible points}
\label{sec:proofaccess}
We start by the proof of Proposition~\ref{prop:acces}, because we use this result to prove Proposition~\ref{prop:Fiaccessible}. More precisely, we prove that when $0$ is in the support of the $\mu^i$, then if $x^*$ is $\{F^i\}$ - accessible from $x$, then for all $i^* \in E$, there is $s^* > 0$ such that $(x^*, s^*, i^*)$ is $\{F^i, \mu^i, Q\}$ - accessible from $z = (x,0,i)$ for all $i \in E$.
\subsection{Proof of Proposition \ref{prop:acces}}
\subsubsection{$(P_t)$ - accessible points are $\{F^i, \mu^i, Q\}$-accessible }
Suppose first that $z^*$ is $(P_t)$-accessible from $z = (x,s,i)$. To show that $z^*$ is $\{F^i, \mu^i, Q\}$-accessible from $z$, we need to verify that every neighborhood of $z^*$ in $K$ contains a point from $\gamma^+(z)$. Let $U$ be a neighborhood of $z^*$ in $K$. Since $z^*$ is $(P_t)$-accessible from $z$, there is $t > 0$ such that 
$$
0 < P_t(z,U) = \PP_z(Z_t \in U) = \sum_{m=1}^{\infty} \PP_z(Z_t \in U, T_{m-1} \leq t < T_m).   
$$
Hence, there is $m \geq 1$ such that $\PP_z(Z_t \in U, T_{m-1} \leq t < T_m) > 0$. Assume first that 
$$
0 < \PP_z(Z_t \in U, t < T_1) = \id_U(\varphi_t^i(x), s+t, i) \PP_z(t < T_1), 
$$
so in particular 
$$
\hat z := (\varphi_t^i(x), s+t, i) \in U. 
$$
We choose $(t,i)$ as the control sequence. To see that $(t,i)$ is admissible with respect to $z$, we only need to verify that the intersection of $(s+t, \infty)$ and the support of $\mu_{\varphi_{-s}^i(x),0}^i$ is nonempty. To this end, it is enough to show that 
$$
\mu_{\varphi_{-s}^i(x), 0}^i((s+t, \infty)) > 0. 
$$
We have 
\begin{align*}
0 <& \PP_z(T_1 > t)  = \PP(\psi_z(U_1) > t) \\
=& \PP(\inf\{r \geq 0:  G_z(r) \leq U_1\} > t) = \PP (G_z(t) > U_1) \\
=& \PP \biggl( \frac{ G^i(\int_{-s}^t \lambda^i(\varphi_r^i(x)) \ dr)}{ G^i(\int_{-s}^0 \lambda^i(\varphi_r^i(x)) \ dr)} > U_1 \biggr) 
= \frac{ G^i(\int_{-s}^t \lambda^i(\varphi_r^i(x)) \ dr)}{ G^i(\int_{-s}^0 \lambda^i(\varphi_r^i(x)) \ dr)}. 
\end{align*}
As a result, 
\begin{align*}
0 <&  G^i \biggl(\int_{-s}^t \lambda^i(\varphi_r^i(x)) \ dr \biggr) \\
=&  G^i \biggl(\int_0^{s+t} \lambda^i(\varphi_u^i(\varphi_{-s}^i(x))) \ du \biggr) = G_{\varphi_{-s}^i(x),0}^i(s+t) = \mu_{\varphi_{-s}^i(x),0}^i((s+t, \infty)). 
\end{align*}
The point $\hat z$ is then also an element of $\gamma^+(z)$.  

Now suppose that $\PP_z(Z_t \in U, t < T_1) = 0$. Then there is $m \geq 2$ such that 
$$
\PP_z(Z_t \in U, T_{m-1} \leq t < T_m) > 0. 
$$
Expanding the expression on the left further, we see that there is $(i_2, \ldots, i_m) \in E^{m-1}$ such that 
\begin{equation}   \label{eq:positive_prob} 
\PP_z(Z_t \in U, T_{m-1} \leq t < T_m, B) > 0, 
\end{equation} 
where 
$$
B := \{\mathbf{I}_1 = i_2, \ldots, \mathbf{I}_{m-1} = i_m\}. 
$$ 
It will also be convenient to denote the index $i$ from the starting point $z$ by $i_1$ from now on. We shall iteratively construct a sequence $(s_1, \ldots, s_m) \in (0,\infty)^m$ and derive a control sequence from it that is admissible with respect to $z$. We iteratively define the random variables 
$$
\tilde S_1 := \psi_z(U_1) 
$$
and for $2 \leq k \leq m$
$$
\tilde S_k := \psi^{i_k}_{\Phi_{(\tilde S_1, \ldots, \tilde S_{k-1})}^{(i_1, \ldots, i_{k-1})}(x), 0}(U_k). 
$$
We will show by induction that, conditional on the event $B$, $\tilde S_k = S_k$ for $1 \leq k \leq m$. In the proof, we will use the fact that, conditional on $B$, 
$$
\mathbf{X}_k = \Phi_{(S_1, \ldots, S_k)}^{(i_1, \ldots, i_k)}(x), \quad 1 \leq k \leq m.  
$$
By definition, $\tilde S_1 = S_1$. In the induction step, suppose that there is $k \in \{1, \ldots, m-1\}$ such that, conditional on the event $B$, $\tilde S_j = S_j$ for $1 \leq j \leq k$. Then  
\begin{align*}
S_{k+1} =& \psi_{\mathbf{Z}_k}(U_{k+1}) = \psi_{\mathbf{X}_k,0}^ {\mathbf{I}_k}(U_{k+1}) \\
=& \psi^{i_{k+1}}_{\Phi_{(S_1, \ldots, S_k)}^{(i_1, \ldots, i_k)}(x), 0}(U_{k+1}) 
= \psi^{i_{k+1}}_{\Phi_{(\tilde S_1, \ldots, \tilde S_k)}^{(i_1, \ldots, i_k)}(x), 0}(U_{k+1}) 
= \tilde S_{k+1}. 
\end{align*}
As a result, 
\begin{equation}   \label{eq:equal_J_M_B_set} 
\{Z_0 = z, Z_t \in U, T_{m-1} \leq t < T_m, B\} = J \cap M \cap B, 
\end{equation} 
where 
\begin{align*}
J :=& \{(\varphi^{i_m}_{t - (\tilde S_1 + \ldots + \tilde S_{m-1})}(\Phi_{(\tilde S_1, \ldots, \tilde S_{m-1})}^{(i_1, \ldots, i_{m-1})}(x)), t - (\tilde S_1 + \ldots + \tilde S_{m-1}), i_m) \in U\}, \\
M :=& \{\tilde S_1 + \ldots + \tilde S_{m-1} \leq t < \tilde S_1 + \ldots + \tilde S_m\}. 
\end{align*}
In addition, using the independence of $(U_k)_{k \geq 1}$ and $(V_k)_{k \geq 1}$ and the fact that $(\tilde S_1, \ldots, \tilde S_m)$ is measurable with respect to the $\mathcal{F}(U)$, we obtain, conditioning on $\mathcal{F}(U)$, that 
%Together with~\eqref{eq:positive_prob},~\eqref{eq:equal_J_M_B_set}, and $J, M \in \Uc$, this shows that \textcolor{blue}{on peut éventuellement omettre les détails de cette preuve, en disant juste que l'on utilise l'indépendance de $U_k$ et des $V_k$, ainsi que le fait que les $\tilde{S}_k$ sont mesurables par rapport à $\mathcal{U} $ - que j'avais d'ailleurs proposé de noter $\mathcal{F}_m(U)$ dans l'introduction}
\begin{equation}   \label{eq:prob_J_M_B} 
0 < \PP(J,M,B) = \E[\id_{J,M} \prod_{j=1}^{m-1} q_{i_j, i_{j+1}}(\Phi_{(\tilde S_1, \ldots, \tilde S_j)}^{(i_1, \ldots, i_j)}(x))] = \E[F(\tilde S_1, \ldots, \tilde S_m)], 
\end{equation} 
where 
$$
F: [0,\infty)^m \to \R, \ (s_1, \ldots, s_m) \mapsto \id_{J^{\sigma}, M^{\sigma}}(s_1, \ldots, s_m) \prod_{j=1}^{m-1} q_{i_j, i_{j+1}}(\Phi_{(s_1, \ldots, s_j)}^{(i_1, \ldots, i_j)}(x))
$$
and  
\begin{align*}
J^{\sigma} :=& \{\sbf \in [0,\infty)^m: \ (\varphi^{i_m}_{t - (s_1 + \ldots + s_{m-1})}(\Phi_{(s_1, \ldots, s_{m-1})}^{(i_1, \ldots, i_{m-1})}(x)), t - (s_1 + \ldots + s_{m-1}), i_m) \in U\}, \\
M^{\sigma} :=& \{\sbf \in [0,\infty)^m: \ s_1 + \ldots + s_{m-1} \leq t < s_1 + \ldots + s_m\}. 
\end{align*}
By construction of $(\tilde S_1, \ldots, \tilde S_m)$ and independence of the $U_k, k=1,...,m,$ we obtain, conditioning successively on $U_1, U_2, \ldots, U_m$ that
\[
\E[F(\tilde S_1, \ldots, \tilde S_m)] = \int \cdots \int F(s_1, \ldots, s_m) d \mu_{\Phi^{(i_1, \ldots i_{m-1})}_{(s_1, \ldots, s_{m-1})}(x), 0}^{i_m} (s_m) \ldots d \mu_{\Phi^{i_1}_{s_1}(x), 0}^{i_2}(s_2) d\mu_z(s_1). 
\]
Hence, since $\E[F(\tilde S_1, \ldots, \tilde S_m)] > 0$, one can iteratively choose $s_1, s_2, \ldots, s_m$ such that $s_1 \in \supp(\mu_z)$, for $k=2, \ldots, m$, $s_k \in \supp(\mu_{\Phi^{(i_1, \ldots i_{k-1})}_{(s_1, \ldots, s_{k-1})}(x), 0}^{i_k})$ and $F(s_1, \ldots, s_m) > 0$.
This implies in particular that $(s_1, \ldots, s_m) \in J^{\sigma} \cap M^{\sigma}$, so $s_1 + \ldots + s_{m-1} \leq t$. We distinguish between two cases. If $s_1 + \ldots + s_{m-1} < t$, then 
$$
((s_1, \ldots, s_{m-1}, t - (s_1 + \ldots + s_{m-1})), (i_1, \ldots, i_m))
$$
is a control sequence. It is also admissible with respect to $z$: Since $s_1 \in \supp(\mu_z)$, we have $s + s_1 \in \supp(\mu_{\varphi^i_{-s}(x), 0}^i)$. Condition (b) follows from 
$$
q_{i_j, i_{j+1}}(\Phi_{(s_1, \ldots, s_j)}^{(i_1, \ldots, i_j)}(x)) > 0, \quad 1 \leq j \leq m-1,  
$$
and the fact that $s_k \in \supp(\mu^{i_k}_{\Phi_{(s_1, \ldots, s_{k-1})}^{(i_1, \ldots, i_{k-1})}(x), 0})$ for $2 \leq k \leq m-1$. As far as Condition (c) is concerned, notice that $t - (s_1 + \ldots + s_{m-1})  < s_m$ because $(s_1, \ldots, s_m) \in M^{\sigma}$. Then $s_m$ lies in the intersection of $(t - (s_1 + \ldots + s_{m-1}), \infty)$ and the support of $\mu^{i_m}_{\Phi_{(s_1, \ldots, s_{m-1})}^{(i_1, \ldots, i_{m-1})}(x), 0}$. The point 
$$
\hat z := (\Phi_{(s_1, \ldots, s_{m-1}, t - (s_1 + \ldots + s_{m-1}))}^{(i_1, \ldots, i_m)}(x), t - (s_1 + \ldots + s_{m-1}), i_m) 
$$
is then an element of $\gamma^+(z)$. Since $(s_1, \ldots, s_m) \in J^{\sigma}$, we also have $\hat z \in U$. 

Now assume that $s_1 + \ldots + s_{m-1} = t$. Fix $\eps \in (0, s_m)$. Then 
$$
((s_1, \ldots, s_{m-1}, \eps), (i_1, \ldots, i_m)) 
$$
is an admissible control sequence with respect to $z$. Let 
$$
\hat z := (\Phi_{(s_1, \ldots, s_{m-1}, \eps)}^{(i_1, \ldots, i_m)}(x), \eps, i_m) \in \gamma^+(z).  
$$
Since $U$ is open and since $(s_1, \ldots, s_m) \in J^{\sigma}$, continuity of the flows implies $\hat z \in U$ for $\eps$ sufficiently small.

\subsubsection{$\{F^i, \mu^i, Q\}$-accessible points are $(P_t)$-accessible }     \label{ssec:F_mu_Q_implies_P}
We start with a lemma.
\begin{lemma}
\label{lem:supportmuz}
Let $z_0 = (x_0,s_0,i) \in K$ and $s_1$ such that $s_1$ lies in the support of $\mu_{z_0}$. Then, for all $\delta > 0$, there exists a neighbourhood $V$ of $(x_0,s_0)$ in $K^i$ and a constant $c > 0$ such that, for all $(x,s) \in V$, $\mu_{(x,s,i)}((s_1 - \delta, s_1 + \delta])\geq c$.
\end{lemma}

\begin{proof}
First, we prove that $s_1$ is in the support of $\mu_{z_0}$ if and only if $a(x_0,s_0,s_1)$ is in the support of $\mu^i$, where, for all $(x,s,r) \in \R^d \times \R_+ \times \R_+$, we define
\[
a(x,s,r) = \int_{-s}^{r} \lambda^i( \varphi^i_u(x) ) du.
\]
Indeed, we have that $s_1$ is in the support of $\mu_{z_0}$ if and only if, for every $\delta_1, \delta_2 > 0$, $\mu_{z_0}((s_1 - \delta_1, s_1 + \delta_2]) > 0$. By definition of $\mu_{z_0}$, this means that 
\[
  G^i \left( \int_{-s_0}^{s_1-\delta_1} \lambda^i( \varphi^i_u(x_0) ) du \right) -   G^i \left( \int_{-s_0}^{s_1+\delta_2} \lambda^i( \varphi^i_u(x_0) ) du \right) > 0,
\]
hence
\[
\mu^i\left( \left(a(x_0,s_0,s_1) - \int_{s_1-\delta_1}^{s_1} \lambda^i( \varphi^i_u(x_0) ) du, a(x_0,s_0,s_1) + \int_{s_1}^{s_1+\delta_2} \lambda^i( \varphi^i_u(x_0) ) du\right] \right) > 0.
\]
Since $\lambda^i$ is bounded away from $0$, this is equivalent to
\[
\mu^i\left( \left(a(x_0,s_0,s_1) - \varepsilon_1, a(x_0,s_0,s_1) + \varepsilon_2 \right] \right) > 0
\]
for all $\varepsilon_1, \varepsilon_2 > 0$, which means that $a(x_0,s_0,s_1)$ is in the support of $\mu^i$. Now, let $\delta > 0$ and let $ \varepsilon = \min \{ \int_{s_1 - \delta}^{s_1} \lambda^i ( \varphi^i_u(x_0) ) du; \int^{s_1 + \delta}_{s_1} \lambda^i ( \varphi^i_u(x_0) ) du \} > 0$. We have
$$
a(x_0, s_0, s_1) - a(x,s,s_1 - \delta)
= \int_{s_1 - \delta}^{s_1} \lambda^i ( \varphi^i_u(x_0) ) du + a(x_0, s_0, s_1 - \delta) - a(x,s,s_1 - \delta). 
$$
By continuity of $a$, it is thus possible to find a neighbourhood $V$ of $(x_0, s_0)$ such that, for all $(x,s) \in V$, one has $a(x_0, s_0, s_1) - a(x,s,s_1 - \delta) \geq \varepsilon/2$. Similarly, we can assume that for all $(x,s) \in V$, $a(x,s,s_1 + \delta)  - a(x_0, s_0, s_1)\geq \varepsilon/2$. Thus, for all $(x,s) \in V$, we have
\begin{align*}
\mu_{(x,s,i)} ( ( s_1 - \delta, s_1 + \delta]) & = \frac{\mu^i( ( a(x,s,s_1-\delta), a(x,s, s_1 + \delta)])}{\mu^i( (a(x,s,0), +\infty) )}\\
& \geq \mu^i( ( a(x_0,s_0,s_1)- \tfrac{\varepsilon}{2}, a(x_0,s_0,s_1)+ \tfrac{\varepsilon}{2}]) = c > 0,
\end{align*}
where $c$ is positive since $a(x_0,s_0,s_1)$ is in the support of $\mu^i$. 
\end{proof} 

Now we prove that if $z^*$ is $\{F^i, \mu^i, Q\}$-accessible, then it  is $(P_t)$-accessible. We start by showing that if $z^* \in \gamma^+(z)$ for zome $z \in K$, then for every neighbourhood $U$ of $z^*$, there exist a neighbourhood $V$ of $z$ and $t > 0$  such that, for all $z' \in V$, $P_t(z', U) > 0$. We procede by induction. For all $n \geq 0$, we consider the proposition \textbf{P($n$)} : 

\medskip 

\noindent For all  $z = (x,s,i) \in K$, for every admissible control sequence $(\sbf, \ibf)=(s_1, \ldots, s_n, i_1, \ldots, i_n)$ with respect to $z$ and for every neighbourhood $U$ of $z^* := (\Phi_{\sbf}^{\ibf}(x), s^*, i_n)$ (with $s^* = s + s_1$ if $n =1$ and $s^* = s_n$ if $n > 1$), there exist a neighbourhood $V$ of $z$ and $t > 0$ such that, for all $z' \in V$, $P_t(z', U) > 0$. 

\medskip 

We prove \textbf{P($1$)}. Let $z=(x,s,i) \in K$ and $(s_1, i_1)$ an admissible control sequence with respect to $z$. That is, $i_1 = i$ and the intersection of $(s+s_1, \infty)$ with the support of $\mu_{\varphi^i_{-s}(x),0}^i$ is nonempty. This implies that $  G_z(s_1) > 0$. Let $U$ be a neighbourhood of $z^*=\phi_{s_1}(z)=(x^*, s^*, i^*)$, with $x^* = \varphi^i_{s_1}(x)$, $s^* = s + s_1$ and $i^* = i_1$. Then, for all $z' \in K$, 
\begin{align*}
\mathbb{P}_{z'}( Z_{s_1} \in U) & \geq \mathbb{P}_{z'}\left( Z_{s_1} \in U, T_1 > s_1 \right)\\
& = \1_{\phi_{s_1}(z') \in U}   G_{z'}(s_1).
\end{align*}
By continuity of $\phi$, since $\phi_{s_1}(z)=z^* \in U$, there exists a neighbourhood $V_1$ of $z$ such that $\1_{\phi_{s_1}(z') \in U} = 1$ for all $z' \in V$. Furthermore, we have $  G_z(s_1) > 0$, which is equivalent to 
\[
\int_{-s}^{s_1} \lambda^i \left( \varphi_u^i(x) \right) du < {\bar t}^i. 
\]
By continuity of the above left quantity, there exists a neighbourhood $ V_2$ of $z$ such that, for all $z' \in V_2$, $G_{z'}(s_1) > 0$. Taking $V = V_1 \cap V_2$ and $t=s_1$ concludes the proof of \textbf{P($1$)}.

Now, assume that \textbf{P($n$)} is true for some $n \geq 1$. We prove that \textbf{P($n+1$)} is also true. Let $z = (x,s,i) \in K$ and  $(\sbf, \ibf)=(s_1, \ldots, s_n, s_{n+1}, i_1, \ldots, i_n, i_{n+1})$  an admissible control sequence with respect to $z$, of length $n+1$. It is easily seen from the definition that $(\sbf_-, \ibf_-) = (s_2, \ldots s_n, s_{n+1}, i_2, \ldots i_n, i_{n+1})$ is an admissible control sequence with respect to $z_1 = (\varphi_{s_1}^{i_1}(x), 0, i_2)$, of length $n$. Note that the point $z^*=(x^*, s^*, i^*)$ defined by 
$$
x^* = \Phi_{\sbf}^{\ibf}(x), \quad s^* =s_{n+1}, \quad i^* = i_{n+1}
$$
is reachable from $z$ with the control sequence $(\sbf, \ibf)$ and from $z_1$ with the control sequence $(\sbf_-, \ibf_-)$. Let $U$ be a neighbourhood of $z^*$. Since $(\sbf_-, \ibf_-)$ is an admissible control sequence of length $n$ with respect to $z_1$, \textbf{P($n$)}  implies that there exist a neighbourhood $V$ of $z_1$ and a time $t_1 > 0$ such that 
$P_{t_1}(z',U) > 0$, for all $z' \in V$. Now we claim that there exist a neighbourhood $W$ of $z$ and a time $t_0$ such that, for all $z'' \in W$, $P_{t_0}(z'',V) > 0$. This will imply \textbf{P($n+1$)} since for all $z'' \in W$,
\begin{align*}
P_{t_0 + t_1}(z'', U) & = \int_K P_{t_1}(z', U) P_{t_0}(z'', dz')\\
& \geq \int_V P_{t_1}(z', U) P_{t_0}(z'', dz') > 0.
\end{align*}
We prove the claim. For all $z' = (x', s', i') \in K$ with $i' = i$ and for all $t \geq 0$, we have
\begin{align*}
\mathbb{P}_{z'}( Z_t \in V) & \geq \mathbb{P}_{z'} (Z_t \in V, T_1 \leq t < T_2 )\\
& = \mathbb{E}_{z'} \left[ \1_{T_1 \leq t} \mathbb{P}_{z'}( Z_t \in V, t < T_2 | \mathcal{F}_{T_1}) \right]\\
& = \sum_{j \in E} \int_0^t q_{i,j}( \varphi_u^{i}(x') )   G_{\varphi^{i}_u(x'), 0}^j ( t - u ) \1_{\phi_{t-u}( \varphi_u^{i}(x'), 0, j) \in V} \mu_{z'}(du).\\
& \geq \int_0^t q_{i,i_2}( \varphi_u^{i}(x') )   G_{\varphi^{i}_u(x'), 0}^{i_2} ( t - u ) \1_{\phi_{t-u}( \varphi_u^{i}(x'), 0, i_2) \in V} \mu_{z'}(du).
\end{align*}
Since $(\sbf, \ibf)$ is an admissible control sequence with respect to $z$, we know that $s_1$ is in the support of $\mu_z$. Furthermore, we have $q_{i, i_2}( \varphi^i_{s_1}(x) ) > 0$. By continuity of the flow and of the map $q_{i, i_2}$, this implies that for $x'$ close to $x$ and $u$ close to $s_1$, $q_{i, i_2}( \varphi^i_{u}(x') ) > 0$. In addition, since $V$ is a neighbourhood of $z^1 = (\varphi_{s_1}^{i_1}(x), 0, i_2)$, we have that $\phi_{t-u}( \varphi_u^i(x'), 0, i_2) \in V$ for $u$ close to $s_1$, $t - u$ close to $0$ and $x'$ close to $x$. Finally, one can choose $\varepsilon > 0$ such that 
\[
\int_0^{\varepsilon} \lambda^{i_2} \left( \varphi^{i_2}_r  \circ \varphi^i_{s_1}(x) \right) dr < {\bar t}^i,
\]
that is $  G_{\varphi^i_u(x'), 0}^{i_2}(t-u) > 0$ 
for $u$ close to $s_1$, $t-u$ close to $0$ and $x'$ close to $x$. Thus, in conjunction with Lemma~\ref{lem:supportmuz}, it is possible to choose $\varepsilon > 0$ and a neighbourhood $W$ of $z$ such that, for all $z' \in W$,
\begin{align*}
\mathbb{P}_{z'}( Z_{s_1 + \frac{\varepsilon}{2}} \in V) 
& \geq \int_0^{s_1 + \frac{\varepsilon}{2}} q_{i,i_2}( \varphi_u^i(x') )   G_{\varphi^i_u(x'), 0}^{i_2} ( t - u ) \1_{\phi_{t-u}( \varphi_u^i(x'), 0, i_2) \in V} \mu_{z'}(du)\\
& \geq 
 \int_{s_1 - \frac{\varepsilon}{2}}^{s_1 + \frac{\varepsilon}{2}} q_{i,i_2}( \varphi_u^i(x') )   G_{\varphi^i_u(x'), 0}^{i_2} ( t - u ) \1_{\phi_{t-u}( \varphi_u^i(x'), 0, i_2) \in V} \mu_{z'}(du) \\
 & > 0.
\end{align*}
This concludes the proof of the claim. In particular, we have shown that every point in $\gamma^+(z)$ is $(P_t)$-accessible from $z$. 

\medskip 

Now let $z^*$ be a point in the relative closure of $\gamma^+(z)$ in $K$ and let $U$ be a neighbourhood of $z^*$. Then, there exists $z_1 \in \gamma^+(z) \cap U$. Since $z_1$ is $(P_t)$-accessible from $z$, there exists $t > 0$ such that $P_t(z,U) > 0$. This proves that $z^*$ is $(P_t)$-accessible from $z$. As this argument holds for any $z \in K$, we have proved that $\{F^i, \mu^i, Q\}$-accessibility implies $(P_t)$-accessibility.

\subsection{Proof of Proposition \ref{prop:Fiaccessible}}

We start by two lemmas. For a measure $\mu$ on $\R_+$, we denote by $\mathrm{supp} (\mu)$ its topological support. 

\begin{lemma}
\label{lem:0insupport}
Let $i \in E$ and assume that $0$ is in the support of $\mu_i$. Then, for all $x \in \R^d$ and all $\varepsilon > 0$, there exists $v \in (0, \varepsilon]$ such that $v \in \mathrm{supp}(\mu_{x,i})$.
\end{lemma}

\begin{proof}
First, let us prove that for all $\varepsilon > 0$, there is $t \in (0, \varepsilon]$ such that $t \in \mathrm{supp}(\mu_{i})$. Since $0$ belongs to the support of $\mu_i$ and $\mu_i(\{0\}) = 0$, we have $\mu_i( (0, \varepsilon) ) > 0$. Now, recall that the support of $\mu_i$ is the complement of the union of the open sets with measure $0$. Hence,  since $(0, \varepsilon)$ is an open set with positive measure, the intersection of $(0,\varepsilon)$ and the support of $\mu_i$ has to be nonempty. Now, let $\varepsilon > 0$ and $x \in \R^d$. Recall (see proof of Lemma \ref{lem:supportmuz}) that $v \in \mathrm{supp}(\mu_{x,i})$ if and only if $\int_0^v \lambda_i(\varphi_u^i(x) ) du$ is in the support of $\mu_i$. Now, we know that there exists $t \in \mathrm{supp}(\mu_i)$ such that $t \leq \int_0^{\varepsilon} \lambda_i(\varphi_u^i(x) ) du$. Furthermore, there exists $v \leq \varepsilon$ such that $t = \int_0^{v} \lambda_i(\varphi_u^i(x) ) du$. In particular, $v$ belongs to the support of $\mu_{x,i}$, which concludes the proof. 
\end{proof}
The secound lemma controls the distance between two composite flows following the same sequence of states but with different holding times. For a sequence $\ubf =(u_1, \ldots, u_m) \in \R^m$, we denote $\|u\| = \sum_{i=1}^m |u_i|$ its $L^1$-norm.
\begin{lemma}
\label{lem:distancecomposite}
Let $M$ be a compact subspace of $\R^d$ and $T > 0$. Then, there exists positive constants $C_M$ and $L_M$ such that, for all $x, y \in M$, for all $m \geq 1$,  $\ibf \in E^m$,  $\ubf, \vbf \in \R_+^m$ such that $\| u \|, \|v \| \leq T$, we have
\[
\| \Phi_{\ubf}^{\ibf}(x) - \Phi_{\vbf}^{\ibf}(y) \| \leq C_M e^{L_M T} \left( \| u - v \| + \|x-y\| \right).
\]
\end{lemma}

\begin{proof}
We prove by induction on the length $m$ of the sequence $\ibf$ the following stronger  result. For all $x, y \in M$, for all $m \geq 1$,  $\ibf \in E^m$,  $\ubf, \vbf \in \R_+^m$ such that $\| u \|, \|v \| \leq T$, we have
\[
\| \Phi_{\ubf}^{\ibf}(x) - \Phi_{\vbf}^{\ibf}(y) \| \leq e^{L_M \|\ubf\|} \|x- y \| + C_M\left( \sum_{i=1}^{m} |u_i - v_i| e^{\sum_{j=i+1}^m u_j} \right),
\]
where $\sum_{j=m+1}^m u_j = 0$. 
 For $m=1$, we have to prove that for all $x,y, \in M$ and $u,v \in [0, T]$,
\[
\| \varphi^i_u(x) - \varphi^i_v(x) \| \leq e^{L_M u}\|x-y\|  + C_M | u - v |,  
\]
for some $C_M$ and $L_M$. Since $M \times [0, T]$ is compact and $(x,t) \mapsto \varphi^i_t(x)$ is continuous, the set $\varphi^i ( M \times [0, T])$ is compact for all $i$. We let $M_T$ be the union of the $\varphi^i ( M \times [0, T])$, for $i \in E$. This is a compact subset of $\R^d$. Since the $F^i$ are smooth, there exists constants $C_M$ and $L_M$ such that, for all $x \in M_T$, $\|F^i(x)\| \leq C_M$ and $\| DF^i(x) \| \leq L_M$ for all $i \in E$. In particular, since 
\[
\varphi^i_t(x) = x + \int_0^t F^i ( \varphi_u^i(x) ) du,
\]
one has 
\[
\| \varphi^i_u(x) - \varphi^i_v(x)\| \leq C_M | u - v |
\]
and
\[
\| \varphi^i_u(x) - \varphi^i_u(y) \| \leq \| x - y \| + L_M \int_0^u \| \varphi^i_s(x) - \varphi^i_s(y) \| ds,
\]
which,  by Gronwall Lemma and the triangular inequality, implies the result for $m=1$.

Now assume the result is true for some $m \geq 1$. Let $\ibf = (i_1, \ldots, i_m, i_{m+1}) \in E^{m+1}$, $\ubf = (u_1, \ldots, u_m, u_{m+1}), \vbf = ( v_1, \ldots, v_m, v_{m+1}) \in \R_+^{m+1}$ with $\| u \|, \|v\| \leq T$. Define $\ibf_-$, $\ubf_-$, $\vbf_-$ from $\ibf, \ubf, \vbf$ by withdrawing the last component. Then
\begin{align*}
\| \Phi_{\ubf}^{\ibf}(x) - \Phi_{\vbf}^{\ibf}(y) \| & = \| \varphi^{i_{m+1}}_{u_{m+1}} \circ \Phi_{\ubf_-}^{\ibf_-}(x) -\varphi^{i_{m+1}}_{v_{m+1}} \circ \Phi_{\vbf_-}^{\ibf_-}(y) \| \\
& \leq e^{K_M u_{m+1}} \|  \Phi_{\ubf_-}^{\ibf_-}(x) - \Phi_{\vbf_-}^{\ibf_-}(y) \| + C_M |u_{m+1} - v_{m+1}|
\end{align*}
The induction hypothesis implies that 
\[
\|  \Phi_{\ubf_-}^{\ibf_-}(x) - \Phi_{\vbf_-}^{\ibf_-}(y) \| \leq  e^{L_M \|\ubf_-\|} \|x- y \| + C_M\left( \sum_{i=1}^{m} |u_i - v_i| e^{\sum_{j=i+1}^m u_j} \right),
\]
hence
\[
e^{K_M u_{m+1}} \|  \Phi_{\ubf_-}^{\ibf_-}(x) - \Phi_{\vbf_-}^{\ibf_-}(y) \| + C_M |u_{m+1} - v_{m+1}| \leq e^{L_M \|\ubf\|} \|x- y \| + C_M\left( \sum_{i=1}^{m+1} |u_i - v_i| e^{\sum_{j=i+1}^{m+1} u_j} \right),
\]
which concludes the proof.
\end{proof}

We now pass to the proof of Proposition \ref{prop:Fiaccessible}. First, note that by definition of $\{F^i, \mu^i, Q\}$ - accessible, if $(x^*, i^*, s^*)$ is $\{F^i, \mu^i, Q\}$ - accessible from $z=(x,s,i)$, then $x^*$ is $\{F^i\}$ - accessible from $x$. Hence, to prove Proposition \ref{prop:Fiaccessible}, it is sufficient to prove that for all $z = (x,s,i) \in K$, all  $x^*$ that are attainable from $x$, and all $i^* \in E$, there exists $s^* > 0$ such that $z^*=(x^*, i^*, s^*)$ lies in the relative closure of $\gamma^+(z)$ in $K$. Let $(\ibf, \sbf)=(i_1, \ldots, i_m; s_1, \ldots, s_m)$ a control sequence such that $x^* = \Phi^{\ibf}_{\vbf}(x)$. We first prove the result in the case where $z=(x, i_1, 0)$. The idea is the following. Since the process can jumps from one state to the other at arbitrary small times, we can add transitions in the sequence $\ibf$ without mooving too far away the resulting point from $x^*$. We will provide the algorithm to do so, but first we need a bit of notations. Let $i, j \in E$ and $y \in \R^d$. Since the matrix $Q(y)$ is irreducible, there exists $i = i_1, \ldots, i_n =j$ such that $q_{i_l, i_{l+1}}(y) > 0$ for $l= 1, \ldots, n-1$. Furthermore, $n$ can be chosen less or equal to $N+1$, where $N$ is the cardinality of $E$. For a control  sequence $(\jbf, \ubf)$ of size $n$ and $y \in \R^d$, we set $y_1 = y$, and $y_{k+1} = \varphi^{j_{k+1}}_{u_{k+1}}(y_k)$ for $k=1 \ldots, n$. Note that $y_{n+1} = \Phi^{\jbf}_{\ubf}(y)$. Finally, for $i, j \in E$, set $j_0 = i$,  $j_{n+1} = j$ and $p_{i,j}(y,\jbf, \ubf) = \prod_{k=0}^{n} q_{j_k, j_{k+1}}(y_{k+1})$. We introduce the following set 
\[
\mathbb{T}_{i,j}(y) = \{ (\jbf, \ubf) \in E^n \times (0, +\infty)^n,   0 \leq n \leq N \: :    p_{i,j}(y,\jbf, \ubf) > 0; \: u_k \in \mathrm{supp}(\mu_{y_k,j_k}) \}.
\]
Since for $\ubf = \mathbf{0}$, $y_k = y$ for all $k$, there exists $\jbf$ such that $ p(y,\jbf, \mathbf{0}) > 0$. Furthermore, by continuity of  $y \mapsto Q(y)$,  for all $\varepsilon > 0$ sufficiently small, if  $ p(y,\jbf, \mathbf{0}) > 0$, then by Lemma \ref{lem:0insupport}, one can fin $\ubf$ such that $\|\ubf\| \leq \varepsilon$ and $(\jbf, \ubf) \in \mathbb{T}_{i,j}(y)$.

Let $x \in \R^d$ and  $(\ibf, \sbf) \in E^m \times (0, +\infty)^m$. Let $T = 2 (s_1 + \ldots + s_m)$ and $M$ the closed ball of center $x$ and radius $1$. By Lemma \ref{lem:distancecomposite}, there exists $C, L > 0$ such that, for all $y, y' \in M$, for all $n \geq 1$,  $\jbf \in E^n$,  $\ubf, \vbf \in \R_+^n$ such that $\| u \|, \|v \| \leq T$, we have
\[
\| \Phi_{\ubf}^{\ibf}(y) - \Phi_{\vbf}^{\ibf}(y') \| \leq C e^{L T} \left( \| u - v \| + \|y-y'\| \right).
\]
In particular, 
\[
\| \Phi_{\ubf}^{\ibf}(x) - \Phi_{\vbf}^{\ibf}(x) \| \leq C e^{L T}  \| u - v \| 
\]
Let $\varepsilon > 0$, we set 
\[
h_{\sbf}( \varepsilon) = \frac{\varepsilon}{(m+1) C e^{LT}}.
\]
Without loss of generality, we may assume that $h_{\sbf}( \varepsilon) \leq s_k$, for $k = 1 \ldots, m$. Finally, for $j \in E$, let $S(j,\sbf, \varepsilon) > 0$ be a point in the support of $\mu_j$ such that $S(j,\sbf, \varepsilon) \leq \lambda_{\min} h_{\sbf}( \varepsilon)$ (such a point exists by Lemma \ref{lem:0insupport}). Let $y \in \R^d$, and $S > 0$, and set 
\[
t(y,j,S) = \inf \{ t \geq 0 \: : \int_0^t \lambda_j( \varphi_u^j(y) ) du \geq S \};
\]
then $t(y,j,S) \in [ \frac{S}{\lambda_{\max}}; \frac{S}{\lambda_{\min}} ]$. We denote $t_{\sbf, \varepsilon}(y,j) = t(y,j, S(j,\sbf, \varepsilon))$. By construction of $S(j,\sbf, \varepsilon)$, we have $t_{\sbf, \varepsilon}(y,j) \in  [ \frac{ S(j,\sbf, \varepsilon)}{\lambda_{\max}}; h_{\sbf}( \varepsilon)]$ and $t_{\sbf, \varepsilon}(y,j) $ belongs to the support of $\mu_{y,j}$. With all these notations, we can describe Algorithm \ref{algo:1} and proves the following proposition, which will imply Proposition \ref{prop:Fiaccessible}.
\begin{algorithm}
\label{algo:1}
\caption{Construction of a reachable point arbitrary close to an attainable point}
\begin{algorithmic}
\REQUIRE $x \in \R^d$, $m \geq 1$, $(\ibf, \sbf) \in E^m \times (0, +\infty)^m$, $\varepsilon > 0$
\STATE $k = 1$, $y = x$, $\jbf = \{\}$, $\vbf = \{\}$
\FOR{$l=1, \ldots, m$} 
\STATE $\tilde{t}=0$, 
\WHILE{$\tilde{t} < s_l -  h_{\sbf}(\varepsilon)$ } 
\STATE  $t \leftarrow t_{\sbf, \varepsilon}(y,i_l)$
\STATE $\vbf \leftarrow (\vbf, t)$, $\jbf \leftarrow (\jbf, i_l)$, $y \leftarrow \varphi^{i_l}_t(y)$
\IF{  $ t + \tilde t \leq s_l - h_{\sbf}( \varepsilon)$}
\STATE pick $(\jbf_1, \ubf) \in \mathbb{T}_{i_l,i_l}(y)$ with $\|\ubf\| \leq h_{\sbf}(\varepsilon) 2^{-k}$
\STATE $\vbf \leftarrow (\vbf, \ubf)$, $\jbf \leftarrow (\jbf, \jbf_1)$, $y \leftarrow \Phi^{\jbf_1}_{\ubf}(y)$
\ENDIF
\STATE $k \leftarrow k+1$, $\tilde t \leftarrow \tilde t + t$
\ENDWHILE
\IF{$l < m -1$}
\STATE pick $(\jbf_1, \ubf) \in \mathbb{T}_{i_l,i_{l+1}}(y)$ with $\|\ubf\| \leq h_{\sbf}(\varepsilon) 2^{-k}$
\STATE  $\vbf \leftarrow (\vbf,  \ubf)$, $\jbf \leftarrow (\jbf,  \jbf_1)$, $y \leftarrow \Phi^{\jbf_1}_{\ubf}(y)$
\STATE $k \leftarrow k+1$
\ENDIF
\ENDFOR
\end{algorithmic}
\end{algorithm}
\begin{proposition}
\label{prop:algo}
For any $x \in \R^d$,  $(\ibf, \sbf) \in E^m \times (0, +\infty)^m$, $\varepsilon > 0$, Algorithm \ref{algo:1} provides an admissible control sequence $(\jbf, \vbf)$  with $j_1 = i_1$ and the last element of $\jbf$ being $i_m$ such that
\[
\| \Phi^{\jbf}_{\vbf}(x) - \Phi^{\ibf}_{\sbf}(x) \| \leq \varepsilon.
\]
\end{proposition}
\begin{proof}
First, we notice that the algorithm eventually ends in a finite number of steps. Indeed, since $t_{\sbf, \varepsilon}(y,i_l) \geq  \frac{ S(j,\sbf, \varepsilon)}{\lambda_{\max}}$, after $n$ steps in a loop while, we have $\tilde t \geq n \frac{ S(j,\sbf, \varepsilon)}{\lambda_{\max}}$, hence the loop has at most $ \max_{l}(s_l) \frac{ \lambda_{\max}}{ S(j,\sbf, \varepsilon)}$ steps. Next, by construction of the algorithm, the resulting control sequence $(\jbf, \vbf)$ is admissible. Finally, let $n_1, \ldots, n_m$ denote the number of steps in the loop WHILE corresponding to $l=1, \ldots m$ respectively. For $l = 1, \ldots, m$, during these $n_k$ steps, the algorithm furnishes a control sequence $(\jbf_l, \vbf_l)$ such that $\jbf_l = (i_l, \jbf_l^1, i_l, \ldots, \jbf^{n_l - 1}_l, i_l)$ and $\vbf_l =( t_l^1, \vbf_l^1, \ldots, \vbf_l^{n_l - 1}, t_l^{n_l})$ (where $\vbf_l^{n_l-1}$ is empty if $n_l=1$). For $l = 1, \ldots, m - 1$, the final step of the loop FOR gives a control sequence that we denote by $(\jbf^{n_l}_l, \vbf^{n_l}_l)$. Thus, we can write
\[
\jbf = ( i_1, \jbf_1^1, i_1, \ldots, \jbf^{n_1-1}_l, i_1, \jbf^{n_1}_l, i_2, \ldots , \jbf^{n_m-1}_m, i_m)
\]
and 
\[
\vbf = ( t_1^1, \vbf_1^1, \ldots, \vbf_1^{n_1 - 1}, t_1^{n_1}, \vbf_1^{n_1}, t_2^1, \ldots, \vbf_m^{n_m - 1}, t_m^{n_m})
\]
Let 
\[
\mathbf{\hat{u}} = ( t_1^1, \mathbf{0}, \ldots, t_1^{n_1 - 1}, \mathbf{0}, s_1 - (t_1^1 + \ldots, t_1^{n_1 - 1}), \mathbf{0}, t_2^1, \ldots, \mathbf{0}, s_m - (t_m^1 + \ldots, t_m^{n_m - 1}))
\]
Then, by the flow property, we have $\Phi_{\mathbf{\hat{u}}}^{\jbf}(x) = \Phi_{\ubf}^{\ibf}(x)$. Furthermore, $\mathbf{\hat{u}}$ and $\vbf$ have the same length, and $\| \mathbf{\hat{u}}\| = s_1 + \ldots + s_m \leq T$, where we recall that $T = 2(s_1 + \ldots + \ldots s_m)$. Furthermore, 
\[
\| \mathbf{\hat{u}} - \vbf \| = \sum_{l = 1}^{m-1} \sum_{k = 1}^{n_l} \| \vbf_l^k \| + \sum_{k = 1}^{n_m - 1 } \| \vbf_m^k \| + \sum_{l = 1}^m \left( s_k - (t_k^1 + \ldots + t_k^{n_k}) \right).
\]
By construction, for all $l= 1, \ldots, m$, $\|\vbf_l^k\| \leq 2^{-(n_1 + \ldots + n_{l-1}) + k}  h_{\sbf}(\varepsilon)$, where $n_0 = 0$, and $s_k - (t_k^1 + \ldots + t_k^{n_k}) \leq h_{\sbf}(\varepsilon)$,  hence
\[
\| \mathbf{\hat{u}} - \vbf \| \leq \left( \sum_{k=1}^{n_1 + \ldots + n_m - 1} 2^{-k} + m \right)  h_{\sbf}(\varepsilon) \leq (m+1) h_{\sbf}(\varepsilon).
\]
For the same reason, $\|  \vbf \| \leq h_{\sbf}(\varepsilon) + s_1 + \ldots + s_m \leq T$. Thus, Lemma \ref{lem:distancecomposite} and the choice of $h_{\sbf}(\varepsilon)$ implies that 
\[
\| \Phi^{\jbf}_{\vbf}(x) - \Phi^{\ibf}_{\ubf}(x) \| \leq \varepsilon,
\]
which concludes the proof. 
\end{proof}
Proposition \ref{prop:algo} proves that if a point $x^*$ which is attainable from $x$ with a control sequence $(\ibf, \sbf)$ with $\ibf = (i_1, \ldots, i_m)$.  then there exists $s^*$ such that $(x^*, i_m, s^*)$ is $\{F^i, \mu^i, Q\}$ - accessible from $(x, i_1, 0)$. Now, by slightly modifiying the algorithm, we can prove that for all $j_1$ and $i^*$, there exists $s^*$ such that $(x^*, i^*, s^*)$ is $\{F^i, \mu^i, Q\}$ - accessible from $(x, j_1, 0)$. Indeed, if $j_1 \neq i_1$, it suffices to add arbitrary short jumps of the process leading the discrete components from $j_1$ to $i_1$ and then to start the algorithm. In the same way, if $i^* \neq i_m$, it suffices to add arbitrary short jumps to lead the discrete components from $i^*$ to $i_m$. 

\section*{Acknowledgments}
The authors would like to thank Michel Bena\"im for many stimulating discussions about PDMPs and for his consistent support throughout the years. TH gratefully acknowledges support from the Einstein Foundation Berlin through grant IPF-2021-651 (2022--2024) and, until 2020, from the Swiss National Science Foundation through grant $200021-175728/1$. 

\bibliographystyle{alpha}

\bibliography{semimarkov}

\end{document}